\documentclass[11pt]{article}
\usepackage{authblk}
\usepackage[english]{babel}
\usepackage{booktabs}
\usepackage[margin=1in]{geometry}
\usepackage[svgnames]{xcolor}
\usepackage{amsthm}
\usepackage{amssymb}
\usepackage{amsmath}
\usepackage{mathtools}
\usepackage{physics}
\usepackage{amsfonts}
\usepackage{graphicx}
\usepackage{float}
\usepackage{caption}
\usepackage{subcaption}
\usepackage{makecell}
\usepackage{verbatim}
\usepackage{multirow}
\usepackage{dsfont}
\usepackage{algorithm}
\usepackage{algorithmic}
\usepackage{url}
\usepackage{natbib}
\usepackage{bbm}
\usepackage{mathrsfs}
\usepackage[textsize = tiny]{todonotes}
\usepackage{subcaption}
\numberwithin{equation}{section}
\usepackage[colorlinks,pdfdisplaydoctitle,linkcolor=blue,citecolor=blue]{hyperref}
\usepackage{tikz}
\usetikzlibrary{decorations.pathreplacing}
\usetikzlibrary{fadings}
\usetikzlibrary{arrows.meta, positioning, shapes, fit, backgrounds, calc}

\newcommand{\Lip}{\mathsf{Lip}}

\newcommand\R{\mathbb R}

\newcommand\scrS{\mathcal S}

\newcommand\scrX{\mathcal X}
\newcommand\scrY{\mathcal Y}
\newcommand\scrD{\mathcal D}

\newcommand\Ex{\mathbb E}

\newcommand\argmin{\arg\min}

\newcommand\eps{\epsilon}

\newcommand{\GhatSL}{\widehat{\mathcal G}^{\rm SL}}
\newcommand{\GhatML}{\widehat{\mathcal G}^{\rm ML}}

\newtheorem{theorem}{Theorem}[section]
\newtheorem{lem}[theorem]{Lemma}
\newtheorem{rem}[theorem]{Remark}
\newtheorem{prop}[theorem]{Proposition}
\newtheorem{cor}[theorem]{Corollary}
\newtheorem{ass}{Assumption}
\newtheorem{definition}{Definition}
\newtheorem{ex}[theorem]{Example}

\title{
Multilevel Monte Carlo for Zeroth-Order Stochastic Bilevel Optimization with Inexact Lower Level Solves
}

\author{Neil K. Chada\thanks{Department of Mathematics, City University of Hong Kong, Hong Kong SAR, \texttt{neilchada123@gmail.com}}  
\qquad Mohammad S. Salehi\thanks{Independent scholar, London, UK, \texttt{sadeghsalehi1997@gmail.com}} \qquad  Mathias Staudigl\thanks{Department of Business Analytics and Decision Sciences, Vienna University of Economics and Business, Welthandelsplatz 2, 1020 Wien, Austria, \texttt{mathias.staudigl@wu.ac.at}} \\  
Simon Weissmann\thanks{Institute of Mathematics, University of Mannheim, 68138 Mannheim, Germany, \texttt{simon.weissmann@uni-mannheim.de}}}
\date{\today}

\begin{document}
\maketitle

\begin{abstract}
    Learning regularization parameters from data leads to stochastic bilevel optimization problems in which each upper-level evaluation requires a computationally expensive, typically inexact, lower-level solve. We propose a derivative-free method that combines Gaussian smoothing of the upper-level objective with multilevel Monte Carlo estimators built from a hierarchy of lower-level solution accuracies. We establish convergence to approximate stationary points of the smoothed objective and derive end-to-end computational complexity bounds in terms of three solver-dependent quantities: approximation bias, coupled level-difference errors, and computational cost. This framework identifies conditions under which multilevel sampling improves upon single-level Monte Carlo. For linearly convergent deterministic solvers, including gradient descent and FISTA, it removes a logarithmic cost factor. For stochastic solvers, final-iterate SGD and Polyak–Ruppert averaged SGD have the same mean-square approximation error rate but different coupled level-difference decay. Under the consecutive hierarchies considered here, final-iterate SGD yields no improvement in the asymptotic complexity exponent, whereas Polyak–Ruppert averaging yields an exponent improvement. We also extend the analysis to proximal and constrained upper-level problems. Numerical experiments in image denoising and computed tomography illustrate reductions in lower-level computational work while maintaining comparable reconstruction quality.
\end{abstract}

Gaussian smoothing; inverse problems.

\section{Introduction}\label{sec:introduction}
Variational regularization methods are popular in modern inverse problems \cite{scherzer2009variational}. A critical aspect in this daily routine is the selection of hyperparameters. A by-now common strategy to learn parameters of a variational regularization model from data is \emph{bilevel learning}; see, e.g. \cite{KP13,CCDSV16,DSV17,HKB18} and references in \cite{AMOS19}. Bilevel optimization is a dominant research paradigm in mathematical programming, and is known to be a challenging class of optimization problems \cite{bolte2026geometric}. For the practical implementation in inverse problems, it is thus of utmost importance to have reliable learning algorithms which come with certificates concerning the number of data samples required in the supervised learning task, as well as an estimate on the worst-case run time of numerical algorithms employed to solve the variational reconstruction model for a given candidate set of hyperparameters. In this paper, we are concerned with the development of an integrated simulation and learning environment for \emph{simultaneously} recovering model parameters from data, and model-based tuning of hyperparameters, to effectively perform the variational reconstruction. 
\subsection{Mathematical Formulation}
Let \((\Omega,\mathcal F,\mathbb P)\) be a probability space. Let \(\mathcal X\) be a real Hilbert space of model parameters, \(\mathcal D\) be a real Hilbert space of observation data, and let \(\mathcal Y\subseteq \mathbb R^d\) be a set of hyperparameters. 
We consider random variables \(X:\Omega\to\mathcal X\), \(D:\Omega\to\mathcal D\),  where \(X(\omega)\) denotes the ground truth and \(D(\omega)\) denotes the corresponding observed data. We write
\(Z := (X,D):\Omega\to\mathcal X\times\mathcal D\)
for the training pair and denote its law by \(\mathbb P_Z := \mathbb P\circ Z^{-1}\). Unless stated otherwise, all expectations in the sequel are taken with respect to this probability measure and all measurability statements are understood with respect to the Borel $\sigma$-algebras on the respective spaces. We are interested in linear inverse problems, in which an unknown quantity $X$ is observed only through noisy data $D$, generated by the statistical model
\begin{equation}\label{eq:IP}\tag{IP}
    D(\omega) = A X(\omega) + W(\omega) \qquad \omega \in\Omega,
\end{equation}
where \(A:\mathcal X\to\mathcal D\) is a given and deterministic forward operator and \(W:\Omega\to\mathcal D\) is stochastic observational noise. Given a specific set of hyperparameters \(y\in\mathcal Y\) and measurements $d\in\scrD$, the \emph{reconstruction operator} is a mapping $h:\scrY\times\scrD\to\scrX$, defined as 
\begin{equation}\label{eq:LLP}
    h(y,d)\in \argmin_{x\in\mathcal X}\{ F(x,y,d)\equiv \mathcal L(Ax,d)+\mathcal S_y(x)\}.
    \tag{LLP}
\end{equation}
Here, \(\mathcal L:\mathcal D\times\mathcal D\to\mathbb R\) is a data-fidelity term and \(\mathcal S_y:\mathcal X\to\mathbb R\cup\{+\infty\}\) is a parameter-dependent regularizer. 
In the bilevel learning architecture, we refer to it as the \emph{lower-level problem} \eqref{eq:LLP}. Throughout the paper, we assume that the reconstruction operator $h:\mathcal Y\times\mathcal D\to\mathcal X$ is well-defined and measurable. 
Given a loss function \(g:\mathcal X\times\mathcal X\to\mathbb R\), the \emph{upper-level objective} (ULP) is
\begin{equation}\label{eq:ULP}
\Psi(y):=\mathbb E\big[g(h(y,D),X)\big]:=\int_{\mathcal X\times\mathcal D} g(h(y,s),x)\,\mathbb P_Z({\mathrm d} x,{\mathrm d} s). \tag{ULP}
\end{equation}
In this implicit form, the \emph{bilevel learning problem} (BLP) is therefore
\begin{equation}\label{eq:BOP}
    \min_{y\in\mathcal Y}\Psi(y). \tag{BLP}
\end{equation}
While problem \eqref{eq:BOP} is innocent looking, it is in practice usually very hard to solve. First, the implicit function $y\mapsto \Psi(y)$ is typically non-convex and therefore finding a global optimum of the implicit function $\Psi$ is in general computationally intractable, while even the computation of approximate stationary points can be challenging \cite{zhang2020complexity}. Second, each query of the objective function requires an evaluation of the reconstruction operator $h(y,d)$ for $\mathbb{P}_{Z}$-almost every $d\in\scrD$. This typically involves the numerical solution of the lower-level problem \eqref{eq:LLP}, and thus is computationally demanding.

\subsection{Solution strategy}
In order to overcome the bottlenecks described above, \cite{SWL24} followed a derivative-free optimization approach, in the spirit of \cite{grimm2017discrete,LR25}. The core elements of this algorithmic solution strategy are the basis of this paper, and thus will be briefly described in the following. 

Since the reconstruction operator $h(y,d)$ is defined only implicitly, a direct first-order method for the upper-level objective requires derivatives of the solution map with respect to the hyperparameters. To illustrate this, suppose that $\mathcal Y=\mathbb R^d$, $\mathcal X=\mathbb R^n$, and that the lower-level problem is unconstrained and sufficiently smooth. Formally differentiating \(\Psi(y)=\mathbb E[g(h(y,D),X)]\) gives
\[ \nabla\Psi(y) = \mathbb E[ D_y h(y,D)^\top \nabla_1 g(h(y,D),X)]\,. \]
The lower-level optimality condition is
\[
\nabla_x F(h(y,d),y,d)=0.
\]
Differentiating this identity with respect to $y$ yields
\[ \nabla_{xx}^2F(h(y,d),y,d)D_yh(y,d) + D_y[\nabla_xF](h(y,d),y,d)=0\,.
\]
Whenever the lower-level Hessian is invertible,
\[
D_yh(y,d) = - \bigl(\nabla_{xx}^2F(h(y,d),y,d)\bigr)^{-1} D_y[\nabla_xF](h(y,d),y,d)\,,
\]
and hence
\[ \nabla\Psi(y) = -\mathbb E[ D_y[\nabla_xF](h(y,D),y,D)^\top \bigl(\nabla_{xx}^2F(h(y,D),y,D)\bigr)^{-1} \nabla_1g(h(y,D),X) ].
\]
This equation is the formal justification for the development of popular approximation methods using randomly truncated Neumann series \cite{ghadimi2018approximation}. However, these approximation methods require second-order differentiability of the lower level loss function, and sufficient regularity of the reconstruction operator. In particular, it requires smooth regularizers $\mathcal{S}_{y}(x)$, and unconstrained domains for the optimization. These assumptions are a rare commodity in inverse problems. To cope with non-smooth lower level problems and potential constraints, \emph{derivative-free optimization} (DFO) methods, using only zero-order oracles of the implicit function $\Psi(y)$, are an attractive toolbox \cite{LR25}. The zeroth-order oracle used in \cite{SWL24} employs the Gaussian smoothing technique \cite{nesterov2017random,BalGhad22,berahas2022theoretical}. Let \(U\sim\mathcal N(0,I_d)\) be independent of \(Z\), and define 
\[
    \Psi_\eta(y):= \mathbb E_U[\Psi(y+\eta U)] \qquad \eta>0.
\]
Under standard regularity assumptions (cf. Section \ref{subsec:setup}), \(\Psi_\eta\) is differentiable and its
gradient can be represented using only function evaluations of \(\Psi\):
\[
\nabla\Psi_\eta(y) = \mathbb E_U \Big[\frac{\Psi(y+\eta U)-\Psi(y)}{\eta}U\Big] =\mathbb E
    \left[
        \frac{
            g(h(y+\eta U,D),X)-g(h(y,D),X)
        }{\eta}
        U\right].
\]
where the last expectation is over the product probability space supporting the independent variables \(U\) and \(Z=(X,D)\). This identity shows that \emph{Monte Carlo} (MC) sampling techniques can be employed to obtain unbiased random estimators for the gradient $\nabla\Psi_{\eta}(y)$, whenever we can sample the random pair $(X,D)$ efficiently and the evaluation of the reconstruction operator $h(y,D)$ is not too costly. Under these assumptions, \cite{SWL24} develop a variance reduced zeroth-order proximal gradient type method. It is demonstrated to be a robust solution method, in the sense that it provides approximate solutions to the entire bilevel problem \eqref{eq:BOP}, even if the lower-level reconstruction operator is only available via an inexact oracle. This is realistic when this mapping is retrieved from a solver which returns an approximate value of the optimal parameter $h^{\eps}(y,D)\approx h(y,D)$. However, inexact reconstruction operators lead to an additional bias in the gradient estimator, which needs to be controlled carefully by improving the numerical precision of the reconstruction operator over time. In practice, this typically means longer run time of the employed lower-level solver, and thus increasing computational costs and increasing number of function evaluations. A general overview of methods is provided in \cite{audet2026benchmarking}.

\paragraph{Multilevel Monte Carlo gradient estimators.} 
To improve the computational costs, we adopt a \emph{Multilevel Monte Carlo} (MLMC) sampling strategy. While traditionally MLMC methods have been introduced to solve infinite-dimensional integration problems in connection with stochastic differential equations in mathematical finance \cite{MBG08,MBG15}, it became an attractive variance-reduction technique for stochastic optimization problems in finite- and infinite-dimensional settings; see e.g. \cite{Cliffe11,WWZ2021,VBV2019,HuHe21,rowbottom2026multi}. Adopting the MLMC philosophy, we replace the single accuracy level $\widehat{h}=h^{\eps}$ used in a Monte Carlo sampling approach (hitherto \emph{single-level Monte Carlo} (SLMC)), by a hierarchy \(\widehat h_0,\widehat h_1,\dots,\widehat h_L\). The idea behind this schedule of hierarchies is that higher levels represent more expensive computational models, either from an imposed lower error tolerance, more finely grained discretizations, or other sources of approximation errors. The resulting multilevel estimator uses many samples on cheap coarse levels and fewer samples on expensive fine levels. If most of the variance is captured by the coarse grain estimators, we only need to use fewer samples on the fine grained levels, and thereby reduce the overall computational and sample complexity of the algorithm. Using this structure, we want to understand which lower-level solvers indeed allow us to exploit this cost-variance trade-off. Specifically, the central \emph{research question} addressed in this paper is:
\begin{quote}
\emph{Which properties of the lower-level solver (e.g. gradient descent, accelerated gradient descent, etc.) and the induced hierarchy determine whether multilevel sampling improves the complexity of derivative-free methods for stochastic bilevel optimization?}

\end{quote} 

Our analysis identifies the approximation bias, coupled level-difference errors, and computational cost of the lower-level solver as the fundamental quantities governing whether multilevel estimation improves the complexity of derivative-free bilevel optimization. In particular, our results reveal a qualitative distinction between lower-level algorithms: multilevel sampling does not automatically improve complexity, and its effectiveness depends critically on the decay of coupled level differences.
An overview of our strategy is presented in Figure \ref{fig:mlmc-overview}.

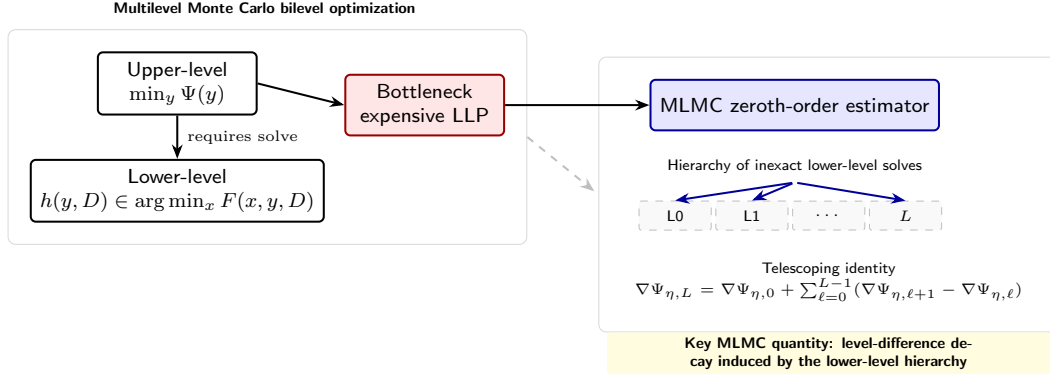
\begin{figure}[htbp]
\centering
\resizebox{0.85\textwidth}{!}{%
\begin{tikzpicture}[
    block/.style={rectangle, draw, thick, rounded corners=2pt, 
                  minimum width=2.2cm, minimum height=0.7cm, align=center, 
                  font=\scriptsize\sffamily},
    bottleneck/.style={block, fill=red!10, draw=red!60!black},
    solution/.style={block, fill=blue!10, draw=blue!60!black},
    arrow/.style={-{Stealth[length=2mm]}, thick},
    level/.style={rectangle, draw=gray!50, dashed, rounded corners=1pt,
                  minimum width=1.0cm, minimum height=0.4cm, align=center,
                  font=\tiny\sffamily, fill=gray!5},
    label/.style={font=\tiny\itshape},
]

\node[block] (upper) {Upper-level\\$\min_y \Psi(y)$};
\node[block, below=0.6cm of upper] (lower) {Lower-level\\$h(y,D) \in \arg\min_x F(x,y,D)$};
\draw[arrow] (upper) -- node[right, label, font=\tiny] {requires solve} (lower);

\node[bottleneck, right=1.2cm of upper, yshift=-0.3cm] (bottleneck) {Bottleneck\\expensive LLP};

\draw[arrow, bend left=15] (upper.east) -- (bottleneck.west);

\node[solution, right=2.0cm of bottleneck] (mlmc) {MLMC zeroth-order estimator};

\draw[arrow, bend left=15] (bottleneck.east) -- node[above, label, font=\tiny] {} (mlmc.west);

\node[below=0.25cm of mlmc, font=\tiny\sffamily] (hlabel) {Hierarchy of inexact lower-level solves};

\node[level, below left=0.25cm and -0.75cm of hlabel] (l0) {L0};
\node[level, right=0.05cm of l0] (l1) {L1};
\node[level, right=0.05cm of l1] (l2) {$\cdots$};
\node[level, right=0.05cm of l2] (lL) {$L$};

\draw[arrow, draw=blue!60!black, shorten <=1pt] (hlabel.south) -- (l0.north);
\draw[arrow, draw=blue!60!black, shorten <=1pt] (hlabel.south) -- (l1.north);
\draw[arrow, draw=blue!60!black, shorten <=1pt] (hlabel.south) -- (lL.north);

\node[below=0.3cm of l2, font=\tiny\sffamily, align=center, text width=5.5cm] (decomp) {
Telescoping identity\\ $\nabla\Psi_{\eta,L} = \nabla\Psi_{\eta,0} + \sum_{\ell=0}^{L-1}(\nabla\Psi_{\eta,\ell+1} - \nabla\Psi_{\eta,\ell})$
};

\node[below=0.3cm of decomp, align=center, font=\tiny\bfseries\sffamily, text width=6cm, fill=yellow!15, rounded corners=1pt, inner sep=2pt] (summary) {
Key MLMC quantity: level-difference decay induced by the lower-level hierarchy

};

\begin{scope}[on background layer]
    \node[draw=gray!30, rounded corners=3pt, fit=(upper)(lower)(bottleneck), 
          inner sep=0.3cm, label={[font=\tiny\sffamily]above:Problem}] (prob-box) {};
    \node[draw=gray!30, rounded corners=3pt, fit=(mlmc)(hlabel)(l0)(lL)(decomp), 
          inner sep=0.3cm, label={[font=\tiny\sffamily]above:Solution}] (sol-box) {};
\end{scope}

\draw[arrow, dashed, draw=gray!50, shorten >=2pt] (prob-box.east) -- (sol-box.west);

\node[above=0.05cm of prob-box, font=\tiny\bfseries\sffamily, align=center] {
Multilevel Monte Carlo bilevel optimization
};

\end{tikzpicture}
} 
\caption{Problem structure and multilevel solution. Upper-level evaluations require lower-level solves (bottleneck). MLMC replaces single-level Monte Carlo with a hierarchy of lower-level accuracies. Effectiveness depends on level-difference decay of the lower-level hierarchy.}
\label{fig:mlmc-overview}
\end{figure}

\paragraph{Main results} 
Adopting the DFO optimization framework, we provide the following answers to our research question:
\begin{enumerate}
    \item We formulate a multilevel derivative-free bilevel optimization method
    for stochastic upper-level objectives whose integrands require approximate
    lower-level solves. Our method combines Gaussian smoothing with a multilevel
    Monte Carlo hierarchy over lower-level solution accuracies; see
    Algorithm~\ref{alg:ml-dfo} and the estimator definitions
    \eqref{eq:sl-gradient-estimator} and \eqref{eq:ml-gradient-estimator}. This induces a
    biased zeroth-order stochastic gradient oracle for the smoothed upper-level
    objective. The mean-square error bounds in Lemma~\ref{lem:gradient-mse}
    explicitly separate Monte Carlo sampling error and 
\begin{align*}
&\text{the lower-level bias } b_{\ell}=\sup_{y\in\scrY}\mathbb E\big[\textcolor{black}{\Lip_0(g,X)^2}\|\widehat h_\ell(y,D)-h(y,D)\|_{\mathcal X}^2\big],\\
&\text{ the level-difference error } v_{\ell+1}=\sup_{y\in\scrY}\mathbb E\big[\textcolor{black}{\Lip_0(g,X)^2}\|\widehat h_{\ell+1}(y,D)-\widehat h_\ell(y,D)\|_{\mathcal X}^2\big], 
\end{align*}
and the finite-difference/smoothing constants, where \textcolor{black}{$\Lip_0(g,x)$ denotes the Lipschitz constant of $h\mapsto g(h,x)$ given $x\in\mathcal X$}.
    \item We prove convergence of the resulting derivative-free method under a
    controlled mean-square error of the estimator. Theorem~\ref{thm:controlled-mse-convergence}
    gives a descent bound for the smoothed objective, while
    Corollary~\ref{cor:smoothed-stationarity} yields the corresponding iteration
    complexity for finding a point with small
    \(\mathbb E\|\nabla\Psi_\eta\|^2\). We also relate smoothed stationarity to
    stationarity of the original objective in the smooth case
    (Corollary~\ref{cor:original-stationarity}) and to Goldstein stationarity in
    the nonsmooth Lipschitz case (Proposition~\ref{prop:goldstein-stationarity}).

    \item We develop a solver-independent MLMC complexity theory based on three
    lower-level quantities: the bias \(b_\ell\), the coupled level-difference
    errors \(v_{\ell+1}\), and the cost \(c_\ell\). The generic estimator
    complexity result, Theorem~\ref{thm:generic-estimator-complexity}, gives the
    single-level and multilevel costs for any prescribed hierarchy. The resulting
    rates are then specialized to linearly convergent lower-level solvers in
    Theorem~\ref{thm:linear-ll-complexity} and to sublinearly convergent solvers
    in Theorem~\ref{thm:polynomial-ll-complexity}. We verify the abstract lower-level assumptions for concrete deterministic
    and stochastic solvers. For linearly convergent deterministic solvers, gradient descent and FISTA satisfy the required bias, level-difference, and
    cost conditions; see Proposition~\ref{prop:gd-verification} and
    Proposition~\ref{prop:fista-verification}. For sublinearly convergent stochastic solvers, the consecutive hierarchy for final-iterate SGD yields no improvement in the asymptotic multilevel complexity exponent (Theorem~\ref{thm:fi-sgd-complexity}), whereas Polyak--Ruppert averaged SGD improves the level-difference decay and yields a multilevel speedup
    (Theorem~\ref{thm:pr-sgd-complexity}). Moreover, we show that the absence of the required exponent improvement for final-iterate SGD is sharp within the consecutive-level MLMC allocation model, see Proposition~\ref{prop:fi-sgd-sharpness}. Finally, we combine the estimator complexity with the outer convergence theory to obtain total work bounds for finding a smoothed stationary point, see Corollary~\ref{cor:solver-stationarity-complexity} and  Table~\ref{tab:stationarity-complexity}.  A simplified version is provided in Table \ref{tab:intro-summary}.
    \item We extend the analysis to proximal and constrained outer-level variants. The localized estimator bounds in Lemma~\ref{lem:localized-mse-explicit-tail} quantify the Gaussian tail price of using smoothing with unbounded support,
    while Theorem~\ref{thm:prox-controlled-mse} and Corollary~\ref{cor:projected-stationarity-explicit-tail} give stationarity guarantees for composite and projected multilevel derivative-free methods.
    \item We illustrate the theory on numerical examples from imaging inverse problems. The experiments demonstrate the practical cost reductions predicted by the theory and include empirical illustration of complexity rates; see Section~\ref{sec:numerics}, in particular Section~\ref{subsubsec:numerical-complexity-rates}.
\end{enumerate}

\begin{table}[h!]
\small
\centering
\begin{tabular}{lccc}
\toprule
Lower-level solver
& Bias
& Coupled level-difference error
& MLMC improvement \\
\midrule
GD / FISTA
& Geometric
& Geometric
& Yes, logarithmic gain \\

Final-iterate SGD (consecutive)
& $\mathcal{O}(\ell^{-1})$
& $\mathcal{O}(\ell^{-2})$
& No exponent gain \\

Polyak--Ruppert SGD (consecutive)
& $\mathcal{O}(\ell^{-1})$
& $\mathcal{O}(\ell^{-3})$
& Yes, exponent gain \\
\bottomrule
\end{tabular}
\caption{Solver-dependent multilevel regimes. The effectiveness of MLMC is
determined not only by the lower-level approximation bias, but also by the
decay of the coupled level differences.}
\label{tab:intro-summary}
\end{table}

Table~\ref{tab:intro-summary} highlights the main solver-dependent phenomenon underlying our analysis. In particular, for the consecutive hierarchies considered below, final-iterate SGD and Polyak--Ruppert averaged SGD have the same lower-level bias rate, but different coupled level-difference rates, leading to qualitatively different multilevel complexity regimes. The resulting estimator and overall optimization complexities, as well as the associated level choices and multilevel sample allocations, are summarized in Table \ref{tab:stationarity-complexity} and \ref{table:comp}.
 

\subsection{Related work} 
Bilevel optimization in inverse problems has been widely used for data-driven regularization and parameter learning; see, for instance, \cite{AMOS19,HKB18,HERTRICH2026} and the
references therein. A prominent class of applications arises in imaging, where bilevel approaches have been developed for denoising, smoothing, deblurring, and
related variational reconstruction tasks \cite{KP13,CCDSV16,DSV17,LLS22}. Further developments connect bilevel learning with data-driven regularization
models and learned reconstruction architectures \cite{AMOS19,CST22,CW14}, as well as with the identification of regularization functions and parameters \cite{FFS18}. Many of these approaches rely on first-order information, either through derivatives of the forward operator, differentiability of the lower-level solution map, or adjoint-based hypergradient computations. In large-scale inverse problems such derivative information may be unavailable,
unstable, or expensive to compute. This has motivated derivative-free bilevel methods \cite{LR25,ER24,SMR25}. 

Our work is also related to multilevel methods in deterministic and stochastic optimization with biased or inexact oracles. A general multilevel principle in optimization is to replace evaluations of an objective
\(\Psi\), or of its gradient, by evaluations of approximations \(\Psi_\ell\), where the level \(\ell\in\mathbb N\) controls both the accuracy and the computational cost. Higher levels are more accurate but more expensive, and the algorithmic question is how to balance optimization error, discretization or solver error, and computational work. Multilevel optimization
algorithms of this type have been studied, for example, for optimal control problems \cite{MR4284423,MR4294188}. In \cite{WWZ2021}, optimal strategies for increasing the level accuracy along the optimization process were derived and applied to optimization and sampling in inverse problems. Related ideas also appear in PDE-constrained optimization under uncertainty, where MLMC estimators are used to reduce the cost of objective and gradient evaluations; see, for
instance, \cite{doi:10.1137/16M109870X} for an error analysis of minimizers in a strongly convex setting. The works \cite{MNT2020,VBV2019} apply MLMC gradient
estimators within iterative stochastic gradient methods, thereby reducing the cost of gradient estimation when a hierarchy of increasingly accurate models is available. More broadly, MLMC and randomized multilevel ideas have been used in stochastic optimization through sample average approximation
\cite{blanchet2015unbiased}, Robbins--Monro and Polyak--Ruppert stochastic approximation \cite{dereich2019general}, stochastic compositional optimization \cite{blanchet2017unbiased}, distributionally robust optimization \cite{levy2020large}, and proximal or Moreau-envelope based stochastic methods \cite{asi2021stochastic}. The closest works to ours are recent MLMC methods for stochastic optimization with biased first-order oracles. The general framework of \cite{HuHe21}
analyzes several MLMC gradient estimators for stochastic optimization with biased oracles and characterizes their complexity in terms of abstract bias, variance, and cost conditions. More recently, \cite{NEURIPS2023_f77d9409} introduced an MLMC estimator within a double-loop stochastic gradient method for contextual stochastic bilevel optimization, where the lower-level problem depends on contextual information and the MLMC construction reduces the cost of biased hypergradient estimation.


 \subsection{Organization of the paper}
Section~\ref{sec:method} introduces the Gaussian smoothing framework, the multilevel gradient
estimator, and the resulting derivative-free bilevel algorithm. It also gives the auxiliary mean-square error bounds used throughout the paper. Section~\ref{sec:convergence} proves convergence under abstract estimator error bounds and derives stationarity complexity results. Section~\ref{sec:solver-complexity} verifies the abstract assumptions for linearly convergent deterministic solvers and sublinearly convergent stochastic solvers, including gradient descent, FISTA, final-iterate SGD, and Polyak--Ruppert averaged SGD. Section~\ref{sec:prox-extensions} describes the extension of our analysis to proximal and constrained variants. Section~\ref{sec:numerics} presents numerical experiments. Omitted
proofs are collected in the appendix.

\section{Multilevel derivative-free gradient estimation}
\label{sec:method}

\subsection{Standing assumptions and notation}
\label{subsec:setup}
Throughout the paper, \(Z=(X,D)\sim \mathbb P_Z\) denotes a training pair, where \(X\in\mathcal X\) is the ground truth and \(D\in\mathcal D\) is the observed data. $\scrX$ is a real Hilbert space with norm $\norm{\cdot}_{\scrX}$, $\scrD$ is a real Hilbert space of observations, and $\scrY\subseteq\R^{d}$ is a set of admissible hyperparameters. In Sections~\ref{sec:method}--\ref{sec:solver-complexity} we work with $\scrY=\R^d$, constraints are treated in Section~\ref{sec:prox-extensions}. For \((y,d)\in\mathcal Y\times\scrD\), the exact lower-level reconstruction is
denoted by \(h(y,d)\). We assume that a hierarchy of approximate lower-level solution maps \(\widehat h_\ell:\mathcal Y\times\mathcal D\to\mathcal X\), \(\ell\ge0\) is available, where larger \(\ell\) reflects a more accurate and therefore more expensive lower-level solve. When the lower-level solver is randomized, $\widehat h_\ell$ is understood as a random mapping, with the auxiliary solver randomness suppressed from the notation. All expectations involving $\widehat h_\ell$ are then also taken with respect to this randomness. For \(Z=(X,D)\), define the corresponding sample losses
\[
    G_\ell(y;Z)
    :=
    g(\widehat h_\ell(y,D),X),\text{ and } G_\infty(y;Z) := g(h(y,D),X)\,.
\]
The level-\(\ell\) and exact upper-level objectives are
\[
    \Psi_\ell(y)
    :=
    \mathbb E[G_\ell(y;Z)],\quad \Psi(y)
    \equiv
    \Psi_\infty(y)
    :=
    \mathbb E[G_\infty(y;Z)]\,.
\]
We impose the following standing assumptions.
\begin{ass}[Measurability and integrability]
\label{ass:meas-integrability}
For every \(y\in\mathcal Y\) and every \(\ell\ge0\), the maps
\(\scrD\ni d\mapsto h(y,d)\), \(\scrD\ni d\mapsto \widehat h_\ell(y,d)\) are measurable, \(\mathbb E\left[|G_\infty(y;Z)|^2\right]<\infty\) and \(\mathbb E\left[|G_\ell(y;Z)|^2\right]<\infty\).
\end{ass}

\begin{ass}[Lipschitz loss]
\label{ass:lipschitz-loss}
There exists a measurable function \(\Lip_{0}(g,\cdot):\mathcal X\to[0,\infty)\) such that
\[|g(h_{1},x)-g(h_{2},x)|\le \Lip_0(g,x)\|h_1-h_2\|_{\mathcal X}\]
and \textcolor{black}{\(\mathbb E[\Lip_0(g,X)^2]=\int_{\Omega}\Lip_{0}(g,X(\omega))^{2}d\mathbb{P}(\omega)<\infty\)} for all \(h_1,h_2,x\in\mathcal X\). If \(\Lip_{0}(g,\cdot)\) is uniformly bounded, we write \(\Lip_{0}(g,x)\le \overline{\Lip_g}\).\footnote{Given a function $g$, we define $\Lip_{0}(g)$ as the Lipschitz modulus.}
\end{ass}

\begin{rem}
\label{rem:squared-loss}
The squared loss \(g(h,x)=\|h-x\|_{\mathcal X}^2\) is not globally
Lipschitz on \(\mathcal X\). It satisfies Assumption~\ref{ass:lipschitz-loss}
on bounded subsets of \(\mathcal X\). In particular, if the reconstructions and
ground truths are almost surely contained in a bounded subset of
\(\mathcal X\), then the squared loss is covered by the assumption.
\end{rem}
\begin{ass}[Lipschitz upper-level objectives]
\label{ass:lipschitz-upper}
There exist constants $\Lip_0(\Psi)$, and $(\Lip_0(\Psi_\ell))_{\ell\in\mathbb{N}}$, where \(\Lip_0(\Psi_\ell)\) are uniformly bounded in \(\ell\),
i.e., \(\sup_{\ell\ge0}\Lip_0(\Psi_\ell)<\infty\), such that
\[
    |\Psi(y_1)-\Psi(y_2)|
    \le
    \Lip_0(\Psi)\|y_1-y_2\| \quad \text{and}\quad |\Psi_\ell(y_1)-\Psi_\ell(y_2)|
    \le
    \Lip_0(\Psi_\ell)\|y_1-y_2\|\,,
\]
for all \(y_1,y_2\in\mathcal Y\) and all \(\ell\ge0\).
\end{ass}
General conditions for the Lipschitz continuity of the implicitly defined function $\Psi_{\ell}$ have been worked out in the literature at various places. A transparent, and not overly demanding, assumption ensuring this property is the combination of Assumption \ref{ass:lipschitz-loss}, together with a uniform Lipschitz assumption on the reconstruction operators $y\mapsto \hat{h}_{\ell}(y,D)$ and $y\mapsto h(y,D)$, respectively. General conditions for ensuring the Lipschitz continuity of the reconstruction operator can be found in \cite{DonRoc}. In the context of inverse problems, a practically relevant set of sufficient conditions for local Lipschitz continuity is given in \cite{LR25}. In particular, under strong convexity in the lower-level variable and suitable joint regularity with respect to the lower- and upper-level variables, the solution map is locally Lipschitz. Alternatively, if the lower-level objective is uniformly strongly convex in the lower-level variable and the mixed derivatives with respect to the lower- and upper-level variables are uniformly bounded, then the solution map is globally Lipschitz. Combined with the compact physical parameterization discussed in Appendix~\ref{sec:parameter-transformation}, such local estimates yield the uniform Lipschitz bounds required below, provided the corresponding Lipschitz moduli are integrable. 
Let \(U\sim\mathcal N(0,I_d)\) be independent of the training pair \(Z=(X,D)\). For a locally Lipschitz function \(\varphi:\mathcal Y\to\mathbb R\), define its \emph{Gaussian smoothing} by
\[ \varphi_\eta(y) := \mathbb E_U[\varphi(y+\eta U)]=\int_{\mathbb R^d}
        \varphi(y+\eta u)\pi(u)\,{\mathrm d}u, \quad \eta>0\,.
\]
where \(\pi\) denotes the standard Gaussian density on \(\mathbb R^d\). 
Under Lipschitz continuity, the smoothed objective has the following useful properties.
\begin{lem}[Properties of Gaussian smoothing \cite{nesterov2017random}]
\label{lem:gaussian-smoothing}
Let \(\varphi:\mathbb R^d\to\mathbb R\) be Lipschitz continuous with Lipschitz
constant \(\Lip_0(\varphi)\). For \(\eta>0\), define
\(\varphi_\eta(y):=\mathbb E[\varphi(y+\eta U)],\ U\sim\mathcal N(0,I_d),\ y\in\R^d\). Then $y\mapsto \varphi_\eta(y)$ satisfies the following properties:
\begin{enumerate}
    \item Lipschitz continuity: \(|\varphi_\eta(y_1)-\varphi_\eta(y_2)|\le \Lip_{0}(\varphi)\|y_1-y_2\|\) for all $y_1,y_2\in\R^d$;
    \item Smoothing error: \(\sup_{y\in\R^d}\ |\varphi_\eta(y)-\varphi(y)|\le\eta \Lip_{0}(\varphi) \mathbb E\|U\|\le\eta \Lip_0(\varphi)\sqrt d\);
    \item Differentiability: \(\nabla \varphi_\eta(y)=\mathbb E\Big[\frac{\varphi(y+\eta U)-\varphi(y)}{\eta}U\Big]\) for all $y\in\R^d$;
    \item Lipschitz smoothness: \(\|\nabla\varphi_\eta(y_1)-\nabla\varphi_\eta(y_2)\|\le\frac{\sqrt d}{\eta}\Lip_{0}(\varphi)\|y_1-y_2\|\) for all $y_1,y_2\in\R^d$. In particular, $\Lip_{1}(\varphi_{\eta})\leq \frac{\sqrt{d}\Lip_{0}(\varphi)}{\eta}$.
\end{enumerate}
\end{lem}
Applying Lemma~\ref{lem:gaussian-smoothing} to \(\Psi\) yields \(\Psi_\eta\in C^{1,1}\) and
\begin{equation}
    \nabla\Psi_\eta(y)=\mathbb E
    \Big[\frac{G_\infty(y+\eta U;Z)-G_\infty(y;Z)}{\eta}U
    \Big],
    \label{eq:exact-smoothed-gradient}
\end{equation}
where the expectation is over the independent pair \((U,Z)\). Similarly, for every level \(\ell\ge0\), define
\(\Psi_{\eta,\ell}(y):=\mathbb E_U[\Psi_\ell(y+\eta U)]\) with
\begin{equation}
    \nabla\Psi_{\eta,\ell}(y)
    =
    \mathbb E
    \Big[
        \frac{
            G_\ell(y+\eta U;Z)-G_\ell(y;Z)
        }{\eta}U
    \Big].
    \label{eq:level-smoothed-gradient}
\end{equation}
The resulting gradients are Lipschitz continuous with constants $\Lip_{1}(\Psi_\eta)\le\frac{\sqrt d}{\eta}\Lip_{0}(\Psi)$ and \\ \(\Lip_{1}(\Psi_{\eta,\ell})\le\frac{\sqrt d}{\eta}\Lip_{0}(\Psi_{\ell})\). 

\subsection{Monte Carlo gradient estimator}
\paragraph{Single-level estimator.} We first define the standard single-level derivative-free gradient estimator. Fix a lower-level accuracy \(\ell\ge0\) and a sample size \(N\in\mathbb N\). Let \((U_i,Z_i)_{i=1}^N\) be independent copies of the pair \((U,Z)\), where
\(U_i\sim\mathcal N(0,I_d)\) and \(Z_i=(X_i,D_i)\sim\mathbb P_Z\), with \(U_i\) independent of \(Z_i\). The single-level estimator at level \(\ell\) is defined by
\begin{equation}\label{eq:sl-gradient-estimator}
\GhatSL_{\eta,\ell,N}(y):=\frac1N\sum_{i=1}^N\frac{G_\ell(y+\eta U_i;Z_i)-G_\ell(y;Z_i)}{\eta}U_i .
\end{equation}
The estimator \(\GhatSL_{\eta,\ell,N}(y)\) is an unbiased estimator of the gradient of the level-\(\ell\) smoothed objective:
\begin{equation}\label{eq:sl-unbiased-level-gradient}
\mathbb E[\GhatSL_{\eta,\ell,N}(y)]=\nabla\Psi_{\eta,\ell}(y)\,.
\end{equation}
It is generally biased as an estimator of the exact smoothed gradient \(\nabla\Psi_\eta(y)\), because \(\widehat h_\ell\) replaces the exact lower-level solution \(h\). Hence, the single-level estimator error decomposes into a sampling error and a lower-level approximation bias:
\begin{equation}\label{eq:sl-error-decomposition}
\GhatSL_{\eta,\ell,N}(y)-\nabla\Psi_\eta(y)=\Big(\GhatSL_{\eta,\ell,N}(y)-\nabla\Psi_{\eta,\ell}(y)\Big)+\Big(\nabla\Psi_{\eta,\ell}(y)-\nabla\Psi_\eta(y)\Big).
\end{equation}
This decomposition will be quantified in
Lemma~\ref{lem:gradient-mse}.


\paragraph{Multilevel estimator.} The single-level estimator uses one fixed lower-level accuracy \(\ell\). In contrast, the multilevel estimator combines a hierarchy of lower-level approximations $\widehat h_0,\widehat h_1,\ldots,\widehat h_L$. The construction is based on the telescoping identity
\[
\nabla\Psi_{\eta,L}(y)=\nabla\Psi_{\eta,0}(y)+\sum_{\ell=0}^{L-1}\Big(\nabla\Psi_{\eta,\ell+1}(y)-\nabla\Psi_{\eta,\ell}(y)\Big).
\]
Instead of estimating \(\nabla\Psi_{\eta,L}(y)\) directly at the finest level, we estimate the level-zero term and each correction separately. Let $\mathbf N=(N_0,\ldots,N_L)$ be a vector of level-dependent sample sizes. For each level
\(\ell=0,\ldots,L\), let \((U_i^{(\ell)},Z_i^{(\ell)})_{i=1}^{N_\ell}\)
be independent copies of \((U,Z)\), and assume that samples are independent across levels. The level-zero estimator is
\begin{equation}\label{eq:ml-level-zero-estimator}
\widehat{\mathcal G}^{(0)}_{\eta,N_0}(y):=\frac1{N_0}\sum_{i=1}^{N_0}\frac{G_0(y+\eta U_i^{(0)};Z_i^{(0)})-G_0(y;Z_i^{(0)})}{\eta}U_i^{(0)}\,.
\end{equation}
For \(\ell=0,\ldots,L-1\), the level \((\ell+1)\) correction is
\begin{equation}\label{eq:ml-correction-estimator}
\begin{aligned}
    \Delta\widehat{\mathcal G}^{(\ell+1)}_{\eta,N_{\ell+1}}(y)
    :=
    \frac1{N_{\ell+1}}
    \sum_{i=1}^{N_{\ell+1}}
    \Bigg[
    &
    \frac{
        G_{\ell+1}(y+\eta U_i^{(\ell+1)};Z_i^{(\ell+1)})
        -
        G_\ell(y+\eta U_i^{(\ell+1)};Z_i^{(\ell+1)})
    }{\eta}
    \\
    &-
    \frac{
        G_{\ell+1}(y;Z_i^{(\ell+1)})
        -
        G_\ell(y;Z_i^{(\ell+1)})
    }{\eta}
    \Bigg]
    U_i^{(\ell+1)} .
\end{aligned}
\end{equation}
The multilevel derivative-free gradient estimator is then
\begin{equation}\label{eq:ml-gradient-estimator}
    \GhatML_{\eta,L,\mathbf N}(y)
    :=
    \widehat{\mathcal G}^{(0)}_{\eta,N_0}(y)
    +
    \sum_{\ell=0}^{L-1}
    \Delta\widehat{\mathcal G}^{(\ell+1)}_{\eta,N_{\ell+1}}(y).
\end{equation}
Note that the same sample \((U_i^{(\ell+1)},Z_i^{(\ell+1)})\) is used for both \(G_{\ell+1}\) and \(G_\ell\) in each correction term. This coupling is essential for the multilevel variance reduction: under the Lipschitz condition on \(g\), the mean squared error of the correction is controlled by the level-difference
error \(\mathbb E\big[\textcolor{black}{\Lip_0(g,X)^2}\|\widehat h_{\ell+1}(y,D)-\widehat h_\ell(y,D)\|_{\mathcal X}^2\big]\). The estimator is unbiased for the finest-level smoothed gradient,
\[\mathbb E[\GhatML_{\eta,L,\mathbf N}(y)]=\nabla\Psi_{\eta,L}(y)\,,\]
and its bias relative to \(\nabla\Psi_\eta(y)\) is therefore determined by the finest-level approximation error. When the lower-level solver is stochastic, all expectations involving $\widehat h_\ell$ also include the auxiliary solver randomness. Within each multilevel correction, the level $\ell$ and level $\ell+1$ approximations are generated from the same underlying stochastic trajectory, while different Monte Carlo samples and different level estimators use independent copies.

\subsection{The ML-DFO algorithm}
\label{subsec:algorithm}

Having defined the construction of the gradient estimator, we can now present our multilevel derivative-free bilevel optimization method. Given a current upper-level iterate \(Y_k\), the method forms a search direction $\hat{V}_{k+1}$ and then performs an updating step $Y_{k+1} = Y_k-\gamma \widehat{V}_{k+1}$ with a well-calibrated step size $\gamma$. 
The search direction is either the result of a single-level Monte-Carlo estimator (Method '\texttt{SLM}'), or a multilevel estimator (Method '\texttt{MLM}'). 
\begin{algorithm}[htb!]
\caption{Multilevel derivative-free bilevel optimization}
\label{alg:ml-dfo}
\begin{algorithmic}[1]
\STATE \textbf{Input:} Initial point \(Y_0\in\mathcal Y\), smoothing parameter \(\eta>0\),
stepsize \(\gamma>0\), number of outer iterations \(K\), maximum level \(L\),
sample sizes \(\mathbf N=(N_0,\ldots,N_L)\), $\text{Method}\in \{\texttt{SLM},\texttt{MLM}\}$.
\FOR{\(k=0,\ldots,K-1\)}
    \STATE Set 
    \begin{equation}\label{eq:V}
\hat{V}_{k+1}=\begin{cases} \GhatSL_{\eta,L,N_L}(Y_{k}) & \text{if Method $=$\texttt{SLM}}, \\
\GhatML_{\eta,L,\mathbf N}(Y_{k}) & \text{if Method $=$\texttt{MLM}}.\end{cases}
\end{equation}
    \STATE Update $Y_{k+1} = 
    Y_k-\gamma \widehat{V}_{k+1}
    $.
\ENDFOR
\RETURN final iterate \(Y_K\) or random iterate \(Y_{\kappa}\) with independent
\(\kappa \sim{\rm Unif}\{0,\ldots,K-1\}\).
\end{algorithmic}
\end{algorithm}
The corresponding single-level method is obtained by replacing
\(\GhatML_{\eta,L,\mathbf N}(Y_k)\) in Algorithm~\ref{alg:ml-dfo} with the single-level estimator \(\GhatSL_{\eta,L,N}(Y_k)\) at a fixed finest level \(L\). This will be used as the baseline in the complexity comparison. 
The estimator definitions allow arbitrary level-dependent sample sizes \(\mathbf N=(N_0,\ldots,N_L)\). The choices used to achieve a prescribed mean-square accuracy depend on the lower-level bias, level-difference decay, and per-level cost. These choices are derived abstractly in Section~\ref{sec:solver-complexity} and then specialized to concrete lower-level solvers.

\subsection{Zeroth-order biased gradient oracle}
\label{subsec:mse-bounds}
We now record the estimator bounds used in the convergence and complexity analysis. In this subsection, and in the convergence and complexity analysis of Sections~\ref{sec:convergence}--\ref{sec:solver-complexity}, we first consider the unconstrained hyperparameter space \(y\in\mathcal Y:= \mathbb R^d\). Accordingly, the
lower-level bias and level-difference errors below are stated uniformly over \(\mathbb R^d\). Projected and proximal outer variants, which allow constraints on \(y\) and non-smooth convex outer-level terms, are discussed in Section~\ref{sec:prox-extensions}.
The key measures which need to be quantified in order to assess the complexity of an instantiation of Algorithm \ref{alg:ml-dfo} are specified in the definition below.
\begin{definition}[Bias terms]
Given a sequence $(\widehat h_{\ell})_{\ell}$, we define the lower-level bias
\begin{equation}\label{eq:lower-level-bias}
    b_\ell
    :=
    \sup_{y\in\mathbb R^d}
    \mathbb E
    [
        \textcolor{black}{\Lip_0(g,X)^2}\|\widehat h_\ell(y,D)-h(y,D)\|_{\mathcal X}^2
    ],
    \qquad \ell\ge0,
\end{equation}
and the level-difference error
\begin{equation}\label{eq:level-difference-error}
    v_{\ell+1}
    :=
    \sup_{y\in\mathbb R^d}
    \mathbb E
    [
        \textcolor{black}{\Lip_0(g,X)^2}\|\widehat h_{\ell+1}(y,D)-\widehat h_\ell(y,D)\|_{\mathcal X}^2
    ],
    \qquad \ell\ge0.
\end{equation}
\end{definition}
The sequence \(b_\ell\) controls the numerical bias introduced by solving the lower-level problem only approximately, while \(v_{\ell+1}\) controls the mean squared errors of the multilevel correction between levels \(\ell\) and \(\ell+1\). 
The following assumption collects the moment bounds needed to control the finite-difference estimators.
\begin{ass}[Finite-difference moment bounds]
\label{ass:fd-moment-bounds}
There exists a constant \(C_{\rm fd}:=C_{\rm fd}(d,\eta)>0\), independent of \(\ell\ge0\), and \(y\in\mathcal Y\) but possibly depending on the dimension $d$ and $\eta$, such that
\[
    \mathbb E
    \bigg[
        \Big\|
            \frac{
                G_\infty(y+\eta U;Z)-G_\infty(y;Z)
            }{\eta}
            U
        \Big\|^2
    \bigg]
    \le C_{\rm fd}
\]
and, for every \(\ell\ge0\),
\[
    \mathbb E
    \bigg[
        \Big\|
            \frac{
                G_\ell(y+\eta U;Z)-G_\ell(y;Z)
            }{\eta}
            U
        \Big\|^2
    \bigg]
    \le C_{\rm fd}
\]
\end{ass}

If $G_\infty$ and $G_\ell$ are both uniformly bounded, then $C_{\rm fd}\propto d/\eta^2$, and in case $G_\infty$ and $G_\ell$ are both globally Lipschitz continuous in $y$ then $C_{\rm fd}$ is independent of $\eta$ and satisfies $C_{\rm fd}\propto d^2$.
Assumption~\ref{ass:fd-moment-bounds} allows us to derive bounds on the MSE for both the single-level and multilevel gradient estimator. The proof is deferred to Appendix~\ref{app:method-proofs}.
\begin{lem}[Gradient estimator MSE]
\label{lem:gradient-mse}
Suppose that Assumptions~\ref{ass:meas-integrability},
\ref{ass:lipschitz-loss}, \ref{ass:lipschitz-upper} and~\ref{ass:fd-moment-bounds} hold. Then there
exists a constant \(C_{d,\eta}>0\), independent of \(N\) and \(\ell\), but depending on the dimension $d$ and $\eta$, such
that for every \(y\in\mathcal Y\),
\begin{equation}\label{eq:sl-gradient-mse}
    \mathbb E
    \bigg[
        \Big\|
            \GhatSL_{\eta,\ell,N}(y)-\nabla\Psi_\eta(y)
        \Big\|^2
    \bigg]
    \le
    C_{d,\eta}
    \left(
        \frac1N+b_\ell
    \right).
\end{equation}
Similarly, there exists a constant \(C_{d,\eta}>0\), independent of \(L\) and \(\mathbf N\) but depending on the dimension $d$ and $\eta$, such that for every \(y\in\mathcal Y\),
\begin{equation}\label{eq:ml-gradient-mse}
    \mathbb E
    \bigg[
        \Big\|
            \GhatML_{\eta,L,\mathbf N}(y)-\nabla\Psi_\eta(y)
        \Big\|^2
    \bigg]
    \le C_{d,\eta}
    \left(
        \frac1{N_0}
        +
        \sum_{\ell=0}^{L-1}
        \frac{v_{\ell+1}}{N_{\ell+1}}
        +
        b_L
    \right).
\end{equation}
\textcolor{black}{In both cases, it holds $C_{d,\eta}\propto d\eta^{-2} + C_{\rm fd}$.}
\end{lem}
Lemma~\ref{lem:gradient-mse} shows that the derivative-free bilevel setting induces a biased stochastic gradient oracle for the smoothed upper-level objective. We would like to point out that, in contrast to abstract biased-oracle models, the oracle error is not assumed a priori. Instead, it is generated by the interaction between Gaussian finite differences and the accuracy of the lower-level reconstruction solver. 

\section{Approximately solving \eqref{eq:BOP} under a controlled estimator error}
\label{sec:convergence}
Being agnostic about the effective method employed to construct the search direction $\widehat V_{k}$, we derive upper bounds on the rate of convergence for solving the implicit program \eqref{eq:BOP}. Throughout this section, the smoothing parameter \(\eta>0\) is fixed, and the solver-specific choices of \(L\), \(\ell\), as well as the sample sizes are treated as fixed parameters here. Recommendations on how to choose these algorithmic parameters, depending on the chosen lower level solver, are derived in Section~\ref{sec:solver-complexity}.
%
By Lemma~\ref{lem:gaussian-smoothing} and Assumption~\ref{ass:lipschitz-upper}, the smoothed objective \(\Psi_\eta\) is continuously differentiable and \(\Lip_1(\Psi_\eta)\)-smooth with $\Lip_1(\Psi_\eta)\le\frac{\sqrt d}{\eta}\Lip_0(\Psi)$. We consider the stochastic iteration
\begin{equation}\label{eq:outer-iteration}
    Y_{k+1}
    =
    Y_k-\gamma \widehat{V}_{k+1},
\end{equation}
where \(\widehat{V}_{k+1}\) is a possibly biased estimator of \(\nabla\Psi_\eta(Y_k)\). The only property needed for the descent analysis is a bound on the
mean-square estimator error $\mathbb E[\|e_{k+1}\|^2]$, where $e_{k+1}:=\hat{V}_{k+1}-\nabla\Psi_{\eta}(Y_{k})$. Such bounds were established in Lemma~\ref{lem:gradient-mse}; the appropriate
choice of levels and sample sizes will be derived in Section~\ref{sec:solver-complexity}.

\subsection{Descent analysis for the smoothed objective}
\label{subsec:descent-analysis}
The following theorem shows that in order to make the averaged squared norm of
\(\nabla\Psi_\eta\) small, one needs to take into account an objective function gap, and develop a control of the accumulated mean-square error of the stochastic gradient estimator.
\begin{theorem}[Convergence under controlled gradient-estimator error]
\label{thm:controlled-mse-convergence}
Assume that \(\Psi_\eta\) is bounded below by \(\Psi_\eta^\star>-\infty\). Let \((Y_k)_{k\ge0}\) be generated by \eqref{eq:outer-iteration} with a constant step-size
$0<\gamma\le \frac{1}{4\Lip_1(\Psi_\eta)}$. Then, for every \(K\ge1\),
\begin{equation} \label{eq:controlled-mse-convergence}
    \frac1K\sum_{k=0}^{K-1}
    \mathbb E
    \left[
        \|\nabla\Psi_\eta(Y_k)\|^2
    \right]
    \le
    \frac{2(\Psi_\eta(Y_0)-\Psi_\eta^\star)}{\gamma K}
    +
    \frac{4}{K}\sum_{k=0}^{K-1}
    \mathbb E
    \left[
        \|e_{k+1}\|^2
    \right].
\end{equation}
\end{theorem}
\begin{proof}

Since $\Psi_\eta$ is $\Lip_1(\Psi_\eta)$-smooth, for every $k$ we have
\[
\Psi_\eta(Y_{k+1}) \le \Psi_\eta(Y_k) + \langle \nabla \Psi_\eta(Y_k), Y_{k+1} - Y_k \rangle + \frac{\Lip_1(\Psi_\eta)}{2}\|Y_{k+1} - Y_k\|^2.
\]
Using the update rule $Y_{k+1} = Y_k - \gamma \widehat{V}_{k+1}$, this becomes
\[
\Psi_\eta(Y_{k+1}) \le \Psi_\eta(Y_k) - \gamma \langle \nabla \Psi_\eta(Y_k), \widehat{V}_{k+1} \rangle + \frac{\Lip_1(\Psi_\eta)\gamma^2}{2}\|\widehat{V}_{k+1}\|^2.
\]
Define the error $e_{k+1} := \widehat{V}_{k+1} - \nabla \Psi_\eta(Y_k)$. Then $\widehat{V}_{k+1} = \nabla \Psi_\eta(Y_k) + e_{k+1}$, and substituting this into the previous inequality yields
\[
\begin{aligned}
\Psi_\eta(Y_{k+1}) &\le \Psi_\eta(Y_k) - \gamma \langle \nabla \Psi_\eta(Y_k), \nabla \Psi_\eta(Y_k) + e_{k+1} \rangle + \frac{\Lip_1(\Psi_\eta)\gamma^2}{2}\|\nabla \Psi_\eta(Y_k) + e_{k+1}\|^2 \\
&= \Psi_\eta(Y_k) - \gamma \|\nabla \Psi_\eta(Y_k)\|^2 - \gamma \langle \nabla \Psi_\eta(Y_k), e_{k+1} \rangle + \frac{\Lip_1(\Psi_\eta)\gamma^2}{2}\|\nabla \Psi_\eta(Y_k) + e_{k+1}\|^2.
\end{aligned}
\]
Using the elementary inequality $\|a + b\|^2 \le 2\|a\|^2 + 2\|b\|^2$, we obtain
\[
\|\nabla \Psi_\eta(Y_k) + e_{k+1}\|^2 \le 2\|\nabla \Psi_\eta(Y_k)\|^2 + 2\|e_{k+1}\|^2.
\]
Therefore,
\[
\begin{aligned}
\Psi_\eta(Y_{k+1}) &\le \Psi_\eta(Y_k) - \gamma \|\nabla \Psi_\eta(Y_k)\|^2 - \gamma \langle \nabla \Psi_\eta(Y_k), e_{k+1} \rangle  + \Lip_1(\Psi_\eta)\gamma^2 \|\nabla \Psi_\eta(Y_k)\|^2 + \Lip_1(\Psi_\eta)\gamma^2 \|e_{k+1}\|^2.
\end{aligned}
\]
We now bound the inner product term. By Young's inequality, for any $\alpha > 0$,
\[
\langle a, b \rangle \le \frac{\alpha}{2}\|a\|^2 + \frac{1}{2\alpha}\|b\|^2.
\]
Taking $\alpha = \frac{1}{2}$ gives
\[
-\gamma \langle \nabla \Psi_\eta(Y_k), e_{k+1} \rangle \le \frac{\gamma}{4}\|\nabla \Psi_\eta(Y_k)\|^2 + \gamma \|e_{k+1}\|^2.
\]
Inserting this bound into the previous display yields
\[
\begin{aligned}
\Psi_\eta(Y_{k+1}) &\le \Psi_\eta(Y_k) - \gamma \|\nabla \Psi_\eta(Y_k)\|^2 + \frac{\gamma}{4}\|\nabla \Psi_\eta(Y_k)\|^2 + \gamma \|e_{k+1}\|^2 \\
&\quad + \Lip_1(\Psi_\eta)\gamma^2 \|\nabla \Psi_\eta(Y_k)\|^2 + \Lip_1(\Psi_\eta)\gamma^2 \|e_{k+1}\|^2.
\end{aligned}
\]
Grouping the terms involving $\|\nabla \Psi_\eta(Y_k)\|^2$ and $\|e_{k+1}\|^2$, we get
\[
\Psi_\eta(Y_{k+1}) \le \Psi_\eta(Y_k) - \gamma\left(\frac{3}{4} - \Lip_1(\Psi_\eta)\gamma\right)\|\nabla \Psi_\eta(Y_k)\|^2 + \gamma(1 + \Lip_1(\Psi_\eta)\gamma)\|e_{k+1}\|^2.
\]
Now choose the stepsize such that $0 < \gamma \le \frac{1}{4\Lip_1(\Psi_\eta)}$. Then $\Lip_1(\Psi_\eta)\gamma \le \frac{1}{4}$, and consequently
\[
\frac{3}{4} - \Lip_1(\Psi_\eta)\gamma \ge \frac{3}{4} - \frac{1}{4} = \frac{1}{2},
\]
while
\[
1 + \Lip_1(\Psi_\eta)\gamma \le 1 + \frac{1}{4} = \frac{5}{4} \le 2.
\]
Hence the inequality simplifies to
\[
\Psi_\eta(Y_{k+1}) \le \Psi_\eta(Y_k) - \frac{\gamma}{2}\|\nabla \Psi_\eta(Y_k)\|^2 + 2\gamma \|e_{k+1}\|^2.
\]
Rearranging terms gives
\[
\frac{\gamma}{2}\|\nabla \Psi_\eta(Y_k)\|^2 \le \Psi_\eta(Y_k) - \Psi_\eta(Y_{k+1}) + 2\gamma \|e_{k+1}\|^2.
\]
Taking expectations on both sides (conditional on $Y_k$), we obtain
\[
\frac{\gamma}{2}\mathbb{E}\|\nabla \Psi_\eta(Y_k)\|^2 \le \mathbb{E}[\Psi_\eta(Y_k)] - \mathbb{E}[\Psi_\eta(Y_{k+1})] + 2\gamma \mathbb{E}\|e_{k+1}\|^2.
\]
Summing this inequality from $k = 0$ to $K-1$, the intermediate terms $\mathbb{E}[\Psi_\eta(Y_1)], \ldots, \mathbb{E}[\Psi_\eta(Y_{K-1})]$ telescope, leaving
\[
\frac{\gamma}{2}\sum_{k=0}^{K-1} \mathbb{E}\|\nabla \Psi_\eta(Y_k)\|^2
\le
\mathbb{E}[\Psi_\eta(Y_0)] - \mathbb{E}[\Psi_\eta(Y_K)] + 2\gamma \sum_{k=0}^{K-1} \mathbb{E}\|e_{k+1}\|^2.
\]
Since $\Psi_\eta$ is bounded below by $\Psi_\eta^*$, we have $\mathbb{E}[\Psi_\eta(Y_K)] \ge \Psi_\eta^*$, so $-\mathbb{E}[\Psi_\eta(Y_K)] \le -\Psi_\eta^*$. Therefore,
\[
\frac{\gamma}{2}\sum_{k=0}^{K-1} \mathbb{E}\|\nabla \Psi_\eta(Y_k)\|^2
\le
\Psi_\eta(Y_0) - \Psi_\eta^* + 2\gamma \sum_{k=0}^{K-1} \mathbb{E}\|e_{k+1}\|^2.
\]
Finally, dividing both sides by $\gamma K / 2$ yields
\[
\frac{1}{K}\sum_{k=0}^{K-1} \mathbb{E}\|\nabla \Psi_\eta(Y_k)\|^2
\le
\frac{2(\Psi_\eta(Y_0) - \Psi_\eta^*)}{\gamma K}
+
\frac{4}{K}\sum_{k=0}^{K-1} \mathbb{E}\|e_{k+1}\|^2,
\]
which is the desired inequality.
\end{proof}

Obtaining an approximate stationary point of the smoothed objective is likely to deliver accurate solution candidates to the original non-smooth problem, only when the smoothing parameter $\eta$ is chosen sufficiently small. To quantify the dependency on the dimension and the smoothing of the total complexity, we finish this section with the following simple Lemma, which we will use repeatedly the rest of this paper. 
\begin{lem}
\label{lem:3.2}
    Assume that $\inf_{y\in\R^{d}}\Psi(y)>-\infty$. Let $\Delta_{0}:=\Psi(Y_{0})-\inf_{y\in\R^{d}}\Psi(y)$. Then, for all $\eta>0$, 
\begin{equation}\label{eq:optigap:smooth}
\Psi_{\eta}(Y_{0})-\inf_{y\in\mathbb{R}^{d}}\Psi_{\eta}(y)\leq \Delta_{0}+2\eta\Lip_{0}(\Psi)\sqrt{d}. 
\end{equation}
\end{lem}
\begin{proof}
    Let $\Delta_{0}:=\Psi(Y_{0})-\inf_{y\in\mathbb{R}^{d}}\Psi(y)$. By Lemma \ref{lem:gaussian-smoothing}, we have 
$$
\Psi_{\eta}(y)\geq -\eta\Lip_{0}(\Psi)\sqrt{d}+\Psi(y),\text{ and }\Psi_{\eta}(y)\leq \Psi(y)+\eta\Lip_{0}(\Psi)\sqrt{d}.
$$
Hence, $$\Psi_{\eta}(y)-\inf_{y\in\mathbb{R}^{d}}\Psi_{\eta}(y)\leq \Psi(y)-\inf_{y\in\mathbb{R}^{d}}\Psi(y)+2\eta\Lip_{0}(\Psi)\sqrt{d}.$$
In particular, eq. \eqref{eq:optigap:smooth} holds.
\end{proof}
Theorem \ref{thm:controlled-mse-convergence} shows that the convergence rate is controlled by two terms: the optimization error and the mean-square error of the gradient
estimator. The following corollary makes this trade-off explicit.

\begin{cor}[Iteration complexity of the smoothed objective]
\label{cor:smoothed-stationarity}
Let the assumptions of Theorem~\ref{thm:controlled-mse-convergence} hold. Let
\(\kappa\) be uniformly distributed on \(\{0,\ldots,K-1\}\), independently
of all algorithmic randomness. Suppose that, for all \(k=0,\ldots,K-1\), $\mathbb E[\|e_{k+1}\|^2]\le\frac{\varepsilon^2}{8}$. Choose $0<\gamma\le \frac{1}{4\Lip_1(\Psi_\eta)}$ and $K\ge4(\Delta_0+2\eta\Lip_0(\Psi)(\sqrt{d})(\gamma\varepsilon^2)^{-1}$,
then $\mathbb E[\|\nabla\Psi_\eta(Y_{\kappa})\|^2]\le\varepsilon^2$. In particular, with $\gamma=\frac{\eta}{4\sqrt{d}\Lip_{0}(\Psi)}$ and $\eta \Lip_0(\Psi)\sqrt{d}\lesssim 1$, the required number of outer iterations to achieve $\mathbb E[\|\nabla\Psi_\eta(Y_{\kappa})\|^2]\le\varepsilon^2$ is
\[
    K
    =
    \mathcal{O}\left(
        \frac{\sqrt d\,\Lip_0(\Psi)}{\eta\,\varepsilon^2}
    \right).
    \label{eq:outer-complexity-smoothed}
\]
\end{cor}
\begin{proof}
Since $\kappa$ is uniform on $\{0,\ldots,K-1\}$, we have
\[
\mathbb E\|\nabla\Psi_\eta(Y_{\kappa})\|^2
=
\frac1K
\sum_{k=0}^{K-1}
\mathbb E\|\nabla\Psi_\eta(Y_k)\|^2.
\]
Applying Theorem ~\ref{thm:controlled-mse-convergence}, using the assumed bound $\mathbb E[\|e_{k+1}\|^2]\le \varepsilon^2/8$ for all $k$, and the function variation bound
$\Psi_{\eta}(Y_{0})-\Psi^{*}_{\eta}\leq \Delta_{0}+2\eta\Lip_{0}(\Psi)\sqrt{d},$
established in Lemma \ref{lem:3.2}, where $\Delta_0=\Psi(Y_0)-\inf_{y\in\mathbb R^d}\Psi(y)$, we obtain
\[
\mathbb E\|\nabla\Psi_\eta(Y_{\kappa})\|^2
\le
\frac{2(\Delta_0 + 2\eta \Lip_0(\Psi)\sqrt{d})}{\gamma K}
+
\frac{\varepsilon^2}{2}.
\]
To ensure the first term is bounded by $\varepsilon^2/2$, it is sufficient to choose
$$K \ge \frac{4(\Delta_0 + 2\eta \Lip_0(\Psi)\sqrt{d})}{\gamma \varepsilon^2}.$$
 With this choice, the first term satisfies
\[
\frac{2(\Delta_0 + 2\eta \Lip_0(\Psi)\sqrt{d})}{\gamma K}
\le
\frac{2(\Delta_0 + 2\eta \Lip_0(\Psi)\sqrt{d})}{\gamma}
\cdot
\frac{\gamma \varepsilon^2}{4(\Delta_0 + 2\eta \Lip_0(\Psi)\sqrt{d})}
=
\frac{\varepsilon^2}{2}.
\]
Combining this with the estimator-error term gives
$
\mathbb E\|\nabla\Psi_\eta(Y_{\kappa})\|^2
\le
\frac{\varepsilon^2}{2} + \frac{\varepsilon^2}{2}
=
\varepsilon^2,
$
which proves the claim. Since $\Lip_{1}(\Psi_{\eta})\leq \frac{\sqrt{d}\Lip_{0}(\Psi)}{\eta}$, we can choose the constant step size $\gamma=\frac{\eta}{4\sqrt{d}\Lip_{0}(\Psi)}$. This choice gives us the number of iterations
\begin{align*}
\frac{8\sqrt{d}\Lip_{0}(\Psi)(\Delta_{0}+2\eta\Lip_{0}(\Psi)\sqrt{d})}{\eta K}\leq\frac{\varepsilon^{2}}{2}
\Leftrightarrow
K&\geq \frac{16\sqrt{d}\Lip_{0}(\Psi)(\Delta_{0}+2\eta\Lip_{0}(\Psi)\sqrt{d})}{\eta\varepsilon^{2}}\\
&=\mathcal{O}\left(\frac{\sqrt{d}\Lip_{0}(\Psi)}{\eta\varepsilon^{2}}\right).
\end{align*}
This completes the proof.
\end{proof}

\subsection{Relation to stationarity of the original objective}
\label{subsec:original-stationarity}
The stationarity result above concerns the Gaussian-smoothed objective
\(\Psi_\eta\). If the original objective \(\Psi\) is differentiable with
Lipschitz gradient, then smoothed stationarity can be transferred to
stationarity of \(\Psi\). If \(\Psi\) is only assumed to be Lipschitz continuous, then
\(\|\nabla\Psi_\eta(y)\|\) should instead be interpreted as a smoothed
stationarity measure.
\begin{cor}[From smoothed to original stationarity]
\label{cor:original-stationarity}
Assume, in addition, that \(\Psi\) is differentiable and \(\Lip_1(\Psi)\)-smooth. 
Then $\|\nabla\Psi_\eta(y)-\nabla\Psi(y)\|\le\Lip_1(\Psi)\eta\sqrt d$ for all $y\in\R^d$. Consequently, if a random iterate \(Y_{\kappa}\) satisfies $\mathbb E\|\nabla\Psi_\eta(Y_{\kappa})\|^2\le\frac{\varepsilon^2}{2}$ and $\eta\le\frac{{\varepsilon}}{2\Lip_1(\Psi)\sqrt d}$,
then $$\mathbb E\|\nabla\Psi(Y_{\kappa})\|^2\le\varepsilon^2.$$
\end{cor}
\begin{proof}
Since \(\Psi\) is \(\Lip_1(\Psi)\)-smooth, $\nabla\Psi_\eta(y)=\mathbb E_U[\nabla\Psi(y+\eta U)]$. Therefore,
\begin{align*}
    \|\nabla\Psi_\eta(y)-\nabla\Psi(y)\|
    =
    \|
        \mathbb E_U[
            \nabla\Psi(y+\eta U)-\nabla\Psi(y)
        ]
    \|
    &\le
    \mathbb E_U
    \|\nabla\Psi(y+\eta U)-\nabla\Psi(y)\|  \\
    &\le
    \Lip_1(\Psi)\eta\,\mathbb E\|U\|
    \le
    \Lip_1(\Psi)\eta\sqrt d.
\end{align*}
Hence,
\[
    \|\nabla\Psi(y)\|^2
    \le
    2\|\nabla\Psi_\eta(y)\|^2
    +
    2\|\nabla\Psi_\eta(y)-\nabla\Psi(y)\|^2
    \le
    2\|\nabla\Psi_\eta(y)\|^2
    +
    2\Lip_1(\Psi)^2\eta^2d.
\]
Taking expectations and using the assumptions gives $\mathbb E\|\nabla\Psi(Y_{\kappa})\|^2\le\varepsilon^2$.
\end{proof}

If \(\Psi\) is only Lipschitz continuous, then \(\nabla\Psi(y)\) may not exist and Corollary~\ref{cor:original-stationarity} is not applicable. In this case, we can perform a complexity analysis using a relaxed stationarity measure, involving the Goldstein subdifferential \cite{DD2019,DG2019,gold:original,gold}.  
We briefly recall the relevant definitions here, to keep the paper self-contained. 

Let $f:\mathbb{R}^{d}\to\bar{\mathbb{R}}=[-\infty,\infty]$ be a locally Lipschitz continuous function. Its generalised directional derivative in the sense of Clarke is defined as
    $$
    f^{\circ}(y;v):=\lim\sup_{y'\to y,t\to 0^{+}}\frac{f(y'+tv)-f(y')}{t}.
    $$
The Clarke subdifferential of $f$ at $y$ is the set 
$$
\partial_{C}f(y):=\{p\in\R^{d}\vert f^{\circ}(y;v)\geq \langle p,v\rangle \quad\forall v\in\R^{d}\}. 
$$
Given $\delta>0$, the Goldstein $\delta$-subdifferential of $f$ at $y\in\R^{d}$ is the set 
$$
\partial^{\delta}_{G}f(y)=\operatorname{Conv}\big(\bigcup_{y'\in\mathbb{B}(y,\delta)}\partial_{C}f(y')\big). 
$$
A random variable $Y\in L^{0}(\Omega,\R^{d})$ is called a $(\eps,\delta)$-Goldstein stationary point if 
$$
\mathbb{E}[\min\{\norm{p}:p\in\partial^{\delta}_{G}f(Y)\}]\leq \eps. 
$$
The Goldstein subdifferential and Gaussian smoothing are closely related to each other. Indeed, we have the following fact:
\begin{lem}[Theorem 3.6 in \cite{LPSY2024}]\label{lem:GS-inclusion}
    Let $f:\mathbb{R}^{d}\to\mathbb{R}$ be a Lipschitz continuous function and $\mathcal{C}\subset\R^{d}$ a nonempty convex compact subset. Let $f_{\eta}:\mathbb{R}^{d}\to\mathbb{R}$ denote its Gaussian smoothing. Then, for all $(\eps,\delta)>0$, we have 
    \begin{equation}\label{eq:GS}
    \nabla f_{\eta}(y)\in\partial^{\delta}_{G}f(y)+\eps\mathbb{B}(0,1) \qquad\forall \eta\in(0,\bar{\eta}],y\in\mathcal{C}
    \end{equation}
where 
\begin{align*}
    &\bar{\eta}:=\min\{1,\delta/\Gamma\},\;\Gamma:=[-d\mathcal{W}_{-1}(\frac{-\nu^{2/d}}{2\pi e})]^{1/2},\; \nu:=\min\{\frac{\eps}{4\Lip_{0}(f)},(2\pi)^{d/2}-0.5\}
\end{align*}
$\mathcal{W}_{-1}$ denotes the negative branch of the Lambert $W$-function. 
\end{lem}

In terms of these definitions, we can establish the following result: 
\begin{prop}\label{prop:goldstein-stationarity}
Let $\eps_{1},\eps_{2}$ be given positive constants. Fix the smoothing parameter $\eta\in(0,\bar{\eta}]$, so that \eqref{eq:GS} is satisfied for the given parameters $(\eps_{1},\eps_{2})$, and choose $\gamma=\frac{\eta}{4\sqrt{d}\Lip_{0}(\Psi)}$. Let the assumptions of Theorem~\ref{thm:controlled-mse-convergence} hold, and suppose that \(Y_k\in\mathcal C\) almost surely for \(k=0,\ldots,K-1\) for $\mathcal C\subset\R^d$ convex and compact. For $K \ge \frac{4(\Delta_0 + 2\eta \Lip_0(\Psi)\sqrt{d})}{\gamma \varepsilon_1^2}$ let
\(\kappa\) be uniformly distributed on \(\{0,\ldots,K-1\}\), independently
of all algorithmic randomness. Suppose that, for all \(k=0,\ldots,K-1\), $\mathbb E[\|\widehat{V}_{k+1}-\nabla\Psi_\eta(Y_k)\|^2]\le\frac{\eps_{1}^{2}}{8}$. Then, the random variable $Y_{\kappa}\in L^{0}(\Omega,\mathbb{R}^{d})$ is a $(2\eps_{1},\eps_2)$-Goldstein stationary point. 
\end{prop}

\begin{proof}

Since $\nabla\Psi_{\eta}(Y_{\kappa})\in\partial^{\delta}_{G}\Psi(Y_{\kappa})+\eps_1\mathbb{B}(0,1)$, it follows 
$$
\min\{\norm{p}:p\in\partial^{\delta}_{G}\Psi(Y_{\kappa})\}\leq \norm{\nabla\Psi_{\eta}(Y_{\kappa})}+\eps_1.
$$
Hence, choosing $\delta=\eps_{2}$ and $\eps=\eps_{1}$, we obtain for $K=\mathcal{O}\left(\frac{\sqrt{d}\Lip_{0}(\Psi)}{\eta\eps_{1}^{2}}\right)$ from Corollary \ref{cor:smoothed-stationarity}
\begin{align*}
\Ex[\min\{\norm{p}:p\in\partial^{\eps_{2}}_{G}\Psi(Y_{\kappa})\}]\leq \eps_{1}+\mathbb{E}[\norm{\nabla\Psi_{\eta}(Y_{\kappa})}]&\leq \eps_{1}+\sqrt{\Ex[\norm{\nabla\Psi_{\eta}(Y_{\kappa})}^{2}]}\leq \eps_{1} +\eps_{1}=2\eps_{1}\,.
\end{align*}
\end{proof}
Proposition~\ref{prop:goldstein-stationarity} shows that the smoothed stationarity guarantee is meaningful even when the original objective is
nonsmooth. In particular, it is sufficient to use the stepsize \(\gamma=\frac{\eta(\eps_1,\eps_2)}{4\sqrt d\,\Lip_0(\Psi)}\),
and number of iterations
\(K=\mathcal{O}(\frac{\sqrt d\,\Lip_0(\Psi)}{\eta(\eps_1,\eps_2)\eps_1^2})\) to achieve a $(2\eps_{1},\eps_2)$-Goldstein stationary point. The parameter \(\eps_1\) controls the norm of the smoothed gradient and hence the Goldstein stationarity residual, while
\(\eps_2\) determines the radius of the neighborhood over which nearby Clarke subgradients are averaged. The smoothing parameter \(\eta(\eps_1,\eps_2)\) must be
small enough relative to this radius. Thus, decreasing \(\eta(\eps_1,\eps_2)\) improves the
connection between Gaussian-smoothed stationarity and nonsmooth stationarity of
\(\Psi\), but it also increases the outer iteration complexity through the smoothness constant \(\Lip_1(\Psi_\eta)=\mathcal{O}(\eta^{-1})\). The condition \eqref{eq:GS} is guaranteed by
Lemma~\ref{lem:GS-inclusion} whenever the iterates lie almost surely in a compact convex set \(\mathcal C\) and \(\eta\le\bar\eta(\eps_1,\eps_2)\) for the corresponding \((\eps_1,\eps_2)\). This can be achieved automatically in the projected setting with a compact feasible set. In the unconstrained case, it requires an a priori bounded-iterates assumption.

\section{Cost accounting and complexity analysis}
\label{sec:solver-complexity}
The estimator and convergence results in Sections~\ref{sec:method} and~\ref{sec:convergence} are formulated in terms of the abstract quantities
\(b_\ell\) and \(v_{\ell+1}\), defined in \eqref{eq:lower-level-bias}--\eqref{eq:level-difference-error}. The purpose of
this section is twofold. First, we derive single-level and multilevel complexity bounds from generic decay assumptions on these quantities. Second, we verify these assumptions for representative lower-level solvers, including gradient descent and FISTA. A key point is that the complexity analysis of the upper level only requires uniform bounds on the lower-level bias and on the coupled level differences. The specific lower-level algorithm enters the theory only through these two estimates and through the cost of one sample at each level.

\subsection{Abstract complexity framework}
\label{subsec:abstract-complexity}

We now derive the cost of achieving a prescribed mean-square accuracy for the single-level and multilevel gradient estimators. The results in this subsection are independent of the particular lower-level solver. The solver enters only through the bias sequence \(b_\ell\), the level-difference sequence \(v_{\ell+1}\), and the cost of one sample at each level. Let \(c_\ell>0\) denote the computational cost of producing one level-\(\ell\) sample. This includes the cost of computing the lower-level approximation \(\widehat h_\ell(y,D)\) at the points required by the finite-difference estimator. We use the convention
\[
    v_0 := 1,
    \qquad
    c_0 := \text{cost of one level-zero sample}.
\]
The target accuracy in this subsection is
\begin{equation}\label{eq:gradient-mse-target}
\mathbb E\Big[\big\|\widehat{V}-\nabla\Psi_\eta(y)\big\|^2\Big]\le\varepsilon^2,
\end{equation}
for a fixed smoothing parameter \(\eta>0\) and $\widehat{V}\in\{\GhatSL(y),\GhatML(y)\}$. By Lemma~\ref{lem:gradient-mse}, achieving this accuracy requires the upper bounding terms inside the parentheses in \eqref{eq:sl-gradient-mse} or \eqref{eq:ml-gradient-mse} to be of order \(\varepsilon_\eta:=C_{d,\eta}^{-1}\varepsilon^2\).
\begin{ass}[Abstract lower-level rates and costs]
\label{ass:abstract-ll-rates}
The nonnegative sequences of biases \((b_\ell)_{\ell\ge0}\) and level-difference errors
\((v_\ell)_{\ell\ge1}\) defined in \eqref{eq:lower-level-bias}--\eqref{eq:level-difference-error} satisfy $b_\ell<\infty$, $v_\ell<\infty$
and  \(b_\ell\to0\) as \(\ell\to\infty\). Moreover, each level-\(\ell\) sample is associated with positive cost \(c_\ell\), \(\ell\ge0\). 
\end{ass}
The cost sequence $(c_\ell)_{\ell\ge0}$ 
enters the complexity analysis given below. Since the lower-level hierarchies considered here are generated by iterative solvers, computing level \(\ell\) also provides level \(\ell-1\) along the same trajectory, so that the cost of a coupled level correction is of order \(c_\ell\). The uniformity of the bias and level difference errors over the parameters $y\in\R^d$ is discussed in Appendix~\ref{sec:parameter-transformation} in more detail. Throughout this section, we use the standard continuous relaxation of the MLMC allocation problem, treating the level sample numbers $N_\ell$ as positive real variables and disregarding integer-rounding effects in the asymptotic cost estimates.

\begin{theorem}[Generic estimator complexity]
\label{thm:generic-estimator-complexity}
Suppose that Assumptions~\ref{ass:meas-integrability}, \ref{ass:lipschitz-loss}, \ref{ass:lipschitz-upper}, \ref{ass:fd-moment-bounds} and \ref{ass:abstract-ll-rates} hold. Fix \(\eta>0\) and \(\varepsilon>0\). Define \(\varepsilon_\eta:=C_{d,\eta}^{-1}\varepsilon^2\), where $C_{d,\eta}$ is the constant defined in Lemma \ref{lem:gradient-mse}.
\begin{enumerate}
    \item[\rm(i)] Let \(L\) be chosen such that $b_L \lesssim \varepsilon_\eta$. The single-level estimator with $N\propto\varepsilon_\eta^{-1}$ satisfies $\mathbb E[\|\GhatSL_{\eta,L,N}(y)-\nabla\Psi_\eta(y)\|^2]\le\varepsilon^2$.
    Its cost is
    \begin{equation}\label{eq:generic-sl-cost}
    \operatorname{Cost}_{\rm SL}:=N\,c_L\lesssim\varepsilon_\eta^{-1}c_L.
    \end{equation}
    \item[\rm(ii)]
    Let \(L\) again satisfy $b_L \lesssim \varepsilon_\eta$. Define $S_L:=\sum_{\ell=0}^{L}\sqrt{v_\ell c_\ell},\ v_0:=1$.
    Choose the level sample sizes according to
    \begin{equation}\label{eq:generic-ml-sample-allocation}
    N_\ell\propto  \varepsilon_\eta^{-1}S_L\sqrt{\frac{v_\ell}{c_\ell}},\qquad \ell=0,\ldots,L\,.
    \end{equation}
    Then $\mathbb E[\|\GhatML_{\eta,L,\mathbf N}(y)-\nabla\Psi_\eta(y)\|^2]\le\varepsilon^2$ and the corresponding cost satisfies
    \begin{equation}\label{eq:generic-ml-cost-with-ceiling} \operatorname{Cost}_{\rm ML}:=\sum_{\ell=0}^{L}N_\ell \,c_\ell\lesssim\varepsilon_\eta^{-1}\Big(\sum_{\ell=0}^{L}\sqrt{v_\ell c_\ell}\Big)^2.
    \end{equation}
\end{enumerate}
\end{theorem}

\begin{proof}
For the single-level estimator, choosing \(b_L\lesssim\varepsilon_\eta\) and
\(N\propto \varepsilon_\eta^{-1}\) makes the right-hand side of \eqref{eq:sl-gradient-mse} in Lemma~\ref{lem:gradient-mse} of order \(\varepsilon^2\). The cost is \(N c_L\), which gives
\eqref{eq:generic-sl-cost}.

For the multilevel estimator, Lemma~\ref{lem:gradient-mse} gives \eqref{eq:ml-gradient-mse},
\[\mathbb E\Big[\big\|\GhatML_{\eta,L,\mathbf N}(y)-\nabla\Psi_\eta(y)\big\|^2\Big]\le C_{d,\eta}\Big(\frac1{N_0}+\sum_{\ell=0}^{L-1}\frac{v_{\ell+1}}{N_{\ell+1}}+b_L\Big).\]
Using the normalization \(v_0=1\), the level-difference error term can be written as $\sum_{\ell=0}^{L}\frac{v_\ell}{N_\ell}$. The choice \eqref{eq:generic-ml-sample-allocation} gives, up to constants,
$\frac{v_\ell}{N_\ell}\lesssim\varepsilon_\eta\frac{\sqrt{v_\ell c_\ell}}{S_L}$. Summing over \(\ell\) yields $\sum_{\ell=0}^{L}\frac{v_\ell}{N_\ell}\lesssim\varepsilon_\eta$. Together with \(b_L\lesssim\varepsilon_\eta\), this proves the MSE bound. It remains to estimate the cost. Inserting \eqref{eq:generic-ml-sample-allocation}, we obtain
\[\sum_{\ell=0}^{L}N_\ell c_\ell\lesssim\varepsilon_\eta^{-1}S_L\sum_{\ell=0}^{L}c_\ell\sqrt{\frac{v_\ell}{c_\ell}}=\varepsilon_\eta^{-1}S_L\sum_{\ell=0}^{L}\sqrt{v_\ell c_\ell}=\varepsilon_\eta^{-1}S_L^2.
\]
\end{proof}

The sample allocation in the above theorem is obtained from the standard continuous relaxation of the MLMC allocation problem, in which the sample sizes $N_\ell$ are allowed to take arbitrary positive real values. While this convention is sufficient for the complexity analysis, simple deterministic rounding need not preserve the resulting asymptotic cost when the optimal allocation satisfies $N_\ell<1$ on sufficiently fine levels. The continuous allocation can instead be realized in expectation by randomizing the number of samples, for instance by drawing $K_\ell \sim \operatorname{Poisson}(N_\ell)$ and normalizing the corresponding random level sum by the intensity $N_\ell$. This avoids deterministic rounding while preserving the prescribed expected sample allocation and the corresponding expected-work complexity.


\subsection{Linearly convergent lower-level solvers}
\label{subsec:linear-solvers}
We first consider lower-level solvers with a linear convergence rate. This covers deterministic algorithms such as gradient descent or accelerated proximal-gradient methods under strong convexity and smoothness assumptions. In this setting, both the lower-level bias and the coupled level differences decay geometrically with the level.
\begin{ass}[Linear lower-level convergence]
\label{ass:linear-ll-convergence}
There exist constants \(C_b,C_v,C_c>0\) and a contraction factor
\(\rho\in(0,1)\) such that, for all \(\ell\ge0\),
\(b_\ell\le C_b\rho^\ell\) and \(v_{\ell+1}\le C_v\rho^\ell\).
Moreover, the cost of one level-\(\ell\) sample satisfies
\(c_\ell\le C_c(\ell+1)\).
\end{ass}

The cost model \(c_\ell\lesssim \ell+1\) reflects the fact that a level-\(\ell\)
sample requires \(\ell\) lower-level iterations, up to constants. The bias and finite-difference rates are stated in
terms of an abstract geometric decay, and will be verified below for concrete lower-level algorithms.
\begin{theorem}[Complexity under linear lower-level convergence]
\label{thm:linear-ll-complexity}
Suppose that Assumptions~\ref{ass:meas-integrability}, \ref{ass:lipschitz-loss}, \ref{ass:lipschitz-upper}, \ref{ass:fd-moment-bounds}, and~\ref{ass:linear-ll-convergence} hold. Fix \(\eta>0\) and
\(\varepsilon>0\). Then the single-level estimator can be chosen such that $\mathbb E[\|\GhatSL_{\eta,L,N}(y)-\nabla\Psi_\eta(y)\|^2]\le\varepsilon^2$ with cost
\begin{equation}\label{eq:linear-sl-cost}
\operatorname{Cost}_{\rm SL}\lesssim\varepsilon_\eta^{-1}\log(\varepsilon_\eta^{-1}).
\end{equation}
Similarly, the multilevel estimator can be chosen such that $\mathbb E[\|\GhatML_{\eta,L,\mathbf N}(y)-\nabla\Psi_\eta(y)\|^2]\le\varepsilon^2$ with cost
\begin{equation}\label{eq:linear-ml-cost}
\operatorname{Cost}_{\rm ML}\lesssim\varepsilon_\eta^{-1}.
\end{equation}
\end{theorem}
\begin{proof}
We begin with the single-level estimator. Let \(L\) be chosen such that $b_L\le C_b\rho^L\lesssim\varepsilon_\eta$. Equivalently, it is sufficient to choose \(L\propto\log\!\left(\varepsilon_\eta^{-1}\right)\). For the single-level estimator, Theorem~\ref{thm:generic-estimator-complexity} gives \(N\propto\varepsilon_\eta^{-1}\)
and \(\operatorname{Cost}_{\rm SL}= N c_L\). Since \(c_L\lesssim L+1\), we obtain
\[
\operatorname{Cost}_{\rm SL}\lesssim\varepsilon_\eta^{-1}\log(\varepsilon_\eta^{-1})\,.
\]

For the multilevel estimator, consider \eqref{eq:generic-ml-cost-with-ceiling} from Theorem~\ref{thm:generic-estimator-complexity}. For \(\ell\ge1\), by Assumption~\ref{ass:linear-ll-convergence},
\(v_\ell\lesssim\rho^{\ell-1},\ c_\ell\lesssim \ell+1\) and hence,
\(
    \sqrt{v_\ell c_\ell}\lesssim\sqrt{\rho^{\ell-1}(\ell+1)}\,.
\)
Since \(\rho\in(0,1)\), the series
\(
    \sum_{\ell=1}^{\infty}
    \sqrt{\rho^{\ell-1}(\ell+1)}
\)
converges and therefore, $\sum_{\ell=0}^{L}\sqrt{v_\ell c_\ell}\le C$ with a constant $C$ independent of \(L\), \(\eta\), and \(\varepsilon\). Thus,
\[ \operatorname{Cost}_{\rm ML} \lesssim\varepsilon_\eta^{-1}\big(\sum_{\ell=0}^{L}\sqrt{v_\ell c_\ell}\big)^2\lesssim \varepsilon_\eta^{-1}\,.
\]
\end{proof}

\paragraph{Sample allocation.}
Under Assumption~\ref{ass:linear-ll-convergence}, the optimal multilevel allocation from \eqref{eq:generic-ml-sample-allocation} satisfies, up to
constant factors,
\begin{equation}\label{eq:samples_ML_linear}
    N_0\propto\varepsilon_\eta^{-1},\quad N_\ell\propto \varepsilon_\eta^{-1}\sqrt{\frac{v_\ell}{c_\ell}}\lesssim\varepsilon_\eta^{-1}\sqrt{\frac{\rho^{\ell-1}}{\ell+1}},\quad \ell\ge1\,.
\end{equation}
Thus, many samples are taken on the cheap coarse levels, and exponentially fewer samples are taken on fine levels.

\subsubsection{Gradient descent}
\label{subsubsec:gd}
We first verify Assumption~\ref{ass:linear-ll-convergence} for gradient descent applied to \eqref{eq:LLP}. For fixed \(y\in\mathbb R^d\) and \(D\in\mathcal D\), write the lower-level objective as
\[F_{y,D}(x):=\mathcal L(Ax,D)+\mathcal S_y(x)\,.\]
In this subsection, we assume that \(x\mapsto F_{y,D}(x)\) is differentiable, \(\mu\)-strongly convex, and \(L_x\)-smooth in \(x\), with constants \(\mu>0\) and \(L_x<\infty\) that are uniform in \(y\) and \(D\). Let \(h(y,D)\) denote its unique minimizer. Starting from an initial point \(x_0\in\mathcal X\), gradient descent generates
\begin{equation}
    x_{s+1} = x_s-\alpha\nabla_x F_{y,D}(x_s), \qquad s\ge0,
\end{equation}
with step-size \(\alpha\in(0,2/L_x)\). We define the level-\(\ell\) lower-level approximation by the $\ell$-th iterate of the iterative scheme:
\[
    \widehat h_\ell(y,D):=x_\ell.
\]
We assume that one gradient step has constant cost, such that the cost of one level-\(\ell\) sample satisfies $c_\ell=\mathcal{O}(\ell+1)$.
Under the above conditions, Assumption~\ref{ass:linear-ll-convergence} holds.
\begin{prop}[GD rates]
\label{prop:gd-verification}
Assume that, uniformly in \(y\in\mathbb R^d\) and \(D\in\mathcal D\), the lower-level objective \(F_{y,D}\) is \(\mu\)-strongly convex and \(L_x\)-smooth in \(x\). Let gradient descent be run with a fixed step-size \(\alpha\in(0,2/L_x)\), and suppose that the initial error is uniformly bounded:
\(
    \sup_{y\in\mathbb R^d}\ \mathbb E[\textcolor{black}{\Lip_0(g,X)^2}\|x_0-h(y,D)\|_{\mathcal X}^2]<\infty
\).
Then there exist constants \(C_b,C_v>0\) and \(\rho\in(0,1)\) such that $b_\ell\le C_b\rho^\ell$ and  $v_{\ell+1} \le C_v\rho^\ell$ for all $\ell\ge0$.
\end{prop}
\begin{proof}
Under the listed assumptions, the mapping $\operatorname{Id}_{\scrX}(\cdot)-\alpha\nabla_{x}F_{y,D}(\cdot)$ is a contraction \cite{B2008}.
Hence, there exists \(q\in(0,1)\), depending only on \(\mu\), \(L_x\), and \(\alpha\), such that
\[ \|\widehat h_{\ell}(y,D)-h(y,D)\|_{\mathcal X}\le q^\ell \|x_0-h(y,D)\|_{\mathcal X}.
\]
Squaring and taking expectations gives 
\[ b_\ell = \sup_{y\in\mathbb R^d}\  \mathbb E[\textcolor{black}{\Lip_0(g,X)^2}\|\widehat h_\ell(y,D)-h(y,D)\|_{\mathcal X}^2] \le q^{2\ell} \sup_{y\in\mathbb R^d}\ \mathbb E[\textcolor{black}{\Lip_0(g,X)^2}\|x_0-h(y,D)\|_{\mathcal X}^2]\,.\]
Thus, \(b_\ell\le C_b\rho^\ell\) holds with \(\rho=q^2\). For the level differences, we use
\[ \widehat h_{\ell+1}(y,D)-\widehat h_\ell(y,D) = (\widehat h_{\ell+1}(y,D)-h(y,D))+(h(y,D)-\widehat h_\ell(y,D)),
\]
to deduce
\[
    \|\widehat h_{\ell+1}(y,D)-\widehat h_\ell(y,D)\|_{\mathcal X}^2 \le 2\|\widehat h_{\ell+1}(y,D)-h(y,D)\|_{\mathcal X}^2 + 2\|\widehat h_\ell(y,D)-h(y,D)\|_{\mathcal X}^2\,.
\]
Taking expectations and using the bias estimate yields
\[
    v_{\ell+1} \le 2b_{\ell+1}+2b_\ell \le C_v\rho^\ell\,.
\]
\end{proof}

\subsubsection{FISTA}
\label{subsubsec:fista}

Let $\mathcal{X}$ be a finite dimensional euclidean vector space. Given a hyperparameter $y\in\R^{d}$ and observation $D\in\mathcal{D}$, suppose that the lower-level objective is of the form
\[
    F_{y,D}(x):=f_{y,D}(x)+\scrS_y(x),\quad x\in\mathcal X. 
\]
\begin{ass}[Convex composite LLP]\label{ass:FISTA}
We assume that the following conditions hold uniformly in $y\in\R^d$ and $D\in\mathcal D$: 
\begin{itemize}
    \item $\scrS_{y}:\mathcal{X}\to(-\infty,\infty]$ is closed convex and proper. 
    \item $f_{y,D}:\mathcal{X}\to\R$ is $\beta$-Lipschitz smooth (i.e. $\Lip_{1}(f_{y,D})=\beta$), and $\sigma$-strongly convex. 
    \item For all pairs $(y,D)$, there exists a unique solution $h(y,D)=\argmin F_{y,D}(x)$. 
\end{itemize}
\end{ass}
Under these assumptions we can solve this problem using the smooth and strongly convex adaptation of the standard FISTA algorithm \cite{BeckFOM}. Let $x_{0}$ be a given initialization. Set $z_{0}=x_{0}$. Let $\kappa:=\beta/\sigma$ be the condition number of $f_{y,D}$. For $s=0,1,\ldots,$ execute the following steps:
\begin{equation}\label{eq:FISTA:iterates}
\begin{split}
    x_{s+1}=\operatorname{prox}_{\frac{1}{\beta}\scrS_y}(z_{s}-\frac{1}{\beta}\nabla f_{y,D}(z_{s})),\\ 
    z_{s+1}=x_{s+1}+(\frac{\sqrt{\kappa}-1}{\sqrt{\kappa}+1})(x_{s+1}-x_{s})
\end{split}
\end{equation}
The following linear rate of convergence of the objective function values can be proven \cite{BeckFOM}: 
\begin{equation}
    F_{y,D}(x_{s})-F_{y,D}(h(y,D))\leq\left(1-\frac{1}{\sqrt{\kappa}}\right)^{s}\left[F_{y,D}(x_{0})-F_{y,D}(h(y,D))+\frac{\sigma}{2}\norm{x_{0}-h(y,D)}^{2}\right]
\end{equation}
Using the strong convexity error bound 
$$
F_{y,D}(x_{s})-F_{y,D}(h(y,D))\geq \frac{\sigma}{2}\norm{x_{s}-h(y,D)}^{2}
$$
the linear rate on the iterates follows. Set $\hat{h}_{\ell}(y,D)=x_{\ell}$. The linear rate established above gives 
\begin{align}
    \norm{\hat{h}_{\ell}(y,D)-h(y,D)}^{2}_{\mathcal{X}}&\leq (1-\frac{1}{\sqrt{\kappa}})^{\ell}\left[\frac{2}{\sigma}(F_{y,D}(x_{0})-F_{y,D}(h(y,D)))+\norm{x_{0}-h(y,D)}^{2}_{\mathcal{X}}\right] \notag \\
&\leq \rho^{\ell}C(x_{0},y,D) \label{eq:fista-iterate-rate}
\end{align}
where $\rho:=1-\frac{1}{\sqrt{\kappa}}$. 
We assume that one proximal-gradient step has constant cost such that the cost of one level-$\ell$ sample satisfies $c_\ell = \mathcal \mathcal{O}(\ell+1)$. Under the above conditions, Assumption~\ref{ass:linear-ll-convergence} holds.
\begin{prop}[Verification for FISTA]
\label{prop:fista-verification}
Suppose that Assumption~\ref{ass:FISTA} is in place. Let FISTA be run with \eqref{eq:FISTA:iterates} and $x_0\in\mathcal X$, and assume that $C_b(x_0):=\sup_{y\in\mathbb R^d}\mathbb E[\textcolor{black}{\Lip_0(g,X)^2}C(x_0,y,D)]<\infty$. Then with $\rho=1-\kappa^{-1/2}$ it holds that $b_\ell \le C_b(x_0) \rho^\ell$ and $v_{\ell+1}\le C_v\rho^\ell$ for all $\ell\ge0$ where $C_v=2C_b(x_0)(1+\rho)$. 
\end{prop}
\begin{proof}
The bound $b_\ell$ follows from
\eqref{eq:fista-iterate-rate} after taking expectations and then the supremum
over \(y\in\mathbb R^d\). For the level differences, we again use the triangle inequality in squared form:
\[
    \|x_{\ell+1}-x_\ell\|_{\mathcal X}^2
    \le
    2\|x_{\ell+1}-h(y,D)\|_{\mathcal X}^2
    +
    2\|x_\ell-h(y,D)\|_{\mathcal X}^2.
\]
Multiply by $\textcolor{black}{\Lip_0(g,X)^2}$, taking expectations and using $b_\ell \le C_b(x_0)\rho^\ell$ gives
\[
    v_{\ell+1}
    \le
    2b_{\ell+1}+2b_\ell
    \le
    2C_b(x_0)(\rho^{\ell+1}+\rho^\ell)
    =
    2C_b(x_0)(1+\rho)\rho^\ell.
\]
\end{proof}

\subsection{Sublinearly convergent stochastic lower-level solvers}
\label{subsec:sublinear-solvers}
We now consider lower-level solvers whose convergence is only
polynomial. This setting covers stochastic gradient methods under standard strong convexity and certain noise conditions such as bounded-variance or the more general ABC-condition \cite{B2018,DMS23}. In contrast to the linearly convergent case, the multilevel complexity depends sensitively on the decay rate of the coupled level differences. For the consecutive solver hierarchies considered below, this is the regime in which the distinction between final-iterate SGD and Polyak--Ruppert averaged SGD becomes essential.

\begin{ass}[Polynomial lower-level convergence]
\label{ass:polynomial-ll-convergence}
There exist constants \(C_b,C_v,C_c>0\), exponents \(r>0\), \(p>0\), and a cost exponent \(q\ge0\) such that, for all \(\ell\ge0\), \(b_\ell\le C_b(\ell+1)^{-r}\), and \(v_{\ell+1}\le C_v(\ell+1)^{-p}\).
Moreover, the cost of one level-\(\ell\) sample satisfies \(c_\ell\le C_c(\ell+1)^q.\)
\end{ass}
The exponent \(r\) controls the lower-level bias, while \(p\) controls the variance decay of the multilevel corrections. The cost exponent \(q\) describes the growth of the cost of one level sample. For standard iterative solvers where level \(\ell\) corresponds to \(\ell\) lower-level iterations, the natural choice is \(q=1\).

\begin{theorem}[Complexity under polynomial lower-level convergence]
\label{thm:polynomial-ll-complexity}
Suppose that Assumptions~\ref{ass:meas-integrability}, \ref{ass:lipschitz-loss}, \ref{ass:lipschitz-upper}, \ref{ass:fd-moment-bounds}, and~\ref{ass:polynomial-ll-convergence} hold. Fix \(\eta>0\) and
\(\varepsilon>0\), and set \(\varepsilon_\eta := C_{d,\eta}^{-1}\varepsilon^2\). Then the single-level estimator can be chosen such that
\(\mathbb E[\|\GhatSL_{\eta,L,N}(y)-\nabla\Psi_\eta(y)\|^2]\le\varepsilon^2\) with cost
\begin{equation}\label{eq:polynomial-sl-cost}
    \operatorname{Cost}_{\rm SL} = \mathcal{O}\big(\varepsilon_\eta^{-1-q/r}\big)\,.
\end{equation}
For the multilevel estimator, choose \(L+1 \propto \varepsilon_\eta^{-1/r}\). Then the multilevel estimator can be chosen such that
\(\mathbb E[\|\GhatML_{\eta,L,\mathbf N}(y)-\nabla\Psi_\eta(y)\|^2]\le\varepsilon^2\)
with cost
\begin{equation}\label{eq:polynomial-ml-cost-general}
\operatorname{Cost}_{\rm ML}= \mathcal{O}\Big(\varepsilon_\eta^{-1}\Big(\sum_{\ell=1}^{L+1}\ell^{(q-p)/2}\Big)^2\Big)\,.
\end{equation}
Consequently,
\begin{equation}\label{eq:polynomial-ml-cost-regimes}
    \operatorname{Cost}_{\rm ML} =
    \begin{cases}
    \mathcal{O}(\varepsilon_\eta^{-1}),
    & p>q+2,\\[0.4em]
    \mathcal{O}(\varepsilon_\eta^{-1}\log^2(\varepsilon_\eta^{-1})),
    & p=q+2,\\[0.4em]
    \mathcal{O}\!\left(\varepsilon_\eta^{-1-(q+2-p)/r}\right),
    & p<q+2.
    \end{cases}
\end{equation}
\end{theorem}

\begin{proof}
The single-level estimate follows directly from
Theorem~\ref{thm:generic-estimator-complexity}. The bias condition \(b_L\lesssim\varepsilon_\eta\) is satisfied by choosing
\(L+1\propto\varepsilon_\eta^{-1/r}\). With \(N\propto\varepsilon_\eta^{-1}\) and \(c_L\lesssim(L+1)^q\), the single-level cost is
\[
    \operatorname{Cost}_{\rm SL}= N c_L \lesssim
    \varepsilon_\eta^{-1}(\varepsilon_\eta^{-1/r})^q=
    \varepsilon_\eta^{-1-q/r}.
\]
For the multilevel estimator, Theorem~\ref{thm:generic-estimator-complexity} gives
\[\operatorname{Cost}_{\rm ML} \lesssim \varepsilon_\eta^{-1}\Big(\sum_{\ell=0}^{L}\sqrt{v_\ell c_\ell}\Big)^2\,. 
\]
Using \(v_\ell\lesssim \ell^{-p}\) for \(\ell\ge1\) and
\(c_\ell\lesssim \ell^q\), we obtain
\(\sqrt{v_\ell c_\ell}\lesssim\ell^{(q-p)/2},\ \ell\ge1\).
Hence,
\[
    \operatorname{Cost}_{\rm ML}
    =
    \mathcal{O}\Big(
        \varepsilon_\eta^{-1}
        \Big(
            \sum_{\ell=1}^{L+1}
            \ell^{(q-p)/2}
        \Big)^2
    \Big)\,.
\]
This proves \eqref{eq:polynomial-ml-cost-general}. It remains to estimate the sum $\sum_{\ell=1}^{L+1}\ell^{(q-p)/2}$. If \(p>q+2\), then \((q-p)/2<-1\), and the sum is uniformly bounded in \(L\).
If \(p=q+2\), then the sum is of order \(\log(L+1)\). If \(p<q+2\), then
\[
    \sum_{\ell=1}^{L+1}\ell^{(q-p)/2}
    =
    \mathcal{O}\!\left(
        (L+1)^{(q+2-p)/2}
    \right).
\]
Using \(L+1\propto\varepsilon_\eta^{-1/r}\) gives the three regimes in \eqref{eq:polynomial-ml-cost-regimes}.
\end{proof}

This theorem shows that the bias rate \(r\) alone does not determine whether MLMC improves the complexity. The decisive
quantity is the comparison between the level-difference exponent \(p\) and the cost growth exponent \(q\). In the natural iterative cost model \(q=1\), a multilevel improvement requires \(p>2\), the logarithmic borderline is \(p=3\), and the canonical \(\mathcal{O}(\varepsilon_\eta^{-1})\) MLMC cost is obtained for \(p>3\).

\paragraph{Sample allocation.}
Under Assumption~\ref{ass:polynomial-ll-convergence}, the optimal multilevel
allocation from \eqref{eq:generic-ml-sample-allocation} takes the form, up to constant factors,
\[
    N_\ell\propto \varepsilon_\eta^{-1} S_L (\ell+1)^{-(p+q)/2}, \qquad \ell\ge0.
\]
Thus, the number of samples decreases polynomially with the level. The decay
rate is faster when the level-difference error decays rapidly, i.e. when \(p\) is large, and when fine levels are expensive, i.e. when \(q\) is large. The normalizing factor \(S_L\) determines the overall scale of the allocation and satisfies
\[ S_L \lesssim 1+\sum_{\ell=1}^{L}\ell^{(q-p)/2} =
    \begin{cases}
        \mathcal{O}(1), & p>q+2,\\[0.3em]
        \mathcal{O}(\log L), & p=q+2,\\[0.3em]
        \mathcal{O}(L^{(q+2-p)/2}), & p<q+2.
    \end{cases}
\]
\subsubsection{Final-iterate SGD}
\label{subsubsec:final-iterate-sgd}
We first consider the standard final-iterate stochastic gradient descent. For fixed \(y\in\mathbb R^d\) and \(D\in\mathcal D\), suppose that the lower-level objective admits the stochastic representation
\[
    F_{y,D}(x) := \mathbb E_{\xi\sim\mathbb Q_{y,D}}\big[f_{y,D}(x,\xi)\big],
\]
where $f_{y,D}:\mathcal X\times \Xi\to\R$ and \(\xi:\Omega\to\Xi\) is a random variable that describes the auxiliary randomness used by the stochastic lower-level solver. In this subsection, we again assume that \(F_{y,D}\) is differentiable, \(\mu\)-strongly convex, and \(L_x\)-smooth in \(x\), with constants \(\mu>0\) and \(L_x<\infty\) that are uniform in \(y\) and \(D\).
Here, $(\Xi,\Sigma)$ is some abstract measurable space. We assume that, conditional on \(y\) and \(D\), i.i.d.~samples $(\xi_i)_{i\in\mathbb N}$ with $\xi_1\sim \mathbb Q_{y,D}$ can be generated. The loss function $f_{y,D}$ and distribution \(\mathbb Q_{y,D}\) may depend on the parameter \(y\) and the data \(D\), but the samples $(\xi_i)_{i\in\mathbb N}$ are conditionally independent given \(y\) and \(D\). The exact lower-level solution is
\[ h(y,D) = \argmin_{x\in\mathcal X}F_{y,D}(x)\,.\]
The SGD algorithm considered generates iterations
\begin{equation}\label{eq:SGD_iterates} 
x_{k+1} = x_k -\alpha_k \nabla_x f_{y,D}(x_k,\xi_{k+1}),\qquad k\ge0,
\end{equation}
with diminishing step-sizes $\alpha_k\asymp(k+1)^{-1}$. We consider the natural consecutive hierarchy, in which the level-$\ell$ approximation is the final iterate $\widehat h_\ell(y,D):=x_\ell$.
The following proposition records the rates needed for the abstract complexity result under Assumption~\ref{ass:polynomial-ll-convergence}. The required assumption of convergence with rate \eqref{eq:fi-sgd-bias} can be guaranteed under the imposed smoothness and strong convexity of $F_{y,D}$ uniformly over $y$ and $D$ and together with standard stochastic gradient noise assumptions \cite{B2018}.

\begin{prop}[Final-iterate SGD rates]
\label{prop:fi-sgd-rates}
Assume that, uniformly in \(y\in\mathbb R^d\) and \(D\in\mathcal D\), there exists some $C>0$ such that the SGD iterates satisfy
\begin{equation}\label{eq:fi-sgd-bias}
\mathbb E[\|x_\ell-h(y,D)\|_{\mathcal X}^2\mid y,X,D]\le C(\ell+1)^{-1}, \qquad \ell\ge0,
\end{equation}
and that the stochastic gradients have uniformly bounded second moments:
\begin{equation}\label{eq:sgd-gradient-second-moment}
    \mathbb E[\|\nabla_x f_{y,D}(x_\ell,\xi_{\ell+1})\|_{\mathcal X}^2\mid y,X,D]\le C,\qquad \ell\ge0.
\end{equation}
Then, assuming $\mathbb E[\textcolor{black}{\Lip_0(g,X)^2}]<\infty$, the final-iterate SGD satisfies Assumption~\ref{ass:polynomial-ll-convergence} with $r=1,\ p=2$ and $q=1$.
\end{prop}
\begin{proof}
The bias estimate in Assumption~\ref{ass:polynomial-ll-convergence} with \(r=1\) follows directly from \eqref{eq:fi-sgd-bias}
\begin{align*}
\mathbb E[\Lip_0(g,X)^2\|x_\ell-h(y,D)\|_{\mathcal X}^2] &= \mathbb E[\Lip_0(g,X)^2 \mathbb E[ \|x_\ell-h(y,D)\|_{\mathcal X}^2\mid y,X,D]] \\ &\le \mathbb E[\Lip_0(g,X)^2] C(\ell+1)^{-1}
\end{align*} 
For the level differences, note that
\[
    \mathbb E[\textcolor{black}{\Lip_0(g,X)^2}\|x_{\ell+1}-x_\ell\|_{\mathcal X}^2] = \alpha_\ell^2 \mathbb E[\textcolor{black}{\Lip_0(g,X)^2}\mathbb E[\|\nabla_x f_{y,D}(x_\ell,\xi_{\ell+1})\|_{\mathcal X}^2\mid y,X,D]] \le \mathbb E[\textcolor{black}{\Lip_0(g,X)^2}] C\alpha_\ell^2.
\]
Since \(\alpha_\ell\asymp(\ell+1)^{-1}\), we obtain $v_{\ell+1}\lesssim C(\ell+1)^{-2}$, i.e.~\(p=2\) in Assumption~\ref{ass:polynomial-ll-convergence}. Finally, computing \(x_\ell\) requires \(\ell\) stochastic-gradient
steps, so \(c_\ell=\mathcal{O}(\ell+1)\), which corresponds to \(q=1\) in Assumption~\ref{ass:polynomial-ll-convergence}.
\end{proof}

The following theorem shows that, for the consecutive final-iterate hierarchy for SGD, multilevel sampling may improve constants through better sample allocation, but it does not improve the asymptotic complexity exponent.
\begin{theorem}[Complexity for consecutive final-iterate SGD]
\label{thm:fi-sgd-complexity}
Suppose that Assumptions~\ref{ass:meas-integrability}, \ref{ass:lipschitz-loss}, \ref{ass:lipschitz-upper}, \ref{ass:fd-moment-bounds} and the assumptions of Proposition~\ref{prop:fi-sgd-rates} hold. Fix \(\eta>0\) and \(\varepsilon>0\), and set
$\varepsilon_\eta:=C_{d,\eta}^{-1}\varepsilon^2$. Then the single-level estimator achieves $\mathbb E[\|\GhatSL_{\eta,L,N}(y)-\nabla\Psi_\eta(y)\|^2]\le\varepsilon^2$ with cost $\operatorname{Cost}_{\rm SL}=\mathcal{O}(\varepsilon_\eta^{-2})$. Similarly, the multilevel estimator achieves the same accuracy with cost $\operatorname{Cost}_{\rm ML}= \mathcal{O}(\varepsilon_\eta^{-2})$.
\end{theorem}
\begin{proof}
By Proposition~\ref{prop:fi-sgd-rates}, final-iterate SGD satisfies Assumption~\ref{ass:polynomial-ll-convergence} with $r=1$, $p=2$ and $q=1$. The single-level cost follows from
\eqref{eq:polynomial-sl-cost}:
\[
    \operatorname{Cost}_{\rm SL} = \mathcal{O}(\varepsilon_\eta^{-1-q/r}) = \mathcal{O}(\varepsilon_\eta^{-2}).
\]
For the multilevel estimator, since \(p=2<q+2=3\), the third regime in \eqref{eq:polynomial-ml-cost-regimes} gives
\[
    \operatorname{Cost}_{\rm ML} = \mathcal{O}(\varepsilon_\eta^{-1-(q+2-p)/r}) = \mathcal{O}(\varepsilon_\eta^{-1-(1+2-2)}) =\mathcal{O}(\varepsilon_\eta^{-2})\,.
\]
\end{proof}

\paragraph{Sharpness for final-iterate SGD.} Theorem~\ref{thm:fi-sgd-complexity} shows that, for the consecutive hierarchy, final-iterate SGD does not improve the asymptotic multilevel complexity exponent under the natural iterative cost
model. We now show that this conclusion is sharp within the sufficient error-budget allocation model underlying Theorem~\ref{thm:generic-estimator-complexity}. Thus, the absence of an exponent improvement is not merely an artifact of upper-bound analysis.
\begin{prop}[Sharpness for final-iterate SGD]
\label{prop:fi-sgd-sharpness}
Consider the multilevel estimator based on consecutive lower-level approximations $\widehat h_0,\widehat h_1,\ldots,\widehat h_L$. Assume that the lower-level bias, level-difference errors, and level costs satisfy 
\begin{align*} 
c_b(\ell+1)^{-1} \le b_\ell \le C_b(\ell+1)^{-1},\quad 
c_v(\ell+1)^{-2}
    \le
    v_{\ell+1}
    \le
    C_v(\ell+1)^{-2}, \quad 
    c_c(\ell+1)
    \le
    c_\ell
    \le
    C_c(\ell+1),
\end{align*}
for all $\ell\ge0$ and some
constants \(0<c_b,C_b,c_v,C_v,c_c,C_c<\infty\). Let \(\varepsilon_\eta>0\) denote the effective estimator accuracy in the complexity analysis. Then every choice of finest level \(L\) and sample sizes
\(\mathbf N=(N_0,\ldots,N_L)\) satisfying $b_L\lesssim \varepsilon_\eta$ and $\sum_{\ell=0}^{L}\frac{v_\ell}{N_\ell}\lesssim \varepsilon_\eta$
has total cost $\operatorname{Cost}_{\rm ML} =\Omega(\varepsilon_\eta^{-2})$. 
\end{prop}
\begin{proof}
The bias constraint $b_L\lesssim \varepsilon_\eta$ implies $L+1 \gtrsim \varepsilon_\eta^{-1}$. We next lower-bound the work needed to satisfy the variance constraint. By Cauchy's inequality, for any positive sample sizes \(N_\ell\),
\[
    \Big(
        \sum_{\ell=0}^{L}\sqrt{v_\ell c_\ell}
    \Big)^2
    \le
    \Big(
        \sum_{\ell=0}^{L}\frac{v_\ell}{N_\ell}
    \Big)
    \Big(
        \sum_{\ell=0}^{L}N_\ell c_\ell
    \Big).
\]
Hence, every allocation satisfying
\(\sum_{\ell=0}^{L}v_\ell/N_\ell\lesssim\varepsilon_\eta\) must obey
\[
    \operatorname{Cost}_{\rm ML}
    =
    \sum_{\ell=0}^{L}N_\ell c_\ell
    \gtrsim
    \varepsilon_\eta^{-1}
    \Big(
        \sum_{\ell=0}^{L}\sqrt{v_\ell c_\ell}
    \Big)^2.
\]
By assumption, we have, for \(\ell\ge1\), $\sqrt{v_\ell c_\ell}\gtrsim\ell^{-1/2}$, and therefore,
\[
    \sum_{\ell=1}^{L}\sqrt{v_\ell c_\ell}
    \gtrsim
    \sum_{\ell=1}^{L}\ell^{-1/2}
    \gtrsim
    L^{1/2}.
\]
Together, this gives
\[
    \operatorname{Cost}_{\rm ML}
    \gtrsim
    \varepsilon_\eta^{-1}L\gtrsim
    \varepsilon_\eta^{-2}
\]
which proves the claim.
\end{proof}
Proposition~\ref{prop:fi-sgd-sharpness} should be interpreted as a sharpness result for the MLMC hierarchy considered here, not as a universal lower bound for all possible algorithms. It shows that if the lower-level approximations are the consecutive final iterates of SGD and the increments satisfy the natural
rate \(v_{\ell+1}\asymp(\ell+1)^{-2}\), then no choice of MLMC sample sizes can improve the exponent beyond \(\mathcal{O}(\varepsilon_\eta^{-2})\). The multilevel allocation may still improve constants in practice, but within this consecutive hierarchy, an asymptotic exponent gain requires faster decay of the coupled level differences. This is precisely what Polyak--Ruppert averaging provides: it preserves the bias rate \(b_\ell=\mathcal{O}((\ell+1)^{-1})\) while improving the level-difference decay to \(v_{\ell+1}=\mathcal{O}((\ell+1)^{-3})\), leading to an improved complexity derived below.

Note that the rates used in Proposition~\ref{prop:fi-sgd-sharpness} are natural for final-iterate SGD in the presence of non-vanishing stochastic-gradient noise. 
\begin{ex}
Consider the one-dimensional strongly convex quadratic model
\[
    F(x)=\frac{\mu}{2}x^2,
    \qquad
    \nabla f(x,\xi)=\mu x+\xi,
\]
where \(\mathbb E[\xi]=0\), \(\mathbb E[\xi^2]=\sigma_\xi^2>0\) and global minimum $x_\ast= 0$. Let $(x_\ell)$ be generated by SGD, $\mathcal F_\ell=\sigma(x_m,m\le \ell)$ denotes the canonical filtration and, for each $\ell$, $\xi_{\ell+1}$ be independent of $\mathcal F_\ell$. 
With step-sizes \(\alpha_\ell=\frac{a}{\ell+1},\ 2\mu a>1,\) the final-iterate mean-square error satisfies (cf.~\cite{JW2020})
\[
\mathbb E|x_\ell|^2\asymp(\ell+1)^{-1}.
\]
The consecutive increments satisfy
\[
    \mathbb E|x_{\ell+1}-x_\ell|^2
    =
    \alpha_\ell^2
    \mathbb E|\mu x_\ell+\xi_{\ell+1}|^2
    \ge
    \alpha_\ell^2\sigma_\xi^2
    \asymp
    (\ell+1)^{-2}.
\]
For the upper bound, note that 
\[\mathbb E|\mu x_\ell+\xi_{\ell+1}|^2 = \mu^2\mathbb E|x_\ell|^2 +\mathbb E|\xi_{\ell+1}|^2 \lesssim (\ell+1)^{-1} + \sigma_{\varepsilon} ^2\,.\] 
Thus, the decay \(v_{\ell+1}\asymp(\ell+1)^{-2}\) is not just an artifact of the proof. 
\end{ex}

Finally, we emphasize that the result of Proposition~\ref{prop:fi-sgd-sharpness} is specific to the consecutive hierarchy $\widehat h_\ell=x_\ell$. For alternative choices of the level hierarchy, the corresponding bias, level-difference, and cost rates may change, and it is not immediate whether the same complexity lower bound continues to hold. Thus, the absence of an exponent improvement established above should not be interpreted as a universal impossibility result for final-iterate SGD under arbitrary level hierarchies.

\subsubsection{Polyak--Ruppert averaged SGD}
\label{subsubsec:pr-sgd}
We now replace the final SGD iterate by the Polyak--Ruppert averaged iterate. Using the same SGD recursion from \eqref{eq:SGD_iterates}, define
\begin{equation}\label{eq:PRSGD}
    \bar x_\ell := \frac{1}{\ell+1} \sum_{k=0}^{\ell}x_k, \qquad \ell\ge0,
\end{equation}
and set $\widehat h_\ell(y,D):=\bar x_\ell$. In the following, we will abbreviate Polyak--Ruppert averaged SGD by PR-SGD. Averaging does not change the standard \(\mathcal{O}((\ell+1)^{-1})\) mean-square bias rate, but it improves the decay of consecutive level differences. This is the key property that makes multilevel sampling effective for stochastic lower-level solvers. The following lemma derives the improved level differences decay. The required assumption \eqref{eq:pr-average-rate} can be established under smoothness and strong convexity uniformly in $y$ and $D$ \cite{PJ92}.

\begin{lem}[Level-difference decay for PR--SGD]
\label{lem:pr-level-difference}
Assume that, uniformly in \(y\in\mathbb R^d\) and \(D\in\mathcal D\), there exists some $C>0$ such that
\begin{equation}\label{eq:pr-average-rate}
\mathbb E[\|x_{\ell+1}-h(y,D)\|_{\mathcal X}^2\mid y,X,D] \le C(\ell+1)^{-1} \quad \text{and}\quad   \mathbb E[\|\bar x_\ell-h(y,D)\|_{\mathcal X}^2\mid y,D]\le C(\ell+1)^{-1} 
\end{equation}
for all $\ell\ge0$. Then, assuming $\mathbb E[\textcolor{black}{\Lip_0(g,X)^2}]<\infty$, for all $\ell\ge0$,
\begin{equation}\label{eq:pr-level-difference-rate}
    \mathbb E[\textcolor{black}{\Lip_0(g,X)^2}\|\bar x_{\ell+1}-\bar x_\ell\|_{\mathcal X}^2] \le 4 \mathbb E[\textcolor{black}{\Lip_0(g,X)^2}]C(\ell+1)^{-3}\,.
\end{equation}
\end{lem}
\begin{proof}
By definition of the averaged iterates,
\[
    \bar x_{\ell+1} = \frac{1}{\ell+2}\sum_{k=0}^{\ell+1}x_k
 = \frac{\ell+1}{\ell+2}\bar x_\ell + \frac{1}{\ell+2}x_{\ell+1}
\]
such that $\bar x_{\ell+1}-\bar x_\ell=\frac{1}{\ell+2}(x_{\ell+1}-\bar x_\ell)\,$. 
Hence, using $\|a+b\|^2\le 2\|a\|^2+2\|b\|^2$, we obtain
\begin{align*}
    \mathbb E[\textcolor{black}{\Lip_0(g,X)^2}\|\bar x_{\ell+1}-\bar x_\ell\|^2]&=\frac{1}{(\ell+2)^2} \mathbb E [\textcolor{black}{\Lip_0(g,X)^2}\|x_{\ell+1}-\bar x_\ell\|^2] \\&\le
    \frac{2}{(\ell+2)^2}(\mathbb E[\textcolor{black}{\Lip_0(g,X)^2}\|x_{\ell+1}-h(y,D)\|^2]\\ &\qquad+\mathbb E[\textcolor{black}{\Lip_0(g,X)^2}\|\bar x_\ell-h(y,D)\|^2]) \\ &\le 4\mathbb E[\textcolor{black}{\Lip_0(g,X)^2}]C(\ell+1)^{-3}\,,
\end{align*}
where we used similar arguments as in the proof of Proposition~\ref{prop:fi-sgd-rates}.
\end{proof}
The following proposition is an immediate consequence and provides the verification of Assumption~\ref{ass:polynomial-ll-convergence}.
\begin{prop}[PR-SGD rates]
\label{prop:pr-sgd-rates}
Assume that the conditions of Lemma~\ref{lem:pr-level-difference} hold and that computing \(\bar x_\ell\) requires \(\ell\) stochastic-gradient steps. Then PR-SGD satisfies
Assumption~\ref{ass:polynomial-ll-convergence} with $r=1,\ p=3$ and $q=1$.
\end{prop}
The following theorem shows that Polyak--Ruppert averaging yields a multilevel improvement over the single-level complexity.
\begin{theorem}[Complexity for PR-SGD]
\label{thm:pr-sgd-complexity}
Suppose that Assumptions~\ref{ass:meas-integrability}, \ref{ass:lipschitz-loss}, \ref{ass:lipschitz-upper}, \ref{ass:fd-moment-bounds} and the assumptions of Proposition~\ref{prop:pr-sgd-rates} hold. Fix \(\eta>0\) and \(\varepsilon>0\), and set
$\varepsilon_\eta:=C_{d,\eta}^{-1}\varepsilon^2$. Then the single-level estimator achieves $\mathbb E[\|\GhatSL_{\eta,L,N}(y)-\nabla\Psi_\eta(y)\|^2]\le\varepsilon^2$ with $\operatorname{Cost}_{\rm SL}=\mathcal{O}(\varepsilon_\eta^{-2})$. The multilevel estimator achieves the same accuracy with $\operatorname{Cost}_{\rm ML} = \mathcal{O}(\varepsilon_\eta^{-1}\log^2(\varepsilon_\eta^{-1}))$.
\end{theorem}
\begin{proof}
By Proposition~\ref{prop:pr-sgd-rates}, PR-SGD satisfies Assumption~\ref{ass:polynomial-ll-convergence} with $r=1,\ p=3$ and $q=1$. The single-level cost follows from
\eqref{eq:polynomial-sl-cost}. For the multilevel estimator, \(p=q+2\). Therefore, the second regime in
\eqref{eq:polynomial-ml-cost-regimes} gives $\operatorname{Cost}_{ML}=\mathcal{O}(\varepsilon_\eta^{-1}\log^2(\varepsilon_\eta^{-1}))$.
\end{proof}

\subsection{Sample complexity by lower-level solver}
\label{subsec:solver-stationarity-complexity}
We now combine the complexity bounds derived above with the outer-loop convergence result from Section~\ref{sec:convergence}. The goal is to find a random iterate \(Y_{\kappa}\) satisfying $\mathbb E\|\nabla\Psi_\eta(Y_{\kappa})\|^2\le\varepsilon^2$. By Corollary~\ref{cor:smoothed-stationarity}, in the small-smoothing regime $\eta \Lip_0(\Psi)\sqrt{d}\lesssim 1$, it is sufficient to run
\begin{equation}\label{eq:outer-iterations-section4}
    K=\mathcal{O}\Big(\frac{\Lip_1(\Psi_\eta)}{\varepsilon^2}\Big) = \mathcal{O}\Big(\frac{\sqrt d\,\Lip_0(\Psi)}{\eta\varepsilon^2}\Big)
\end{equation}
outer iterations and to ensure that the gradient estimator at each iteration has MSE of order \(\varepsilon^2\). By Lemma~\ref{lem:gradient-mse}, the effective lower-level accuracy is \(\varepsilon_\eta=C_{d,\eta}^{-1}\varepsilon^2\). Thus, the total complexity is obtained by multiplying the per-iteration estimator cost at accuracy \(\varepsilon_\eta\) by the number of outer iterations in \eqref{eq:outer-iterations-section4}.
\begin{cor}[Total work complexity by lower-level solver]
\label{cor:solver-stationarity-complexity}
Fix \(\eta>0\). Suppose that Assumptions~\ref{ass:meas-integrability}, \ref{ass:lipschitz-loss}, \ref{ass:lipschitz-upper}, \ref{ass:fd-moment-bounds} and the assumptions of Corollary~\ref{cor:smoothed-stationarity} hold, and let $\eta \Lip_0(\Psi)\sqrt{d}\lesssim 1$. Then, the following total expected work
bounds are sufficient to find a random iterate \(Y_{\kappa}\) satisfying $\mathbb E[\|\nabla\Psi_\eta(Y_{\kappa})\|^2]\le\varepsilon^2$.
\begin{enumerate}
    \item[\rm(i)] If the lower-level solver satisfies the linear convergence assumptions of Assumption~\ref{ass:linear-ll-convergence}, then the single-level method
    has total cost
    \(
        \mathcal{O}\big(
            C_{d,\eta}\sqrt d\,\Lip_0(\Psi)\,
            \eta^{-1}\varepsilon^{-4}
            \log\!\left(C_{d,\eta}\varepsilon^{-2}\right)
        \big).
    \)
    The multilevel method has total cost
    \(
        \mathcal{O}\big(
            C_{d,\eta}\sqrt d\,\Lip_0(\Psi)\,
            \eta^{-1}\varepsilon^{-4}
        \big).
    \)
    \item[\rm(ii)]
    If the assumptions of Proposition~\ref{prop:fi-sgd-rates} hold, for final-iterate SGD with the consecutive hierarchy, the single-level and multilevel methods both have total cost 
    \(
        \mathcal{O}\!\big(C_{d,\eta}^2\sqrt d\,\Lip_0(\Psi)\,\eta^{-1}\varepsilon^{-6}
        \big).
    \)
    \item[\rm(iii)]
    If the assumptions of Proposition~\ref{prop:pr-sgd-rates} hold, for PR-SGD with the consecutive hierarchy, the single-level method has total cost    
    \(
        \mathcal{O}\big(
            C_{d,\eta}^2\sqrt d\,\Lip_0(\Psi)\,
            \eta^{-1}\varepsilon^{-6}
        \big).
    \)
    The multilevel method has total cost
    \(
        \mathcal{O}\big(
            C_{d,\eta}\sqrt d\,\Lip_0(\Psi)\,
            \eta^{-1}\varepsilon^{-4}
            \log^2(C_{d,\eta}\varepsilon^{-2})
        \big).
    \)
\end{enumerate}
\end{cor}
\begin{proof}
By Corollary~\ref{cor:smoothed-stationarity}, it is sufficient to run $K$ outer iterations satisfying \eqref{eq:outer-iterations-section4}, and to choose the gradient estimator at each iteration with MSE \(\varepsilon^2\). Lemma~\ref{lem:gradient-mse} shows that this corresponds to the required estimator accuracy $\varepsilon_\eta=C_{d,\eta}^{-1}\varepsilon^2$ in the lower-level MLMC complexity estimates. 

{\rm(i)}: For linearly convergent lower-level solvers, Theorem~\ref{thm:linear-ll-complexity} gives per-iteration costs
$\mathcal{O}(\varepsilon_\eta^{-1}\log(\varepsilon_\eta^{-1}))$ for the single-level estimator and $\mathcal{O}(\varepsilon_\eta^{-1})$
for the multilevel estimator. Multiplying by \(K\) in \eqref{eq:outer-iterations-section4} gives the result.
\smallskip

{\rm(ii)}: For final-iterate SGD with the consecutive hierarchy, Theorem~\ref{thm:fi-sgd-complexity} gives per-iteration cost $\mathcal{O}(\varepsilon_\eta^{-2})$
for both the single-level and multilevel estimators. Multiplying by \(K\) in \eqref{eq:outer-iterations-section4} gives the result.
\smallskip

{\rm(iii)}: For PR-SGD, Theorem~\ref{thm:pr-sgd-complexity} gives per-iteration costs $\mathcal{O}(\varepsilon_\eta^{-2})$ for the single-level estimator and $\mathcal{O}(\varepsilon_\eta^{-1}\log^2(\varepsilon_\eta^{-1}))$ for the multilevel estimator. Multiplying by \(K\) in \eqref{eq:outer-iterations-section4} gives the result.
\end{proof}

A summary of these results is provided in both Table \ref{tab:stationarity-complexity} and Table \ref{table:comp}.

\begin{table}[htb!]
\small
\centering
\begin{tabular}{c|c|c|c}
\hline
Lower-level solver
& Estimator
& Per-iteration estimator cost
& Total cost for \(\mathbb E\|\nabla\Psi_\eta(Y_{\kappa})\|^2\le\varepsilon^2\)
\\
\hline
Linear solver
& SL
& \(\mathcal{O}(C_{d,\eta}\varepsilon^{-2}\log(C_{d,\eta}\varepsilon^{-2}))\)
& \(\mathcal{O}(C_{d,\eta}\sqrt d \Lip_0(\Psi)\,\eta^{-1}\varepsilon^{-4}
\log((C_{d,\eta}\varepsilon^{-2}))\)
\\
Linear solver
& ML
& \(\mathcal{O}(C_{d,\eta}\varepsilon^{-2})\)
& \(\mathcal{O}(C_{d,\eta}\sqrt d \Lip_0(\Psi)\,\eta^{-1}\varepsilon^{-4})\)
\\
FI-SGD (consecutive)
& SL/ML
& \(\mathcal{O}(C_{d,\eta}^2\varepsilon^{-4})\)
& \(\mathcal{O}(C_{d,\eta}^2\sqrt d \Lip_0(\Psi)\,\eta^{-1}\varepsilon^{-6})\)
\\
PR-SGD
& SL
& \(\mathcal{O}(C_{d,\eta}^2\varepsilon^{-4})\)
& \(\mathcal{O}(C_{d,\eta}^2\sqrt d \Lip_0(\Psi)\,\eta^{-1}\varepsilon^{-6})\)
\\
PR-SGD
& ML
& \(\mathcal{O}(C_{d,\eta}\varepsilon^{-2}\log^2(C_{d,\eta}\varepsilon^{-2}))\)
& \(\mathcal{O}(C_{d,\eta}\sqrt d \Lip_0(\Psi)\,\eta^{-1}\varepsilon^{-4}
\log^2(C_{d,\eta}\varepsilon^{-2}))\)
\\
\hline
\end{tabular}
\caption{Estimator and total work complexity for fixed smoothing parameter \(\eta>0\) in the small-smoothing regime $\eta \Lip_0(\Psi)\sqrt{d}\lesssim 1$.}
\label{tab:stationarity-complexity}
\end{table}
\subsection{Discussion}
The preceding results show that multilevel sampling does not automatically improve the complexity of derivative-free bilevel optimization. The decisive quantity is the decay of the coupled level differences. In the linear convergence regime, this decay is geometric and the multilevel estimator achieves the canonical Monte Carlo cost up to the smoothing factor. In the polynomial regime, the bias rate alone is insufficient to determine the
complexity: final-iterate SGD and Polyak--Ruppert averaged SGD with the consecutive hierarchies considered above both have \(\mathcal{O}(\ell^{-1})\) mean-square bias, but they have different level-difference rates. Final-iterate SGD yields \(v_{\ell+1}=\mathcal{O}(\ell^{-2})\), which does not improve the asymptotic complexity exponent, whereas Polyak--Ruppert averaging
yields \(v_{\ell+1}=\mathcal{O}(\ell^{-3})\), leading to the improved multilevel cost $=\mathcal{O}(C_{d,\eta}\varepsilon^{-2}\log^2((C_{d,\eta}\varepsilon^{-2}))$
per outer iteration. Thus, the effectiveness of MLMC is governed by the interaction between lower-level bias, level-difference decay, and per-level cost. This observation is the main solver-dependent message of the paper: MLMC rewards lower-level algorithms whose consecutive approximations are strongly coupled, and Polyak--Ruppert averaging provides such a coupling for stochastic gradient methods.

\begin{table}[htb!]
\centering
\small
\begin{tabular}{c|c|c|c|c}
\hline
Lower-level solver
& Bias \(b_L\)
& Level differences \(v_\ell\)
& Finest level \(L\)
& MLMC allocation \(N_\ell\)
\\
\hline
Linear solver
&
\(\rho^L\)
&
\(\rho^\ell\)
&
\(\displaystyle L\propto \frac{\log(\varepsilon_\eta^{-1})}{\log(\rho^{-1})}\)
&
\(\displaystyle
N_0\propto \varepsilon_\eta^{-1},\quad
N_\ell\propto
\varepsilon_\eta^{-1}
\sqrt{\frac{\rho^\ell}{\ell+1}}
\)
\\[1.2em]
FI-SGD (consecutive)
&
\(L^{-1}\)
&
\(\ell^{-2}\)
&
\(\displaystyle L\propto \varepsilon_\eta^{-1}\)
&
\(\displaystyle
N_\ell\propto
\varepsilon_\eta^{-1}L^{1/2}\ell^{-3/2}
\)
\\[1.2em]
PR-SGD (consecutive)
&
\(L^{-1}\)
&
\(\ell^{-3}\)
&
\(\displaystyle L\propto \varepsilon_\eta^{-1}\)
&
\(\displaystyle
N_\ell\propto
\varepsilon_\eta^{-1}\log(L)\,\ell^{-2}
\)
\\
\hline
\end{tabular}
\caption{Summary of level choices and continuous MLMC sample allocations for the lower-level solver classes. Here
\(\varepsilon_\eta:=C_{d,\eta}^{-1}\varepsilon^2\) denotes the accuracy parameter entering the gradient-estimator MSE bound.}
\label{table:comp}
\end{table}

\begin{rem}
The complexity bounds in Table~\ref{tab:stationarity-complexity} are stated with the constant \(C_{d,\eta}\) left implicit. When the finite-difference moment bound satisfies \(C_{\mathrm{fd}}(d,\eta) \asymp d\eta^{-2}\), 
which is the case under bounded sample loss, we have 
\(C_{d,\eta} \asymp d\eta^{-2}\). 
In the small-smoothing regime $\eta \Lip_0(\Psi)\sqrt{d}\lesssim 1$, the total complexity for PR-SGD becomes
$$
\mathrm{Cost}_{\mathrm{SL}}^{\mathrm{PR}}
= \mathcal{O}\!\left(\frac{d^2 \sqrt{d}\Lip_0(\Psi)}{\eta^5}\,\varepsilon^{-6}\right),
\qquad
\mathrm{Cost}_{\mathrm{ML}}^{\mathrm{PR}}
= \mathcal{O}\!\left(\frac{d \sqrt{d}\Lip_0(\Psi)}{\eta^3}\,\varepsilon^{-4}\log^2(C_{d,\eta}\varepsilon^{-2})\right).
$$
Thus, MLMC not only improves the \(\varepsilon\)-exponent from \(\varepsilon^{-6}\) to 
\(\varepsilon^{-4}\log^2(\varepsilon^{-1})\), but also improves the \(\eta\)-dependence 
from \(\eta^{-5}\) to \(\eta^{-3}\) up to logarithmic factors. In contrast, for final-iterate SGD with the consecutive hierarchy considered above, both SLMC and MLMC share the same \(\eta^{-5}\varepsilon^{-6}\) scaling, reflecting the absence of a multilevel speedup in that setting. This is important as it means  that the algorithm is less sensitive to choosing a very small $\eta$ to achieve high accuracy.
\end{rem}

\section{Extension to proximal outer-level variant}\label{sec:prox-extensions}
The main analysis was formulated and studied for the unconstrained smoothed problem
\[
    \min_{y\in\mathbb R^d}\Psi_\eta(y)\,.
\]
Although this is convenient for presenting the multilevel estimator, it requires uniform lower-level approximation bounds on all \(y\in\mathbb R^d\). In many applications the admissible hyperparameters are constrained, or one may want to
include a nonsmooth convex outer-level penalty. In this section, we describe a proximal variant of the method and show how the uniformity requirements can be localized. The price is that Gaussian smoothing still evaluates the objective outside the feasible set; these excursions are controlled explicitly by
Gaussian tail factors. We consider the composite smoothed problem
\begin{equation}\label{eq:prox-composite-problem}
    \min_{y\in\mathbb R^d}\ \{\Psi_\eta(y)+R(y)\}\,,
\end{equation}
where \(R:\mathbb R^d\to(-\infty,\infty]\) is proper, closed, and convex. 
We assume that the proximal map
\[
    \operatorname{prox}_{\gamma R}(z) :=\argmin_{y\in\mathbb R^d} \big\{R(y)+\frac{1}{2\gamma}\|y-z\|^2\big\}
\]
is efficiently computable. Given a stochastic estimator
\(\widehat{V}_{k+1}\) of \(\nabla\Psi_\eta(Y_k)\), the proximal outer iteration is
\begin{equation}\label{eq:prox-outer-update}
    Y_{k+1} = \operatorname{prox}_{\gamma R} \big(Y_k-\gamma\widehat{V}_{k+1}\big)\,.    
\end{equation}
For non-convex optimization, the natural stationarity measure for \eqref{eq:prox-composite-problem} is the proximal gradient mapping
\begin{equation}\label{eq:prox-gradient-mapping}
    \mathcal M_{\eta,\gamma}^R(y) := \frac1\gamma
    (y-\operatorname{prox}_{\gamma R}(y-\gamma\nabla\Psi_\eta(y)))\,.
\end{equation}

\subsection{Localized assumptions and Gaussian tail factors} \label{subsec:localized-tail-assumptions}
Assume in this subsection that the constrained case \(R=\delta_{\mathcal Y}\) is considered and that \(\mathcal Y\subset\mathbb R^d\) is compact and convex. For a nonempty closed convex set \(\mathcal Y\subset\mathbb R^d\), we denote by
\(\delta_{\mathcal Y}\) its convex indicator function,
\[
    \delta_{\mathcal Y}(y) :=
    \begin{cases}
        0, & y\in\mathcal Y,\\
        +\infty, & y\notin\mathcal Y.
    \end{cases}
\]
The proximal operator becomes in this case the orthogonal projector onto the set $\scrY$. The iteration reads accordingly
as
\[
    Y_{k+1} = \Pi_{\mathcal Y}\big(Y_k-\gamma\widehat{V}_{k+1}\big)\,,
\]
where $\Pi_{\mathcal Y}:\mathbb R^d\to\mathcal Y$ denotes the orthogonal projection onto $\mathcal Y$. Clearly, the projected iteration keeps \(Y_k\in\mathcal Y\). However, the Gaussian finite-difference estimator evaluates the lower-level solver at \(Y_k+\eta U\), where \(U\sim\mathcal N(0,I_d)\), and therefore at points outside \(\mathcal Y\). The proposed localized analysis does not remove the domain requirement, but it weakens the uniform lower-level approximation assumptions: accurate solver estimates are only required on a slightly enlarged set \(\mathcal Y^r\) (to be defined below), while evaluations outside \(\mathcal Y^r\) are controlled through Gaussian tail moments.

For \(r>0\), define 
\begin{equation}\label{eq:fixed-enlargement}
    \mathcal Y^r := \mathcal Y+rB =\{y+z:\ y\in\mathcal Y,\ \|z\|\le r\}\,.
\end{equation}
Then, we can guarantee
\[
    y+\eta U\in\mathcal Y^r \qquad \text{whenever } y\in\mathcal Y \text{ and } \eta\|U\|\le r.
\]
Thus, the probability of leaving the region \(\mathcal Y^r\) decreases as \(\eta\) decreases. To make this precise, we use the following Gaussian tail factor. Let \(s>1\) and \(s':=s/(s-1)\). For \(a\ge0\), define
\begin{equation} \label{eq:gaussian-tail-factor}
    \mathfrak T_{s,d}(a) := \Big(
        \mathbb E
        \big[
            \|U\|^{2s'}
            \mathds 1_{\{\|U\|>a\}}
        \big]
    \Big)^{1/s'},
    \qquad
    U\sim\mathcal N(0,I_d).
\end{equation}
By uniform integrability of the Gaussian measure, we have $\mathfrak T_{s,d}(a)\to0$ as $a\to\infty$. In particular, for a fixed \(r>0\), $\mathfrak T_{s,d}(r/\eta)\to0$ as $\eta\downarrow0$. We now define localized lower-level error quantities on \(\mathcal Y^r\):
\begin{equation}\label{eq:localized-bias-r}
    b_\ell^r := \sup_{y\in\mathcal Y^r}\ \mathbb E
    \big[ \Lip_0(g,X)^2 \|\widehat h_\ell(y,D)-h(y,D)\|_{\mathcal X}^2\big],
\end{equation}
and
\begin{equation}\label{eq:localized-difference-r}
    v_{\ell+1}^r := \sup_{y\in\mathcal Y^r}\ \mathbb E\big[\Lip_0(g,X)^2 \|\widehat h_{\ell+1}(y,D)-\widehat h_\ell(y,D)\|_{\mathcal X}^2\big].
\end{equation}
The weights \(\Lip_0(g,X)^2\) arise from the Lipschitz estimate for the finite-difference numerator. To control the Gaussian tail outside \(\mathcal Y^r\), we assume uniform higher moments of the weighted lower-level errors along Gaussian perturbations.

\begin{ass}[Moment bounds]
\label{ass:tail-moment-bounds}
There exist \(s>1\), constants \(\bar b_\ell<\infty\) and constants \(\bar v_{\ell+1}<\infty\), $\ell\ge0$, such that, for every \(\eta\in(0,\eta_{\max}]\),
\begin{equation}\label{eq:tail-moment-bias}
    \sup_{y\in\mathcal Y}\ \mathbb E \big[(\Lip_0(g,X)^2\|\widehat h_\ell(y+\eta U,D)-h(y+\eta U,D)\|_{\mathcal X}^2)^s\big]\le\bar b_\ell^s,
\end{equation}
and
\begin{equation}\label{eq:tail-moment-difference}
    \sup_{y\in\mathcal Y}\ \mathbb E\big[(\Lip_0(g,X)^2\|\widehat h_{\ell+1}(y+\eta U,D)-\widehat h_\ell(y+\eta U,D)\|_{\mathcal X}^2)^s\big]\le\bar v_{\ell+1}^s,
\end{equation}
where \(U\sim\mathcal N(0,I_d)\) is independent of \(Z=(X,D)\).
\end{ass}

\begin{lem}[Localized MSE bounds]
\label{lem:localized-mse-explicit-tail}
Let \(\mathcal Y\subset\mathbb R^d\) be compact and convex. Suppose Assumptions~\ref{ass:meas-integrability}, \ref{ass:lipschitz-loss}, \ref{ass:fd-moment-bounds}, and
\ref{ass:tail-moment-bounds} hold. Then, for every \(y\in\mathcal Y\), every \(r>0\), and every \(\eta\in(0,\eta_{\max}]\), the single-level estimator
satisfies
\begin{equation}\label{eq:localized-sl-mse-explicit}
    \mathbb E\Big[\big\|\GhatSL_{\eta,\ell,N}(y)-\nabla\Psi_\eta(y)\big\|^2\Big]\le
    C_{d,\eta} \big(\frac1N + b_\ell^r + \bar b_\ell\,\mathfrak T_{s,d}(r/\eta)\big)\,.
\end{equation}
Similarly, the multilevel estimator satisfies
\begin{equation}\label{eq:localized-ml-mse-explicit}
    \mathbb E\Big[\big\|\GhatML_{\eta,L,\mathbf N}(y)-\nabla\Psi_\eta(y)\big\|^2\Big]\le C_{d,\eta}\big(\frac1{N_0}+\sum_{\ell=0}^{L-1}\frac{v_{\ell+1}^r+\bar v_{\ell+1}\mathfrak T_{s,d}(r/\eta)}{N_{\ell+1}}+b_L^r+\bar b_L\mathfrak T_{s,d}(r/\eta)\big)\,.
\end{equation}
In both cases, the constant \(C_{d,\eta}>0\) is independent of \(L\), \(\mathbf N\), \(r\), but satisfies \(C_{d,\eta}\propto d\eta^{-2} + C_{\mathrm{fd}}(d,\eta)\).
\end{lem}

\begin{proof}
We prove the single-level estimate, the multilevel estimate follows by the same argument applied to each correction. Let
\[
    Q_\ell(y;U,Z) := \frac{G_\ell(y+\eta U;Z)-G_\ell(y;Z)}{\eta}U, \qquad Q_\infty(y;U,Z) := \frac{G_\infty(y+\eta U;Z)-G_\infty(y;Z)}{\eta}U.
\]
By the Lipschitz property of \(g\),
\[    \|Q_\ell(y;U,Z)-Q_\infty(y;U,Z)\|
    \le
    \frac{\Lip_0(g,X)\|U\|}{\eta}
    \Big(\|\widehat h_\ell(y+\eta U,D)-h(y+\eta U,D)\|_{\mathcal X} + \|\widehat h_\ell(y,D)-h(y,D)\|_{\mathcal X}
    \Big).
\]
The term at \(y\in\mathcal Y\) is bounded by \(b_\ell^r\), since
\(\mathcal Y\subset\mathcal Y^r\). For the shifted term, we split the expectation into
\[
    \{\eta\|U\|\le r\} \qquad\text{and}\qquad \{\eta\|U\|>r\}.
\]
On \(\{\eta\|U\|\le r\}\), the perturbed point \(y+\eta U\) belongs to \(\mathcal Y^r\), and the contribution is again controlled by \(b_\ell^r\). On the tail event, Hölder's inequality and Assumption~\ref{ass:tail-moment-bounds}
give
\begin{align*}
    &\mathbb E
    \big[
        \|U\|^2 \Lip_0(g,X)^2
        \|\widehat h_\ell(y+\eta U,D)-h(y+\eta U,D)\|_{\mathcal X}^2
        \mathbf 1_{\{\eta\|U\|>r\}}
    \big]
    \\
    &\quad\le
    \Big(
        \mathbb E
        \big[
            (
                \Lip_0(g,X)^2
                \|\widehat h_\ell(y+\eta U,D)-h(y+\eta U,D)\|_{\mathcal X}^2
            )^s
        \big]
    \Big)^{1/s}
    \Big(
        \mathbb E
        \big[
            \|U\|^{2s'}
            \mathds 1_{\{\|U\|>r/\eta\}}
        \big]
    \Big)^{1/s'} \\   & \quad \le
    \bar b_\ell\,\mathfrak T_{s,d}(r/\eta).
\end{align*}
Similarly as in the proof of Lemma~\ref{lem:gradient-mse}, we use the finite-difference sampling variance bound from
Assumption~\ref{ass:fd-moment-bounds}, to deduce \eqref{eq:localized-sl-mse-explicit}.

For the multilevel estimator, the correction numerator satisfies
\[
    \|\Delta Q_{\ell+1}(y;U,Z)\|
    \le
    \frac{\Lip_0(g,X)\|U\|}{\eta}
    \Big(
        \|\widehat h_{\ell+1}(y+\eta U,D)
          -\widehat h_\ell(y+\eta U,D)\|_{\mathcal X}
        +\|\widehat h_{\ell+1}(y,D)
          -\widehat h_\ell(y,D)\|_{\mathcal X}
    \Big).
\]
The local part is bounded by \(v_{\ell+1}^r\), while the tail part is bounded by $\bar v_{\ell+1}\mathfrak T_{s,d}(r/\eta)$ by the same Hölder argument from above. The finest-level bias is controlled as in the single-level case with \(\ell=L\). This proves \eqref{eq:localized-ml-mse-explicit}.
\end{proof}

\subsection{Proximal convergence with localized MLMC estimators}
\label{subsec:prox-convergence-localized}
We now combine the localized estimator bounds with a stochastic proximal gradient descent argument.
\begin{theorem}[Proximal convergence with estimator error]
\label{thm:prox-controlled-mse}
Assume that \(\Psi_\eta+R\) is bounded below by \(\Psi_{\eta,R}^\star>-\infty\). Let \((Y_k)_{k\ge0}\) be generated by \eqref{eq:prox-outer-update} with $0<\gamma\le \frac1{\Lip_1(\Psi_\eta)}$ and $e_{k+1}:=\widehat{V}_{k+1}-\nabla\Psi_\eta(Y_k)$. Then, for every \(K\ge1\),
\begin{equation}\label{eq:prox-convergence-bound}
    \frac1K\sum_{k=0}^{K-1}
    \mathbb E[\|\mathcal M_{\eta,\gamma}^R(Y_k)\|^2]\le\frac{8}{\gamma K}(\Psi_\eta(Y_0)+R(Y_0)-\Psi_{\eta,R}^\star)+\frac{10}{K}\sum_{k=0}^{K-1}\mathbb E\|e_{k+1}\|^2\,.
\end{equation}
\end{theorem}

\begin{proof}
Let $\Delta{Y}_{k+1}:=\frac1\gamma(Y_k-Y_{k+1})$ be the stochastic proximal gradient mapping. The optimality condition for the proximal step gives, with \(z=Y_k\),
\[
    R(Y_{k+1})-R(Y_k)\le \gamma\langle \widehat{V}_{k+1},\Delta Y_{k+1}\rangle-\gamma\|\Delta Y_{k+1}\|^2.
\]
By \(\Lip_1(\Psi_\eta)\)-smoothness of \(\Psi_\eta\),
\[
    \Psi_\eta(Y_{k+1})\le \Psi_\eta(Y_k)- \gamma\langle\nabla\Psi_\eta(Y_k),\Delta Y_{k+1}\rangle+\frac{\Lip_1(\Psi_\eta)\gamma^2}{2} \|\Delta Y_{k+1}\|^2.
\]
Adding the two inequalities and using $\widehat{V}_{k+1}
    =
    \nabla\Psi_\eta(Y_k)+e_{k+1}$,
we obtain, for \(\gamma\le 1/\Lip_1(\Psi_\eta)\),
\[
    \Psi_\eta(Y_{k+1})+R(Y_{k+1})
    \le
    \Psi_\eta(Y_k)+R(Y_k)
    -
    \frac{\gamma}4
    \|\Delta Y_{k+1}\|^2
    +
    \gamma\|e_{k+1}\|^2\,.
\]
Here, we used $$1-\frac{\Lip_1(\Psi_\eta)\gamma}{2}\ge \frac12,$$ and Young's inequality,
\[
    \langle e_{k+1},\Delta Y_{k+1}\rangle
    \le
    \frac14
    \|\Delta Y_{k+1}\|^2
    +
    \|e_{k+1}\|^2\,.
\]
Moreover, by non-expansiveness of the proximal map,
\begin{align*}
    \|\mathcal M_{\eta,\gamma}^R(Y_k)\|^2
    &\le
    2\|\Delta Y_{k+1}\|^2+2\|e_{k+1}\|^2\\
    &\le \frac{8(\Psi_\eta(Y_k)-\Psi_{\eta}(Y_{k+1}) + R(Y_k)-R(Y_{k+1}))}{\gamma} +10 \|e_{k+1}\|^2
\end{align*}
Taking expectations, summing over \(k=0,\ldots,K-1\), and using the lower bound \(\Psi_{\eta,R}^\star\) gives \eqref{eq:prox-convergence-bound} and completes the proof.
\end{proof}

The next corollary gives the localized constrained result with explicit \(\eta\)-dependent tails.
\begin{cor}[Projected MLMC stationarity] \label{cor:projected-stationarity-explicit-tail}
Let \(R=\delta_{\mathcal Y}\), where \(\mathcal Y\subset\mathbb R^d\) is compact and convex. Let the assumptions of Theorem~\ref{thm:prox-controlled-mse} and Lemma~\ref{lem:localized-mse-explicit-tail} hold. Let
\(\kappa \sim{\rm Unif}\{0,\ldots,K-1\}\) be independent of all algorithmic randomness. For the single-level estimator with fixed $L\ge0,\ N\in\mathbb N$, there exists $C>0$ such that it holds
\begin{equation}\label{eq:projected-sl-stationarity-explicit}
    \mathbb E[\|\mathcal M_{\eta,\gamma}^{\delta_{\mathcal Y}}(Y_{\kappa})\|^2]\le C\Big[\frac{\Psi_\eta(Y_0)-\Psi_{\eta,\mathcal Y}^\star}{\gamma K}+C_{d,\eta}\big(\frac1N+b_L^r+\bar b_L\,\mathfrak T_{s,d}(r/\eta)\big)
    \Big].
\end{equation}
For the multilevel estimator, there exists $C>0$, such that it holds
\begin{equation*}
    \mathbb E[\|\mathcal M_{\eta,\gamma}^{\delta_{\mathcal Y}}(Y_{\kappa})\|^2]\le
    C\Bigg[
        \frac{
            \Psi_\eta(Y_0)-\Psi_{\eta,\mathcal Y}^\star
        }{\gamma K}+
        C_{d,\eta}
        \left(
            \frac1{N_0}
            +
            \sum_{\ell=0}^{L-1}
            \frac{
                v_{\ell+1}^r
                +
                \bar v_{\ell+1}\mathfrak T_{s,d}(r/\eta)
            }{N_{\ell+1}}
            +
            b_L^r
            +
            \bar b_L\mathfrak T_{s,d}(r/\eta)
        \right)
    \Bigg],
\end{equation*}
where $\Psi_{\eta,\mathcal Y}^\star:=\inf_{y\in\mathcal Y}\Psi_\eta(y)$. Consequently, to obtain $\mathbb E[\|\mathcal M_{\eta,\gamma}^{\delta_{\mathcal Y}}(Y_{\kappa})\|^2]\le \varepsilon^2$, it is sufficient to choose $K=\mathcal{O}(\frac{\Lip_1(\Psi_\eta)}{\varepsilon^2})$, the local Monte Carlo and lower-level errors of order $\mathcal{O}(C_{d,\eta}^{-1}\varepsilon^2)$, and the tail radius \(r\) such that
\begin{equation}\label{eq:tail-accuracy-condition}
    \bar b_L\mathfrak T_{s,d}(r/\eta)
    +
    \sum_{\ell=0}^{L-1}
    \frac{
        \bar v_{\ell+1}\mathfrak T_{s,d}(r/\eta)
    }{N_{\ell+1}}
    \lesssim
    C_{d,\eta}^{-1}\varepsilon^2\,.
\end{equation}
\end{cor}
We use the following explicit tail representation of $\mathfrak T_{s,d}$. Since \(\|U\|^2\) has a \(\chi_d^2\)-distribution, the quantity $\mathfrak T_{s,d}(\cdot)$ can be estimated as follows:
Let $U\sim N(0,I_d)$ and apply Cauchy--Schwarz,
\[
\mathbb E[\|U\|^{2s'} \mathds 1_{\{\|U\|>a\}}]
\le
\big(\mathbb E \|U\|^{4s'}\big)^{1/2}
\big(\mathbb P(\|U\|>a)\big)^{1/2}.
\]
Since $\|U\|^2\sim \chi_d^2$, we have
\[
\mathbb E \|U\|^{4s'} = 2^{2s'}\frac{\Gamma(2s'+d/2)}{\Gamma(d/2)}\,,
\]
where $\Gamma$ denotes the gamma function $\Gamma(z) = \int_0^\infty t^{z-1}e^{-t}\,{\rm d}t$, $z>0$. Therefore,
\[
\mathbb E[\|U\|^{2s'} \mathds 1_{\{\|U\|>a\}}]
\le
\Big(
2^{2s'}\frac{\Gamma(2s'+d/2)}{\Gamma(d/2)}
\Big)^{1/2}
\big(\mathbb P(\|U\|>a)\big)^{1/2}.
\]
Using the Gaussian concentration bound
\[
\mathbb P(\|U\|>\sqrt d+t)\le e^{-t^2/2}, \qquad t\ge 0,
\]
with $t=a-\sqrt d$, we obtain, for $a\ge \sqrt d$,
\[
\mathbb P(\|U\|>a)\le e^{-(a-\sqrt d)^2/2}.
\]
Hence,
\[
\mathfrak T_{s,d}(a)
\le
2\left(
\frac{\Gamma(2s'+d/2)}{\Gamma(d/2)}
\right)^{\frac{1}{2s'}}
e^{-\frac{(a-\sqrt d)^2}{4s'}}
\]
for all $a\ge \sqrt d$. Consequently, we require $r/\eta \ge \sqrt{d}$ and the left-hand side of \eqref{eq:tail-accuracy-condition} is bounded by
\[
\bar b_L\mathfrak T_{s,d}(r/\eta) + \sum_{\ell=0}^{L-1}
\frac{
    \bar v_{\ell+1}\mathfrak T_{s,d}(r/\eta)
}{N_{\ell+1}}
\le
B_{s,d} \bigg(
\bar b_L
+
\sum_{\ell=0}^{L-1}
\frac{\bar v_{\ell+1}}{N_{\ell+1}}
\bigg)
\exp\!\bigg(
-\frac{(r/\eta-\sqrt d)^2}{4s'}
\bigg),
\]
with $B_{s,d}:=2\big(\frac{\Gamma(2s'+d/2)}{\Gamma(d/2)} \big)^{\frac{1}{2s'}}$. A sufficient condition for \eqref{eq:tail-accuracy-condition} is
\[ \frac{(r/\eta-\sqrt d)^2}{4s'} \ge \log\!\Bigg(\frac{
B_{s,d}
(
\bar b_L
+
\sum_{\ell=0}^{L-1}
\frac{\bar v_{\ell+1}}{N_{\ell+1}}
)
}{
C_{d,\eta}^{-1}\varepsilon^2
}
\Bigg),
\qquad r/\eta\ge \sqrt d.
\]
For $\eta\in(0,\frac{r}{2\sqrt{d}})$, we may choose $N_\ell$ such that
\[
\bar b_L + \sum_{\ell=0}^{L-1} \frac{\bar v_{\ell+1}}{N_{\ell+1}}
\lesssim \frac{C_{d,\eta}^{-1}\varepsilon^2}{B_{s,d}} \exp\bigg(\frac{r^2}{16s'\eta^2}\bigg)\,.
\]

Thus, the additional region on which accurate lower-level approximations are required grows only at the Gaussian concentration scale determined by the smoothing parameter and dimension.



\section{Numerical experiments}
\label{sec:numerics}

In this section, we present numerical experiments illustrating the performance of the proposed ML-DFO method. 
We consider two imaging examples. The first is a large-scale color image denoising problem with a learned Field-of-Experts regularizer. This example
uses a stochastic lower-level solver and is designed to test the multilevel estimator in a high-dimensional parameter-learning setting. The second example is sparse-view computed tomography based on the Radon transform, where the lower-level reconstruction is solved using FISTA. 
All costs are reported in units of cumulative lower-level work. Thus, for a lower-level iterative method, one unit of work corresponds to one lower-level iteration.

These experiments demonstrate the practical relevance of multilevel sampling under implementation choices motivated by computational efficiency, while Section~\ref{subsubsec:numerical-complexity-rates} provides a more targeted empirical illustration of the complexity predictions developed in Section~\ref{sec:solver-complexity}.

\subsection{Large-scale color denoising with a Field-of-Experts regularizer}
\label{subsec:foe-denoising}
Our first example is a color image denoising problem with a learned Field-of-Experts (FoE) regularizer. The upper-level method tunes a convolutional regularizer with multiple color
filters and learned potential weights, while the lower-level problem is solved approximately by a stochastic method. We use RGB images from the BSDS500 dataset corrupted by additive Gaussian noise,
\[
    D = X + W, \qquad W\sim \mathcal N(0,\sigma^2 I),
\]
with \(\sigma=25/255\). For a clean image \(X\) and noisy observation \(D\), the upper-level objective is
\begin{equation}\label{eq:foe-ul}
     \min_{y\in\mathbb R^d}
    \mathbb E\!\left[
        \frac12\|h_{\theta(y)}(D)-X\|^2
    \right],
\end{equation}
where the reconstruction \(h_{\theta(y)}\) is defined by the variational problem
\begin{equation}\label{eq:foe-ll}
    h_{\theta(y)}(D)
    =
    \argmin_x
    \left\{
        \frac12\|x-D\|^2+\lambda R_{\theta(y)}(x)
    \right\}.
\end{equation}
In the reported experiments, we set \(\lambda=1\) and the learned regularizer is of Field-of-Experts type \cite{goujon2024weaklyconvex},
\begin{equation}\label{eq:foe-regularizer}
    R_{\theta(y)}(x) = \tau \sum_{m=1}^{M} \omega_m \sum_{p\in\Omega}\psi_{\nu_m}\big((K_m*x)_p\big),
\end{equation}
where $
    \Omega:=\{1,\ldots,H\}\times\{1,\ldots,W\},
$
with \(H\) and \(W\) representing the image height and width. Thus,
\((K_m*x)_p\) is the scalar response of the \(m\)-th convolutional expert
at pixel \(p\in\Omega\). Moreover, \(K_m\) are color convolution filters, \(\omega_m>0\) are expert weights, \(\nu_m>0\) are smoothing parameters, and \(\tau>0\) is a global scale. Positivity of the physical parameters is enforced by a componentwise transformation of the unconstrained optimization variables. We use the smooth Huber-type potential
\[
    \psi_\nu(t)  =
    \begin{cases}
        \frac{t^2}{2\nu}, & |t|\le \nu,\\
        |t|-\frac{\nu}{2}, & |t|>\nu.
    \end{cases}
\]
In our implementation, \(M=8\) filters of size \(5\times5\) are used. The lower-level problem is solved using the Polyak--Ruppert variation of SGD. The finest lower-level budget is $500$ iterations, and the multilevel hierarchy is chosen according to the target accuracy parameter \(\varepsilon\). The single-level baseline uses the same upper-level optimization routine, but evaluates all samples at the finest lower-level budget.

We compare SLMC and MLMC for the target accuracies \(\varepsilon^2\in\{1,0.5,0.1\}\). The performance metrics are the validation PSNR, validation SSIM, and the training objective, plotted against cumulative lower-level work. 
The number of MLMC levels is inspired by the theoretical bounds reported in Table \ref{table:comp}. In the experiments, we use 
$ L=\left\lceil 2\varepsilon^{-2}\right\rceil$. At level \(\ell\), \(T_\ell\) denotes the maximum number of lower-level
iterations used to compute the corresponding reconstruction. These iteration
budgets are distinct from the level-dependent image sample sizes \(N_\ell\). 
The schedules used in the experiments are summarized in
Table~\ref{tab:foe-level-budgets}.
The level-dependent sample sizes are $N_\ell=\lfloor 4\varepsilon^{-2}(\ell+1)^{-1}\rfloor$,
$\ell=0,\ldots,L-1$, i.e.\ the number of training images drawn at level $\ell$ in each
upper-level step. 
For a fair comparison, the
single-level baseline uses a batch of $N=\sum_{\ell}N_\ell$ images, all evaluated at the
finest budget $T_{L-1}=500$.

\begin{table}[h!]
\centering
\small
\begin{tabular}{c c c c c}
\hline
\(\varepsilon^2\) & \(L\) & Budgets \((T_\ell)_{\ell=0}^{L-1}\) & MLMC sample sizes \((N_\ell)_{\ell=0}^{L-1}\) & SLMC \(N\) \\
\hline
\(1.0\) & \(2\)  & \((250,500)\) & \((4,2)\) & \(6\) \\
\(0.5\) & \(4\)  & \((125,250,375,500)\) & \((8,4,2,2)\) & \(16\) \\
\(0.1\) & \(20\) & \(T_\ell=25(\ell+1)\) & \makecell{\((40,20,13,10,8,6,5,5,4,4,\)\\ \(3,3,3,2,2,2,2,2,2,2)\)} & \(138\) \\
\hline
\end{tabular}
\caption{Level hierarchy in the RGB FoE experiment: lower-level iteration budgets
\(T_\ell\) (computational fidelity) and MLMC sample sizes \(N_\ell\) per upper-level step.}
\label{tab:foe-level-budgets}
\end{table}

For computational efficiency, the implementation uses a practical batched variant of the estimator analyzed in Section~\ref{sec:method}. At each upper-level step, one independent random direction is sampled per level and shared by all $N_\ell$ images in the corresponding level-wise batch. This introduces dependence between the per-image finite-difference samples and therefore the \(N_\ell^{-1}\) variance bound used in the analysis of Section~\ref{sec:method} does not directly apply to this batched implementation. Further, the direction is obtained as $u=g/\|g\|$, $g\sim\mathcal N(0,I_d)$, and hence is uniformly distributed on the unit sphere. Consequently, the implementation corresponds to a spherical-direction finite-difference estimator (equivalently, to the gradient of a ball-smoothed objective up to the constant factor $d^{-1}$), rather than exactly to the Gaussian estimator of Section~\ref{sec:method}; this constant factor is absorbed into the upper-level stepsize. 
Within each multilevel correction, the coarse and fine evaluations use the same direction and the same image batch. All lower-level solves are initialized with $x_0=A^\dagger D$, without warm starts, and use the prescribed iteration budget $T_\ell$. The resulting multilevel estimate is used directly in a plain SGD update with the cosine stepsize schedule described below. Accordingly, these optimization experiments should be viewed as an empirical evaluation of a practical MLMC variant rather than as a direct numerical verification of the theoretical estimator.


Computational cost is reported as the cumulative number of lower-level PR-SGD iterations. The remaining algorithmic parameters were fixed across all runs. The
upper-level DFO method used 
the smoothing parameter \(\eta=10^{-1}\). As upper-level solver, we use SGD with initial stepsize \(10^{-2}\), using a cosine decay schedule
$$
\gamma_k = \gamma_{\min} + \frac{1}{2} (\gamma_{\max} - \gamma_{\min}) \left( 1 + \cos\left( \frac{\pi k}{K} \right) \right),
$$
with $\gamma_{\max/\min}$ being the maximum/minimum learning rate. 
For the convolution filters and FoE smoothing variables, we use one tenth of this stepsize.
We note that the PR-SGD experiments use the practical stepsize schedule $\alpha_k=10^{-1}(k+1)^{-1/2}$, chosen for numerical stability over the finite iteration budgets considered here, whereas the theory in Section~\ref{subsubsec:pr-sgd} is derived for $\alpha_k\asymp(k+1)^{-1}$. Accordingly, the numerical results are intended as a qualitative illustration of the multilevel behavior rather than as a direct verification of the exact asymptotic rates.

We use \(M=8\) experts, and the FoE parameters were initialized with
\(\omega_m=1\), \(\nu_m=10^{-3}\), and \(\tau=\exp(-1)\). Moreover, \[
\begin{aligned}
y
&=
\big(
    \widetilde K_1,\ldots,\widetilde K_M,
    a_1,\ldots,a_M,
    b_1,\ldots,b_M,
    c
\big),\\
K_m
&=
\Pi_0(\widetilde K_m),\qquad
\omega_m=e^{a_m},\qquad
\nu_m=e^{b_m},\qquad
\tau=e^c.
\end{aligned}
\]
Here \(\Pi_0\) imposes the zero-mean filter constraint separately in each
input colour channel:
\[
    \big(\Pi_0(\widetilde K_m)\big)_{c,i,j}
    =
    (\widetilde K_m)_{c,i,j}
    -
    \frac{1}{k^2}
    \sum_{r=1}^{k}\sum_{s=1}^{k}
    (\widetilde K_m)_{c,r,s},
    \qquad k=5.
\]
Thus, the positive expert weights, smoothing parameters, and global scale
are optimized in logarithmic coordinates, while the filters are optimized
through the zero-mean parametrization.

\begin{table}[h!]
\centering
\begin{tabular}{c c c c c c}
\hline
\(\varepsilon^2\) & Estimator & Cost \((\times 10^6)\) & Final PSNR & Final SSIM & Best PSNR / SSIM \\
\hline
1.0 & MLMC & 1.00 & 27.42 & 0.788 & 27.43 / 0.795 \\
1.0 & SLMC & 1.20 & 26.42 & 0.780 & 26.42 / 0.780 \\
0.5 & MLMC & 2.20 & 26.71 & 0.796 & 26.71 / 0.796 \\
0.5 & SLMC & 3.20 & 26.78 & 0.771 & 26.79 / 0.775 \\
0.1 & MLMC & 5.97 & 26.48 & 0.791 & 26.48 / 0.791 \\
0.1 & SLMC & 8.28 & 25.29 & 0.739 & 25.30 / 0.739 \\
\hline
\end{tabular}
\caption{Validation performance for RGB denoising with the learned FoE
regularizer. SLMC denotes the single-level Monte Carlo estimator. Cost is
measured in cumulative lower-level PR-SGD iterations.}
\label{tab:color_foe_results}
\end{table}

\begin{figure}[h!]
\centering
\includegraphics[width=0.9\textwidth]{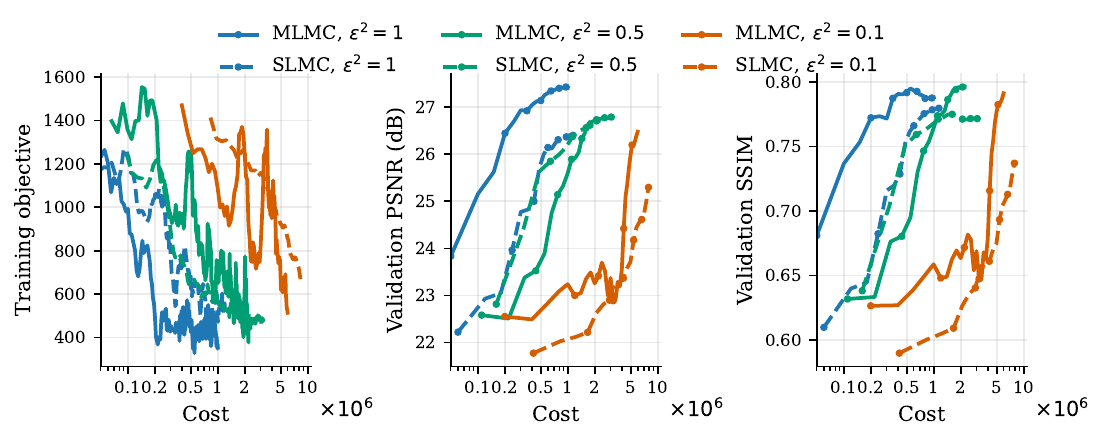}
\caption{Training objective, validation PSNR, and validation SSIM for RGB FoE
denoising. The horizontal axis shows cumulative lower-level work on a
logarithmic scale.}
\label{fig:color_foe_metrics}
\end{figure}

\begin{figure}[h!]
\centering
\includegraphics[width=0.9\textwidth]{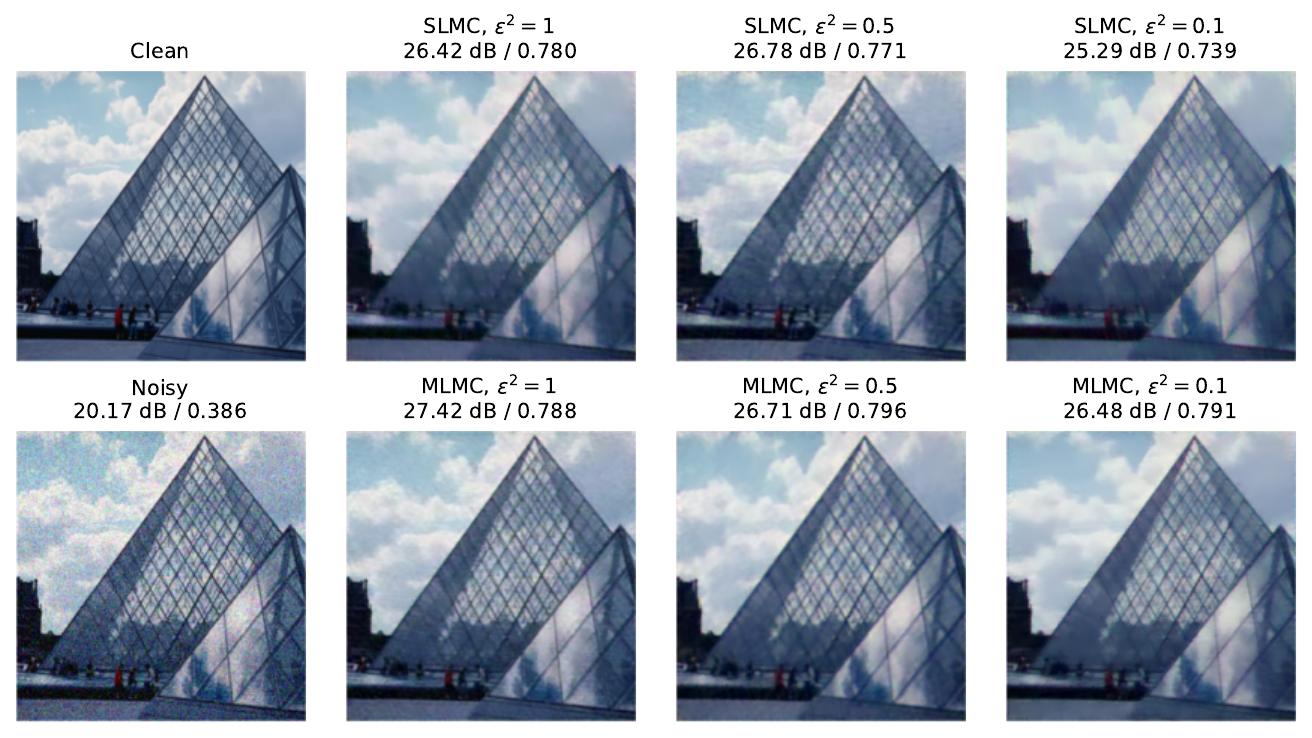}
\caption{Representative RGB denoising reconstructions. The first column shows
the clean and noisy images. For each target accuracy \(\varepsilon^2\), the top
row shows the single-level reconstruction and the bottom row shows the
corresponding MLMC reconstruction.}
\label{fig:color_foe_reconstructions}
\end{figure}

The results are summarized in Table~\ref{tab:color_foe_results} and
Figures~\ref{fig:color_foe_metrics}--\ref{fig:color_foe_reconstructions}. The
multilevel estimator is particularly effective in the color FoE setting,
where the cost of each upper-level query is dominated by repeated lower-level
PR-SGD reconstructions. For \(\varepsilon^2=1\), MLMC improves the final validation
PSNR by about \(1.0\) dB while using less lower-level work than SLMC. For
\(\varepsilon^2=0.5\), both methods obtain comparable PSNR, but MLMC reaches a
higher SSIM at approximately \(69\%\) of the single-level cost. For the most
accurate setting, \(\varepsilon^2=0.1\), MLMC gives the clearest gain: it
improves the final validation PSNR by about \(1.19\) dB and the SSIM by about
\(0.052\), while reducing the cumulative lower-level cost.
This experiment highlights the main practical role of MLMC in large-scale bilevel learning. The learned FoE regularizer is expressive enough to preserve color edges and repeated geometric structures, but it makes each derivative-free upper-level step expensive.
\subsubsection{Empirical illustration of complexity rates}
\label{subsubsec:numerical-complexity-rates}

We next provide a numerical check of the multilevel complexity behavior predicted by Section~\ref{sec:solver-complexity}, using the color image denoising problem with a learned FoE regularizer. The purpose of this experiment is not to estimate asymptotic constants, but to verify that the observed cost trends are consistent with the theoretical comparison between single-level and multilevel sampling. For the denoising problem, we compare the computational work required by SLMC and MLMC to reach prescribed estimator accuracies.  

Our experiments are based on the Gaussian smoothing estimator of Section \ref{sec:method}. We consider targeting different MSEs,  which are specified as $\{1.6, 0.4, 0.15, 0.02\}\times 10^{-2}$. The number of levels $L$ and the per-level sample sizes $N_\ell$ are determined by the theoretical allocation of Table \ref{table:comp}, with the constants calibrated so that the predicted MSE matches each target. The reference gradient $\nabla\Psi_\eta(y)$ was approximated by an independent high-accuracy estimator using $N_{\mathrm{ref}} =10^5 $ samples at the finest level $L_{\mathrm{ref}} = 20$.  The reported MSE at each accuracy point is the average over $S = 25$ independent seeds.  For each solver, the level hierarchy used the iteration budgets $T_\ell =  25(\ell+1)$. For the stepsize of the LLP we choose $\alpha_k = 10^{-1}(k+1)^{-1}$ for the PR-SGD and FI-SGD.

The results are reported in Figure~\ref{fig:rates}. In all tested regimes, MLMC requires less lower-level work than the single-level estimator for comparable accuracy. The empirical curves follow the qualitative behavior predicted by the complexity bounds in Section~\ref{sec:solver-complexity}. 

\begin{figure}[h!]
\centering
\includegraphics[width=0.35\textwidth]{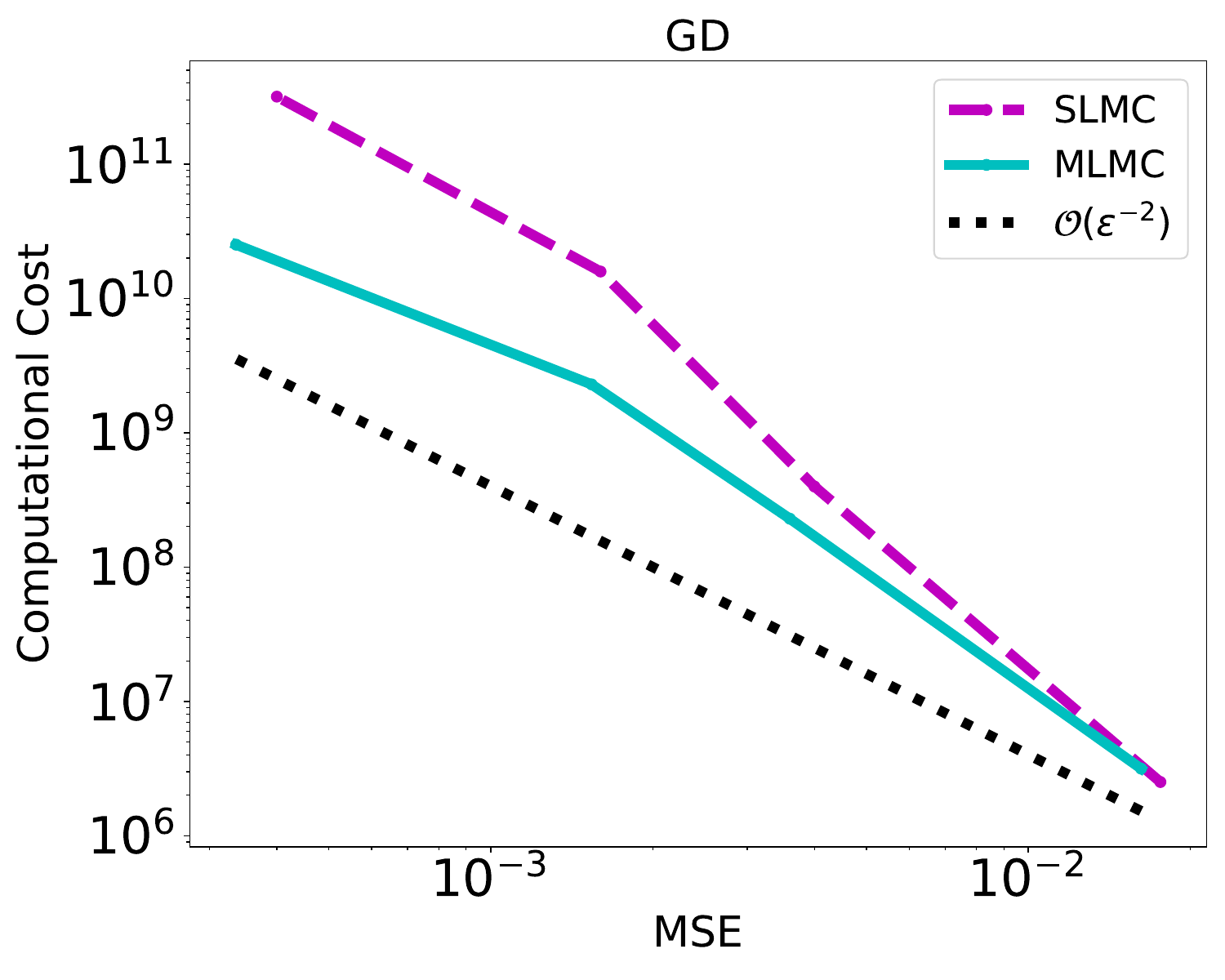} 
\includegraphics[width=0.35\textwidth]{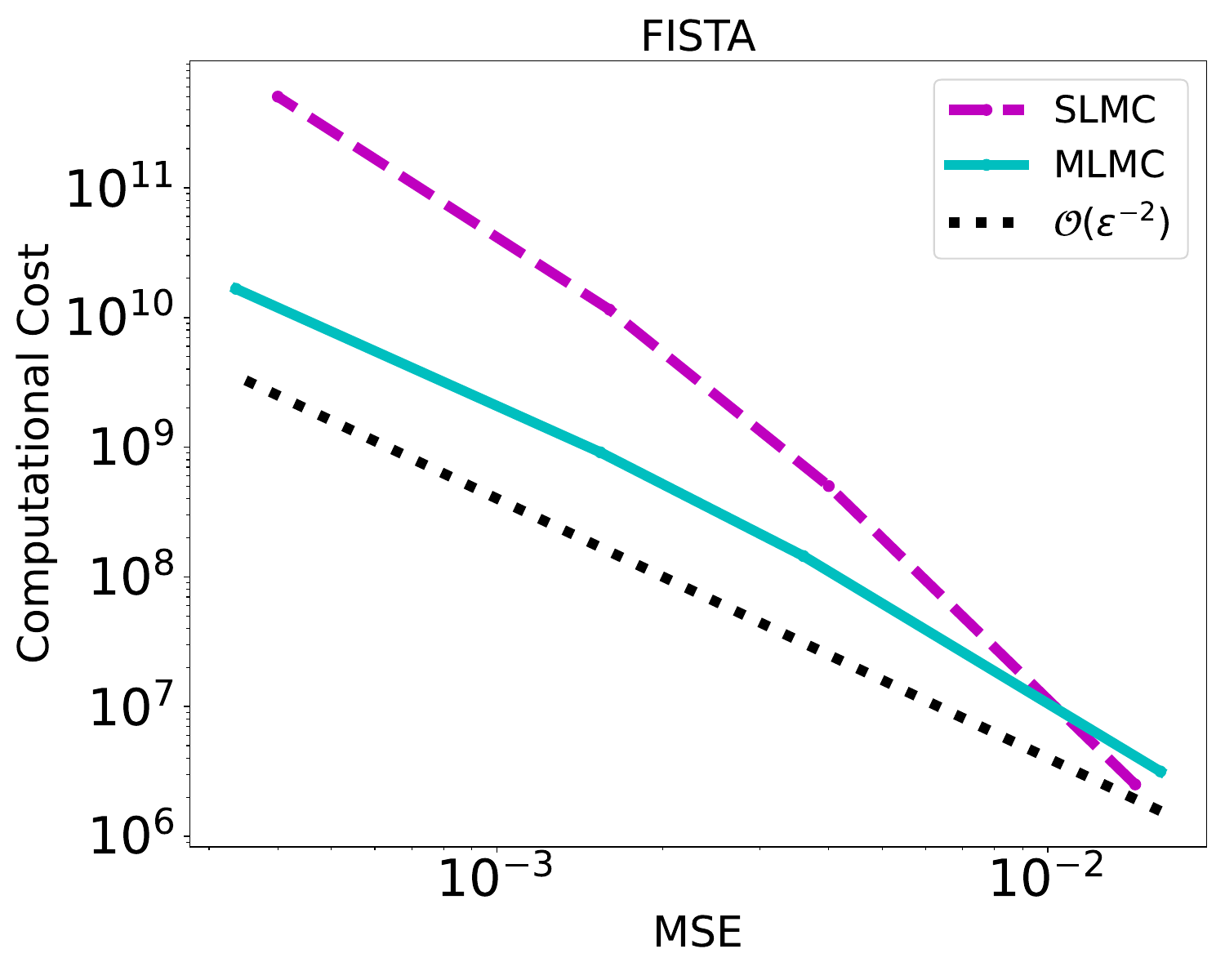}  \\
\includegraphics[width=0.35\textwidth]{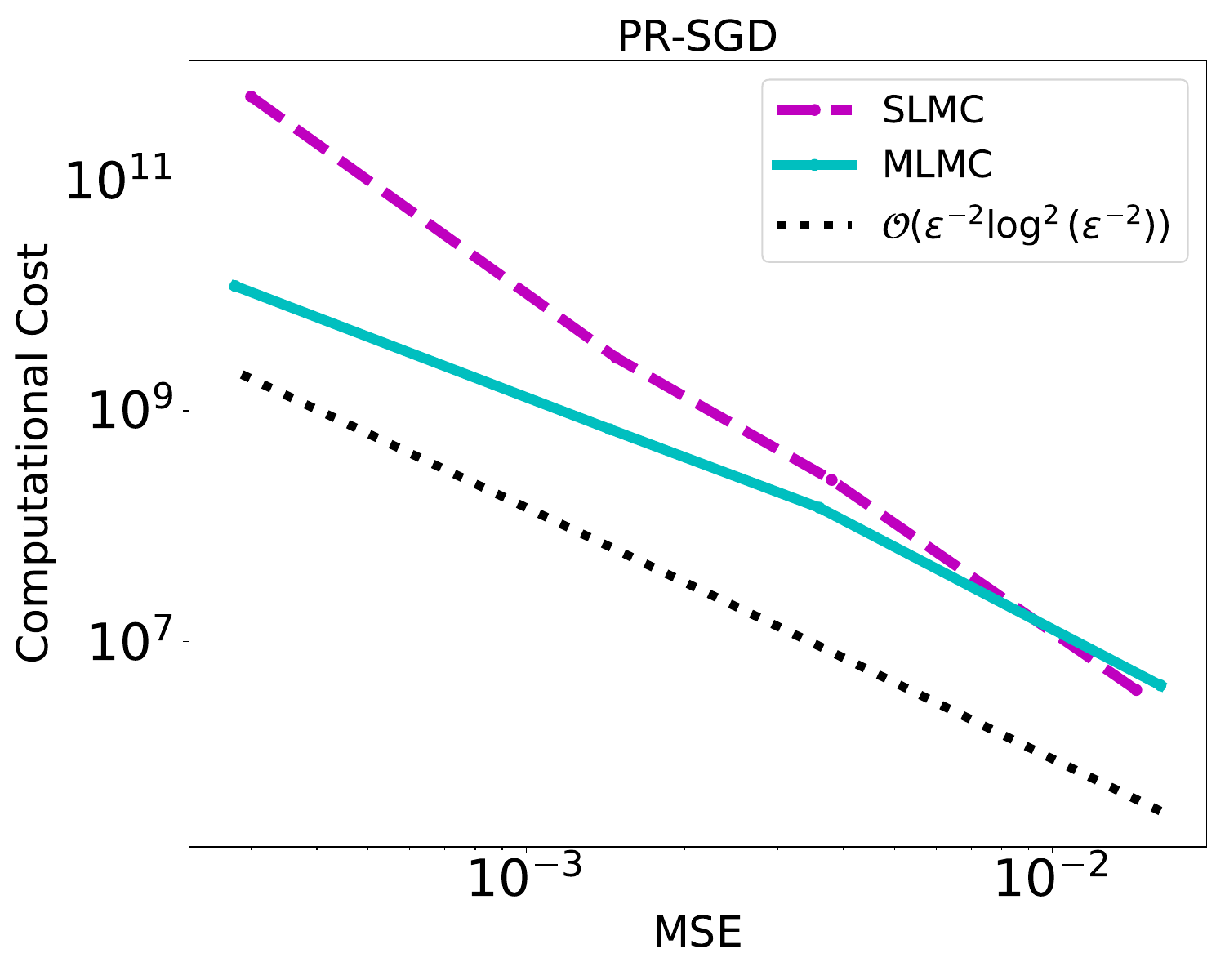} 
\includegraphics[width=0.35\textwidth]{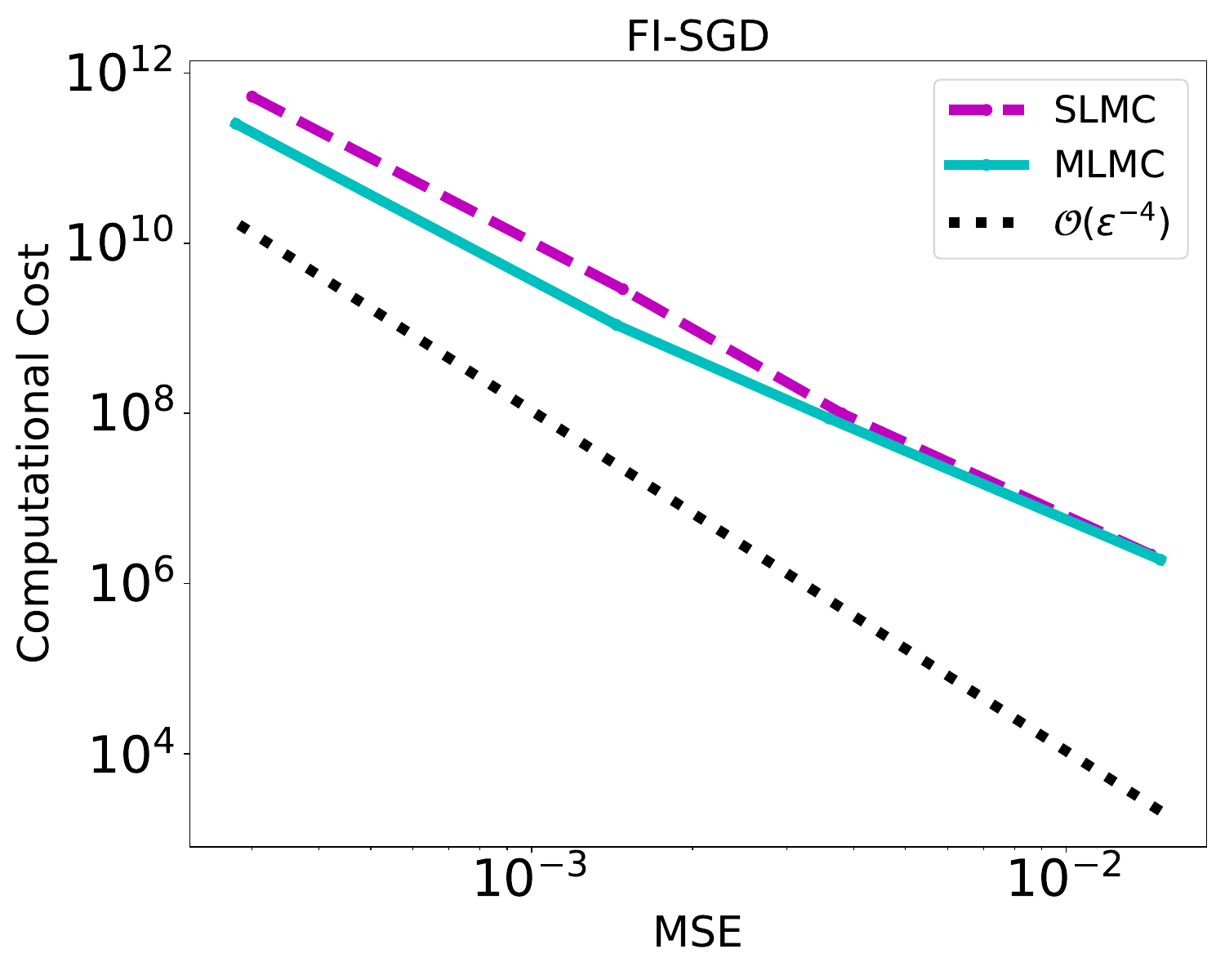} 
\caption{
Complexity rates for achieving a prescribed MSE of the smoothed gradient estimator, comparing SLMC and MLMC across various lower-level solvers. Top left: GD; Top right: FISTA; Bottom left: PR-SGD; Bottom right: FI-SGD. For FI-SGD and PR-SGD, the reported rates correspond to the consecutive hierarchies analyzed in Section~\ref{subsec:sublinear-solvers}.}
\label{fig:rates}
\end{figure}

\subsection{Image reconstruction based on the Radon transform}
\label{subsec:ct-experiment}
Our second numerical example concerns sparse-view computed tomography using the Radon transform. The goal is to recover an image from noisy line-integral measurements. The forward model is
\[
    D = AX + W,
    \qquad W\sim \mathcal N(0,\sigma^2 I),
\]
where \(A:\mathcal X\to\mathbb R^m\) denotes the discrete parallel-beam Radon transform and \(X\in\mathcal X\) is the unknown image. The projection angles are chosen uniformly in \([0,\pi)\). For this experiment, we learn the parameters of a smoothed total variation regularizer with an additional Tikhonov term. The upper-level problem is again
\[
    \min_y
    \mathbb E
    \left[
        \frac12\|h(y,D)-X\|^2
    \right],
\]
where \(h(y,D)\) is defined by the lower-level reconstruction problem
\[
    h(y,D) =\argmin_{x\in\mathcal X}
    \big\{
        \frac12\|Ax-D\|^2
        +
        \lambda\|x\|^2
        +
        \alpha
        \sum_{i,j}
        \sqrt{
            |(\nabla_1 x)_{i,j}|^2
            +
            |(\nabla_2 x)_{i,j}|^2
            +
            \nu^2
        }
    \big\}.
\]

Here \(\nabla_1\) and \(\nabla_2\) denote horizontal and vertical finite differences. The physical regularization parameters are $\lambda,\alpha,\nu>0$, where \(\lambda\) controls the quadratic stabilization, \(\alpha\) controls the
strength of the TV penalty, and \(\nu\) is the smoothing parameter in the differentiable TV approximation. The learned parameter is
$y=(y_1,y_2,y_3)^\top$, with
$\lambda=10^{y_1}$, $\alpha=10^{y_2}$, and $\nu=10^{y_3}$. In the implementation, positivity is enforced through a componentwise parameter transformation. The lower-level problem is solved with FISTA. In the reported CT experiments, we use a restarted FISTA variant, which is more stable for the ill-conditioned Radon operator. 

We use images from the LoDoPaB-CT dataset~\cite{leuschner2021lodopab}, resized to $128\times128$ pixels. Measurements are simulated using $60$ projection angles uniformly spaced in $[0,\pi)$ and additive Gaussian noise with standard deviation $\sigma=0.015$. The parameters are initialized at $y=(-2,-5,-2)$, corresponding to $(\lambda,\alpha,\nu)=(10^{-2},10^{-5},10^{-2})$. The upper-level method uses $\eta=10^{-1}$ and SGD with initial stepsize $10^{-3}$ ($10^{-4}$ for the smoothing parameter), followed by the same cosine schedule as in Section~\ref{subsec:foe-denoising}. The number of levels is $L=\max\{3,\lceil\sqrt{2}\,\varepsilon^{-1}\rceil\}$. The FISTA budgets are geometrically spaced between $50$ and $1000$ iterations, while the level-wise image-batch sizes are $N_0=\lceil3\varepsilon^{-2}\rceil$ and $N_\ell=\lceil4\varepsilon^{-2}(\ell+1)^{-1}\rceil$ for $\ell\ge1$. The single-level baseline uses $\sum_\ell N_\ell$ images at the finest budget. The implementation uses the same batched spherical-direction variant described in Section~\ref{subsec:foe-denoising}.

\begin{table}[h!]
\centering
\small
\begin{tabular}{c c c c c}
\hline
\(\varepsilon^2\) & \(L\) & Budgets \((T_\ell)_{\ell=0}^{L-1}\) & MLMC sample sizes \((N_\ell)_{\ell=0}^{L-1}\) & SLMC \(N\) \\
\hline
\(1.0\) & \(3\) & \((50,224,1000)\) & \((3,2,2)\) & \(7\) \\
\(0.5\) & \(3\) & \((50,224,1000)\) & \((6,4,3)\) & \(13\) \\
\(0.1\) & \(5\) & \((50,106,224,473,1000)\) & \((30,20,14,10,8)\) & \(82\) \\
\hline
\end{tabular}
\caption{Level hierarchy in the CT experiment: FISTA iteration budgets \(T_\ell\) and
MLMC sample sizes \(N_\ell\) per upper-level step.}
\label{tab:ct-level-budgets}
\end{table}

\begin{table}[h!]
\centering
\begin{tabular}{c c c c c c}
\hline
$\varepsilon^2$ & Estimator & Cost $(\times 10^6)$ & Final PSNR & Final SSIM & Best PSNR / SSIM \\
\hline
1.0 & MLMC & 1.26 & 30.79 & 0.798 & 30.80 / 0.805 \\
1.0 & SLMC & 2.80 & 30.90 & 0.803 & 30.90 / 0.803 \\
0.5 & MLMC & 2.03 & 31.02 & 0.803 & 31.02 / 0.807 \\
0.5 & SLMC & 5.20 & 30.91 & 0.800 & 30.91 / 0.800 \\
0.1 & MLMC & 9.52 & 30.49 & 0.798 & 30.49 / 0.798 \\
0.1 & SLMC & 9.84 & 29.78 & 0.748 & 29.78 / 0.748 \\
\hline
\end{tabular}
\caption{Validation performance for the LoDoPaB-CT reconstruction experiment.
SLMC denotes the single-level Monte Carlo estimator. Cost is measured in
cumulative lower-level FISTA iterations.}
\label{tab:ct_tv_results}
\end{table}

\begin{figure}[h!]
\centering
\includegraphics[width=0.9\textwidth]{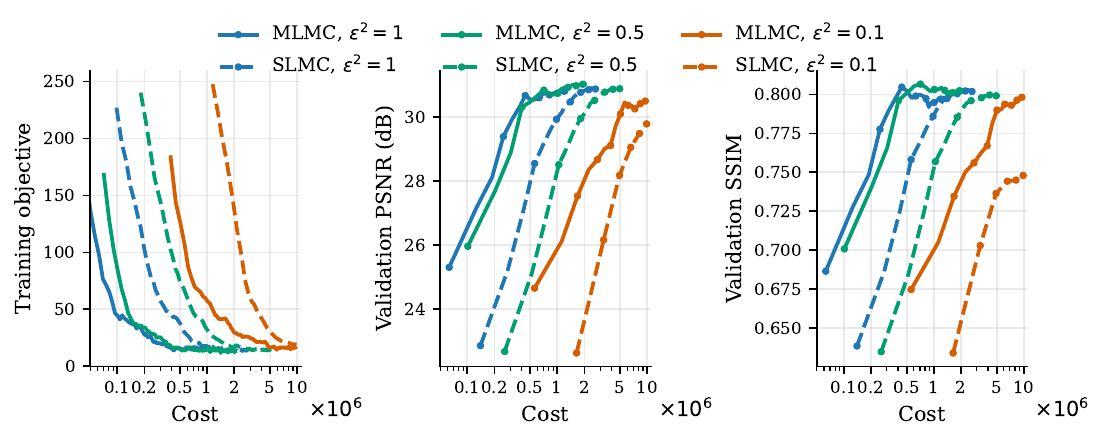}
\caption{Training objective, validation PSNR, and validation SSIM for the
CT-TV experiment. The horizontal axis shows cumulative lower-level work on a
logarithmic scale.}
\label{fig:ct_tv_metrics}
\end{figure}

\begin{figure}[h!]
\centering
\includegraphics[width=0.9\textwidth]{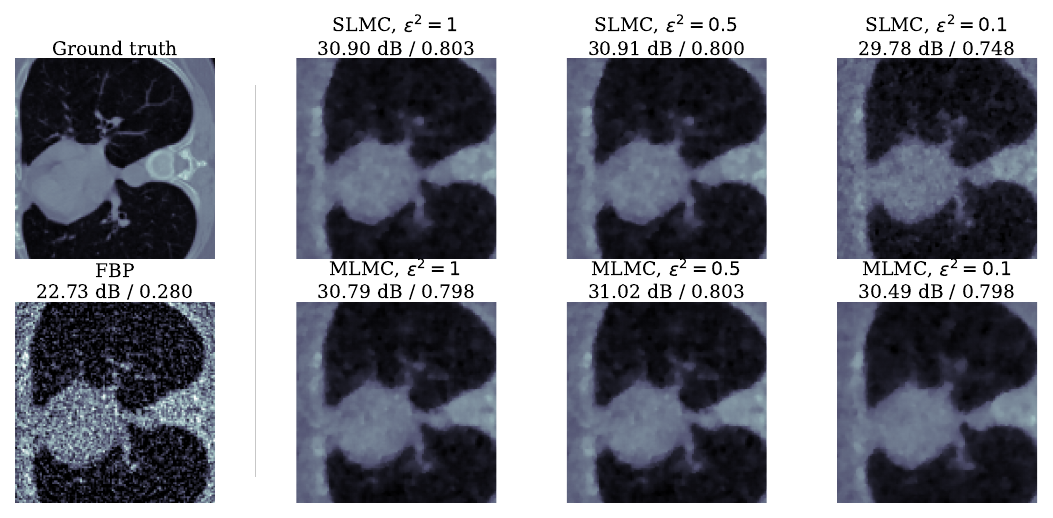}
\caption{Representative CT reconstructions. The reference image and filtered
backprojection (FBP) reconstruction are shown separately from the learned
TV reconstructions. For each accuracy level $\varepsilon^2$, we compare the
single-level estimator with the MLMC estimator using the learned
regularization parameters.}
\label{fig:ct_tv_reconstructions}
\end{figure}

The results are shown in Table~\ref{tab:ct_tv_results} and
Figures~\ref{fig:ct_tv_metrics}--\ref{fig:ct_tv_reconstructions}. The
filtered backprojection reconstruction contains the characteristic streaking
and noise amplification caused by sparse and noisy projection data. The learned
TV reconstructions remove most of these artifacts while preserving the main
anatomical structures, confirming that the bilevel procedure learns
meaningful regularization parameters for the CT inverse problem.

The cost-normalized curves in Figure~\ref{fig:ct_tv_metrics} are consistent
with the predicted multilevel advantage. For $\varepsilon^2=1$, MLMC reaches a
validation quality comparable to SLMC using $1.26\times 10^6$ lower-level
iterations, compared with $2.80\times 10^6$ for SLMC. For
$\varepsilon^2=0.5$, the advantage is stronger: MLMC obtains the best final
validation result, $31.02$ dB PSNR and $0.803$ SSIM, at approximately
$39\%$ of the single-level cost. This is the best trade-off in the reported
CT runs. For the strictest accuracy, $\varepsilon^2=0.1$, the cost difference
is smaller because both estimators spend most of their work at or near the
finest lower-level solve. Nevertheless, MLMC achieves higher final validation quality in this regime, improving the final validation PSNR by about $0.71$ dB and the
SSIM by about $0.05$ compared to SLMC.

\section{Conclusion}
\label{sec:conclusion}

We introduced a multilevel derivative-free method for stochastic bilevel optimization problems arising in inverse problems. The method combines Gaussian smoothing of the upper-level objective with a multilevel Monte Carlo gradient estimator built from approximate lower-level solves. This construction allows one to exploit hierarchies of lower-level accuracies while avoiding derivatives of the solution map with respect to the hyperparameters. The convergence analysis separates the outer derivative-free optimization from the lower-level approximation properties. Under a controlled mean-square error condition on the gradient estimator, we obtained stationarity guarantees for
the smoothed upper-level objective. The multilevel estimator was then analyzed in terms of three solver-dependent quantities: the lower-level bias, the mean squared errors of coupled level differences, and the cost per level. 

A main conclusion is that the benefit of multilevel sampling is determined not only by the accuracy of the lower-level solver, but by the decay of the coupled level differences. Linearly convergent deterministic solvers lead to the canonical multilevel improvement over single-level sampling. For stochastic
lower-level solvers, final-iterate SGD and Polyak--Ruppert averaged SGD have the same mean-square bias order, but different level-difference behavior. Consequently, for the consecutive hierarchy considered in this work, final-iterate SGD does not improve the asymptotic complexity exponent, whereas Polyak--Ruppert averaging yields a genuine multilevel speedup. The numerical experiments on image denoising and computed tomography support the practical relevance of the proposed estimator. In both settings, the multilevel method reduces the lower-level work required to obtain comparable upper-level performance, while preserving the quality of the learned regularization parameters.

Several extensions remain open. First, the analysis could be refined to cover adaptive choices of the smoothing parameter, sample sizes, lower-level accuracies and level schedules during the outer optimization. Second, the relation between the smoothed stationarity measure used here and non-smooth stationarity notions for the original upper-level objective deserves further study. Extending the framework to richer stochastic lower-level solvers, alternative derivative-free upper-level solvers, and to more complex learned regularization models would broaden its applicability to large-scale inverse problems. 

\section*{Acknowledgments}
NKC is supported by a City University of Hong Kong Start-up Grant 7200809, a Hong Kong ECS-RGC Grant 21300826 and a Hong Kong GRF-RGC Grant No. 11309626.
\appendix
\section*{Appendix}
\section{Additional details}
 \subsection{Uniformity in the upper-level parameter}
\label{sec:parameter-transformation}
The definitions of \(b_\ell\) and \(v_{\ell+1}\) involve a supremum over \(y\in\mathbb R^d\). At first sight, this appears restrictive because the Gaussian smoothing estimator evaluates the lower-level problem at perturbed points \(y+\eta U\), which are unbounded. In our implementation for the inverse-problem applications, however, the variable \(y\) used by the outer optimization method is typically not the physical regularization parameter itself. Instead, the physical hyperparameter is obtained through a smooth transformation into an admissible set. We therefore distinguish between the optimization variable
\(y\in\mathbb R^d\) and the physical lower-level hyperparameter $p=T(y)\in\mathcal P$, where \(\mathcal P\subset\mathbb R^m\) is a compact set of admissible hyperparameters and $T:\mathbb R^d\to\operatorname{int}(\mathcal P)$ is a continuously differentiable, globally Lipschitz map. The lower-level
regularizer depends on \(y\) only through \(T(y)\), that is, $\mathcal S_y := \mathcal S_{T(y)}$. Consequently, the lower-level problem takes the form
\begin{equation}\label{eq:transformed-llp}
h(y,D)\in\argmin_{x\in\mathcal X}\big\{\mathcal L(Ax,D)+\mathcal S_{T(y)}(x)\big\}\,.  
\end{equation}
Since \(T(\mathbb R^d)\subset\mathcal P\), all lower-level solves are performed with physical hyperparameters in the compact set \(\mathcal P\), including at Gaussian-perturbed points \(y+\eta U\). Thus, the uniform bounds in \(b_\ell\) and \(v_{\ell+1}\) should be understood as uniformity over the induced compact physical hyperparameter set \(\mathcal P\). More precisely, the constants in the lower-level convergence estimates are assumed to be uniform for all \(p\in\mathcal P\), and hence for all \(y\in\mathbb R^d\) through \(p=T(y)\).

Suppose that \(g\) satisfies Assumption~\ref{ass:lipschitz-loss} and that the lower-level solution map is Lipschitz continuous with respect to the physical hyperparameter \(p\) uniformly on \(\mathcal P\), i.e.,
\[
    \|h_p(D)-h_{p'}(D)\|_{\mathcal X} \le \Lip_{0}(h,D)\|p-p'\|, \qquad p,p'\in\mathcal P,
\]
with $\mathbb E[\Lip_{0}(g,X)\Lip_{0}(h,D)]<\infty$. If \(T:\mathbb R^d\to\mathcal P\) is globally Lipschitz with constant \(\Lip_0(T)\),
then the reduced solution map \(y\mapsto h(y,D):=h_{T(y)}(D)\) satisfies
\[\|h(y_1,D)-h(y_2,D)\|_{\mathcal X}\le \Lip_{0}(h,D)\Lip_0(T)\|y_1-y_2\|.\]
Consequently,
\[|\Psi(y_1)-\Psi(y_2)|\le \mathbb E\big[\Lip_{0}(g,X)\|h(y_1,D)-h(y_2,D)\|_{\mathcal X}\big] \le \Lip_0(T)\,\mathbb E[\Lip_{0}(g,X)\Lip_{0}(h,D)]\,\|y_1-y_2\|\]
and Assumption~\ref{ass:lipschitz-upper} follows from the stability of the lower-level solution map on the compact physical hyperparameter set \(\mathcal P\). Analogously, the same argument applies to the approximate solution maps $\widehat h_\ell$, provided their corresponding stability constants are uniformly bounded in $\ell$.

For the theoretical analysis we may assume a bounded globally Lipschitz parameterization \(T:\mathbb R^d\to\operatorname{int}(\mathcal P)\), where
\(\mathcal P\) is compact. An exemplary component-wise choice is
\[T_j(y)=\lambda_{\min,j}+(\lambda_{\max,j}-\lambda_{\min,j})\frac{1}{1+\exp(-y_j)}\]
embedding $\R^{d}$ in the hypercube $\prod_{j=1}^{d}[\lambda_{\min,j},\lambda_{\max,j}]$. In our implementations where only positivity is enforced, one often uses maps such as \(T_j(y)=10^{y_j}+\lambda_{\min,j}\). This map is locally, but not globally, Lipschitz on \(\mathbb R\). Hence, the parametrization is not covered by the preceding global-Lipschitz sufficient condition. Its theoretical treatment would require additional localization/moment assumptions, as in Section~\ref{sec:prox-extensions}, or a bounded transformation.

\subsection{Omitted proofs for Section~\ref{sec:method}}
\label{app:method-proofs}
\begin{proof}[Proof of Lemma~\ref{lem:gradient-mse}]
We prove the single-level and multilevel estimates separately.

\emph{i) Single-level estimator.}
Fix \(y\in\mathcal Y\), a level \(\ell\ge0\), and sample size \(N\ge1\). Define
\[
    Q_\ell(y;U,Z)
    :=
    \frac{
        G_\ell(y+\eta U;Z)-G_\ell(y;Z)
    }{\eta}U
\]
and
\[
    Q_\infty(y;U,Z)
    :=
    \frac{
        G_\infty(y+\eta U;Z)-G_\infty(y;Z)
    }{\eta}U .
\]
Then
\[
    \GhatSL_{\eta,\ell,N}(y)
    =
    \frac1N\sum_{i=1}^N Q_\ell(y;U_i,Z_i)
    \qquad \text{and}\qquad 
    \nabla\Psi_\eta(y)
    =
    \mathbb E[Q_\infty(y;U,Z)].
\]
We use the decomposition
\[\GhatSL_{\eta,\ell,N}(y)-\nabla\Psi_\eta(y)= (\GhatSL_{\eta,\ell,N}(y)-\mathbb E[Q_\ell(y;U,Z)])+(\mathbb E[Q_\ell(y;U,Z)]-\mathbb E[Q_\infty(y;U,Z)]).
\]
to obtain
\[    \mathbb E
    \Big[
        \big\|
            \GhatSL_{\eta,\ell,N}(y)-\nabla\Psi_\eta(y)
        \big\|^2
    \Big]\le
    2\mathbb E
    \Big[
        \big\|
            \GhatSL_{\eta,\ell,N}(y)
            -
            \mathbb E[Q_\ell(y;U,Z)]
        \big\|^2
    \Big]+
    2\big\|
        \mathbb E[Q_\ell(y;U,Z)-Q_\infty(y;U,Z)]
    \big\|^2 .
\]
For the first term, independence of the samples gives
\[
    \mathbb E
    \Big[
        \big\|
            \GhatSL_{\eta,\ell,N}(y)
            -
            \mathbb E[Q_\ell(y;U,Z)]
        \big\|^2
    \Big]
    =
    \frac1N
    \mathbb E
    \Big[
        \big\|
            Q_\ell(y;U,Z)-\mathbb E[Q_\ell(y;U,Z)]
        \big\|^2
    \Big]\le
    \frac1N
    \mathbb E
    \left[
        \|Q_\ell(y;U,Z)\|^2
    \right]
    \le
    \frac{C_{\mathrm{fd}}}{N}.
\]
where we have used the finite-difference moment bound in Assumption~\ref{ass:fd-moment-bounds}. It remains to control the bias term. By definition,
\[
    Q_\ell(y;U,Z)-Q_\infty(y;U,Z)
    =
    \frac{1}{\eta}
    \big[G_\ell(y+\eta U;Z)-G_\infty(y+\eta U;Z)
        -
        G_\ell(y;Z)+G_\infty(y;Z)
    \big]U .
\]
Using the Lipschitz property of \(g\) in its first argument,
\[|G_\ell(y+\eta U;Z)-G_\infty(y+\eta U;Z)|\le \Lip_{0}(g,X) \|\widehat h_\ell(y+\eta U,D)-h(y+\eta U,D)\|_{\mathcal X},
\]
and similarly
\[
    |G_\ell(y;Z)-G_\infty(y;Z)|
    \le
    \Lip_{0}(g,X)
    \|\widehat h_\ell(y,D)-h(y,D)\|_{\mathcal X}.
\]
Therefore,
\[
    \|Q_\ell(y;U,Z)-Q_\infty(y;U,Z)\|
    \le
    \frac{\Lip_{0}(g,X)\|U\|}{\eta}
    \big(\|\widehat h_\ell(y+\eta U,D)-h(y+\eta U,D)\|_{\mathcal X}+
        \|\widehat h_\ell(y,D)-h(y,D)\|_{\mathcal X}
    \big).
\]
By Cauchy--Schwarz, we have
\begin{align*} 
&\textcolor{black}{\mathbb E[\|U\|^2 \Lip_{0}(g,X)^2 \|\widehat h_{\ell}(y+\eta U,D)
          - h(y+\eta U,D)\|_{\mathcal X}^2]}\\
          &\textcolor{black}{=\mathbb E[\|U\|^2 \mathbb E[\Lip_{0}(g,X)^2 \|\widehat h_{\ell}(y+\eta U,D)
          -h(y+\eta U,D)\|_{\mathcal X}^2 \mid U]] \le \mathbb E[\|U\|^2] b_{\ell}}\,,
\end{align*}
where we have used $\mathbb E[\Lip_{0}(g,X)^2 \|\widehat h_{\ell}(y+\eta u,D) - h(y+\eta u,D)\|_{\mathcal X}^2 \mid U=u] \le b_{\ell}$ for all $u\in \R^d$. For the second term, we obtain a similar bound using the independence of $U$ and $(X,D)$. By Jensen's inequality, 
\[
    \left\|
        \mathbb E[Q_\ell(y;U,Z)-Q_\infty(y;U,Z)]
    \right\|^2
    \le
    \frac{\textcolor{black}{2d}}{\eta^2}b_\ell
\]
Combining the variance and bias estimates yields
\[
    \mathbb E
    \Big[
        \big\|
            \GhatSL_{\eta,\ell,N}(y)-\nabla\Psi_\eta(y)
        \big\|^2
    \Big]
    \le
    C_{d,\eta}
    \big(
        \frac1N+b_\ell
    \big)
\]
for some constant $C_{d,\eta}>0$ possibly depending on dimension and $\eta$ through 
\[\textcolor{black}{C_{d,\eta}\propto C_{\rm fd} + d\eta^{-2}}\,.\]

\emph{ii) Multilevel estimator.}
We now prove the multilevel estimate. Define
\[
    Q_0(y;U,Z)
    :=
    \frac{
        G_0(y+\eta U;Z)-G_0(y;Z)
    }{\eta}U
\]
and, for \(\ell=0,\ldots,L-1\),
\[
    \Delta Q_{\ell+1}(y;U,Z)
    :=
    \Bigg[
    \frac{
        G_{\ell+1}(y+\eta U;Z)
        -
        G_\ell(y+\eta U;Z)
    }{\eta}
    -
    \frac{
        G_{\ell+1}(y;Z)
        -
        G_\ell(y;Z)
    }{\eta}
    \Bigg]U .
\]
Then
\[
    \GhatML_{\eta,L,\mathbf N}(y)
    =
    \frac1{N_0}\sum_{i=1}^{N_0}Q_0(y;U_i^{(0)},Z_i^{(0)})
    +
    \sum_{\ell=0}^{L-1}
    \frac1{N_{\ell+1}}
    \sum_{i=1}^{N_{\ell+1}}
    \Delta Q_{\ell+1}
    (y;U_i^{(\ell+1)},Z_i^{(\ell+1)}).
\]
By the telescoping identity,
\[
    \mathbb E[\GhatML_{\eta,L,\mathbf N}(y)]
    =
    \nabla\Psi_{\eta,L}(y).
\]
Hence
\[
\begin{aligned}
    \GhatML_{\eta,L,\mathbf N}(y)-\nabla\Psi_\eta(y)
    =
    \Big(
        \GhatML_{\eta,L,\mathbf N}(y)
        -
        \nabla\Psi_{\eta,L}(y)
    \Big)
    +
    \Big(
        \nabla\Psi_{\eta,L}(y)
        -
        \nabla\Psi_\eta(y)
    \Big)\,,
\end{aligned}
\]
and since $\mathbb E[\GhatML_{\eta,L,\mathbf N}(y)] = \nabla\Psi_{\eta,L}(y)$, we can write 
\[
\mathbb E\Big[\|\GhatML_{\eta,L,\mathbf N}(y)-\nabla\Psi_\eta(y)\|^2\Big] = \mathbb E\Big[\|\GhatML_{\eta,L,\mathbf N}(y)-\nabla\Psi_{\eta,L}(y)\|^2\Big] + \|\nabla\Psi_{\eta,L}(y)-\nabla\Psi_\eta(y)\|^2
\]
Hence, it suffices to bound the
sampling variance and the finest-level bias. Since the samples are independent across levels,
\[
    \mathbb E
    \Big[
        \|
            \GhatML_{\eta,L,\mathbf N}(y)
            -
            \nabla\Psi_{\eta,L}(y)
        \|^2
    \Big]\le
    \frac{1}{N_0}
    \mathbb E\|Q_0(y;U,Z)\|^2
    +
    \sum_{\ell=0}^{L-1}
    \frac{1}{N_{\ell+1}}
    \mathbb E\|\Delta Q_{\ell+1}(y;U,Z)\|^2.
\]
The level-zero term is bounded by Assumption~\ref{ass:fd-moment-bounds},
\[
    \frac{1}{N_0}\mathbb E\|Q_0(y;U,Z)\|^2
    \le
    \frac{C_{\rm fd}}{N_0}.
\]
For the correction term, the Lipschitz property of \(g\) gives
\[
    |G_{\ell+1}(y+\eta U;Z)-G_\ell(y+\eta U;Z)|\le
    \Lip_0(g,X)
    \|\widehat h_{\ell+1}(y+\eta U,D)
      -\widehat h_\ell(y+\eta U,D)\|_{\mathcal X},
\]
and
\[
    |G_{\ell+1}(y;Z)-G_\ell(y;Z)|
    \le
    \Lip_0(g,X)
    \|\widehat h_{\ell+1}(y,D)-\widehat h_\ell(y,D)\|_{\mathcal X}.
\]
Consequently,
\[
\begin{aligned}
    \|\Delta Q_{\ell+1}(y;U,Z)\|
    \le
    \frac{\Lip_{0}(g,X)\|U\|}{\eta}
    \Big(
        &\|\widehat h_{\ell+1}(y+\eta U,D)
          -\widehat h_\ell(y+\eta U,D)\|_{\mathcal X}
        \\
        &+
        \|\widehat h_{\ell+1}(y,D)
          -\widehat h_\ell(y,D)\|_{\mathcal X}
    \Big)\,.
\end{aligned}
\]
By Cauchy--Schwarz, we have
\begin{align*} 
&\textcolor{black}{\mathbb E[\|U\|^2 \Lip_{0}(g,X)^2 \|\widehat h_{\ell+1}(y+\eta U,D)
          -\widehat h_\ell(y+\eta U,D)\|_{\mathcal X}^2]}\\
          &\textcolor{black}{=\mathbb E[\|U\|^2 \mathbb E[\Lip_{0}(g,X)^2 \|\widehat h_{\ell+1}(y+\eta U,D)
          -\widehat h_\ell(y+\eta U,D)\|_{\mathcal X}^2 \mid U]] \le \mathbb E[\|U\|^2] v_{\ell+1}}\,,
\end{align*}
where we have used $\mathbb E[\Lip_{0}(g,X)^2 \|\widehat h_{\ell+1}(y+\eta u,D) -\widehat h_\ell(y+\eta u,D)\|_{\mathcal X}^2 \mid U=u] \le v_{\ell+1}$ for all $u\in \R^d$. For the second term, we use independence, to deduce 
\[ \textcolor{black}{\mathbb E[\|U\|^2 \Lip_{0}(g,X)^2 \|\widehat h_{\ell+1}(y,D)
          -\widehat h_\ell(y,D)\|_{\mathcal X}^2]} \le \mathbb E[\|U\|^2] v_{\ell+1}\]
Together, this implies
\[
    \mathbb E\|\Delta Q_{\ell+1}(y;U,Z)\|^2
    \le
    \frac{\textcolor{black}{2d}}{\eta^2}
    v_{\ell+1}\,.
\]
It follows that
\[
    \mathbb E
    \left[
        \left\|
            \GhatML_{\eta,L,\mathbf N}(y)
            -
            \nabla\Psi_{\eta,L}(y)
        \right\|^2
    \right]
    \le
    C_{d,\eta}
    \left(
        \frac1{N_0}
        +
        \sum_{\ell=0}^{L-1}
        \frac{v_{\ell+1}}{N_{\ell+1}}
    \right)\,,
\]
for some constant depending on dimension $d$ and $\eta$. It remains to bound the bias
\(\nabla\Psi_{\eta,L}(y)-\nabla\Psi_\eta(y)\). This is the same estimate as in
the single-level case with \(\ell=L\), and gives
\[
    \|\nabla\Psi_{\eta,L}(y)-\nabla\Psi_\eta(y)\|^2
    \le
    \frac{2d}{\eta^2}b_L.
\]
Combining the sampling and bias estimates yields
\[
    \mathbb E
    \left[
        \left\|
            \GhatML_{\eta,L,\mathbf N}(y)-\nabla\Psi_\eta(y)
        \right\|^2
    \right]
    \le C_{d,\eta}
    \left(
        \frac1{N_0}
        +
        \sum_{\ell=0}^{L-1}
        \frac{v_{\ell+1}}{N_{\ell+1}}
        +
        b_L
    \right)\,,
\]
where $C_{d,\eta}\propto C_{\rm fd} + \textcolor{black}{d\eta^{-2}}$ 
similarly as in the single-level case. Note that the numerical constant is independent of $L$. This proves the claim.
\end{proof}

\newpage





\bibliographystyle{plain}
\bibliography{MLMC.bib}

\end{document}